\documentclass{siamonline1116}

\usepackage{xspace}
\usepackage{amsmath}
\usepackage{mathrsfs}
\usepackage{amsfonts}
\usepackage{graphicx}
\usepackage{algorithmic}
\usepackage{enumerate}
\usepackage{multirow}
\usepackage{diagbox}
\usepackage{subfig}
\newsiamremark{remark}{Remark}
\usepackage{acro}
\usepackage{booktabs}
\usepackage{cleveref}

\newcommand{\eg}{\textsl{e.g.}\xspace}
\newcommand{\etc}{\textsl{etc.}\xspace}
\newcommand{\ie}{\textsl{i.e.}\xspace}
\newcommand{\wrt}{w.r.t.\@\xspace}
\newcommand{\signal}{f}
\newcommand{\truesignal}{\signal_{\text{true}}}
\newcommand{\signalother}{h}
\newcommand{\intensityflow}{I} 
\newcommand{\template}{I}
\newcommand{\data}{g}
\newcommand{\noisedata}{\Delta\data}
\newcommand{\diffeo}{\phi}
\newcommand{\truediffeo}{\diffeo^{\ast}}
\DeclareMathOperator{\Id}{Id}
\newcommand{\diffeoother}{\psi}
\newcommand{\gelement}[2]{\phi^{#2}_{#1}}
\newcommand{\vfield}{\nu}
\newcommand{\vfieldother}{\eta}
\newcommand{\velocityfield}{\boldsymbol{\nu}}
\newcommand{\truevelocityfield}{\boldsymbol{\nu}^{\ast}}
\newcommand{\velocityfieldother}{\boldsymbol{\eta}}
\newcommand{\velocityfieldzero}{\boldsymbol{0}}
\newcommand{\RecSpace}{X}
\newcommand{\RecSpaceFlow}{\mathscr{X}}
\newcommand{\DataSpace}{Y}

\DeclareMathOperator{\DiffG}{Diff}
\newcommand{\LieGroup}{G}
\newcommand{\DiffeoGroup}{\LieGroup_{\LieAlgebra}}

\newcommand{\LieAlgebra}{V}
\newcommand{\FlowSpace}[1]{\mathscr{L}^{#1}([0,1], \LieAlgebra)}
\newcommand{\LpSpace}{\mathscr{L}}
\newcommand{\Smooth}{\mathscr{C}}
\newcommand{\Sobolev}{\mathscr{H}}
 
\newcommand{\SBV}{SBV} 
\newcommand{\VectorFields}{\Smooth}
\newcommand{\LinearSpace}{\mathscr{L}} 
\newcommand{\VectorSpace}{H}
\newcommand{\VectorSpaceOther}{U}
\newcommand{\Real}{\mathbb{R}}
\newcommand{\Matrix}{\mathbb{M}}
\newcommand{\domain}{\Omega}
\newcommand{\datadomain}{M} 
\newcommand{\ForwardOp}{\mathcal{T}}
\newcommand{\DataDisc}{\mathcal{D}}
\newcommand{\RegFunc}{\mathcal{R}}
\newcommand{\ShapeFunc}{\mathcal{S}}
\newcommand{\ObjectiveV}{\mathcal{E}}
\newcommand{\MatchingFunctionalV}{\mathcal{J}}
\newcommand{\MatchingFunctionalX}{\mathcal{L}}
\newcommand{\GDiff}{\partial}  
\newcommand{\Diff}{\mathcal{D}} 
\newcommand{\grad}{\nabla}

\DeclareMathOperator*{\Div}{div} 
\newcommand{\DeforOpG}{\mathcal{W}} 
\newcommand{\DeforOpX}{\mathcal{V}} 
\newcommand{\LieGroupMetric}{\text{d}_{\LieGroup}}

\newcommand{\Koperator}{\mathcal{K}}
\newcommand{\kernel}{\mathsf{K}}

\DeclareMathOperator{\Laplace}{\triangle}
\newcommand{\dint}{\,\text{d}}
\newcommand{\Cdot}{\,\cdot\,}
\newcommand{\stepsize}{\alpha}

\graphicspath{{figures/}{figures/J_V/}{figures/Shepp_Logan/}{figures/Test_sigma_lambda/}{figures/Test_topology/Test1/}{figures/Test_topology/Test2/}{figures/Triangles/Test1/}{figures/Triangles/Test2-inverse_Test1/}}
\ifpdf
  \DeclareGraphicsExtensions{.eps,.pdf,.png,.jpg}
\else
  \DeclareGraphicsExtensions{.eps}
\fi

\crefname{line}{step}{steps}
\Crefname{line}{Step}{Steps}

\numberwithin{theorem}{section}

\DeclareAcronym{LDDMM}{
  short = LDDMM,
  long = large deformation diffeomorphic metric mapping}
\DeclareAcronym{RKHS}{
  short = RKHS,
  long = reproducing kernel {H}ilbert space}
\DeclareAcronym{ML}{
  short = ML,
  long = maximum likelihood}
\DeclareAcronym{ET}{
  short = ET,
  long = electron tomography}
\DeclareAcronym{TV}{
  short = TV,
  long = total variation}
\DeclareAcronym{FBP}{
  short = FBP,
  long = filtered back projection}
\DeclareAcronym{SNR}{
  short = SNR,
  long = signal-to-noise ratio}
\DeclareAcronym{ICES}{
  short = ICES,
  long = Institute for Computational Engineering and Sciences}
\DeclareAcronym{CAS}{
  short = CAS,
  long = Chinese Academy of Sciences}
\DeclareAcronym{KTH}{
  short = KTH,
  long =   KTH -- Royal Institute of Technology}
\DeclareAcronym{LSEC}{
  short = LSEC,
  long = State Key Laboratory of Scientific and Engineering Computing}
\DeclareAcronym{ODL}{
  short = ODL,
  long = Operator Discretization Library}
\DeclareAcronym{PDE}{
  short = PDE,
  long = partial differential equation}     
\DeclareAcronym{ODE}{
  short = ODE,
  long = ordinary differential equation}    
\DeclareAcronym{FFT}{
  short = FFT,
  long = Fast Fourier transform} 
\DeclareAcronym{SSIM}{
  short = SSIM,
  long = structural similarity}
\DeclareAcronym{PSNR}{
  short = PSNR,
  long = peak signal-to-noise ratio}

\newcommand{\TheTitle}{Indirect Image Registration with Large Diffeomorphic Deformations} 
\newcommand{\TheAuthors}{C. Chen and O. \"Oktem}

\headers{\TheTitle}{\TheAuthors}

\title{{\TheTitle}\thanks{Version 1.9 revised 2017-10-09.
\funding{The work by Chong Chen and Ozan \"Oktem has been supported by the Swedish Foundation for Strategic Research grant AM13-0049. Chen was also supported in part by the National Natural Science Foundation of China (NSFC) under the grant 11301520.}}}

\author{
  Chong Chen\thanks{Department of Mathematics, KTH--Royal Institute of Technology, 100 44 Stockholm, Sweden; LSEC, ICMSEC, Academy of Mathematics and Systems Science, Chinese Academy of Sciences, Beijing 100190, People's Republic of China (\email{chench@lsec.cc.ac.cn}).}
  \and
  Ozan \"Oktem\thanks{Department of Mathematics, KTH--Royal Institute of Technology, 100 44 Stockholm, Sweden 	 
    (\email{ozan@kth.se}).}
}

\ifpdf
\hypersetup{
  pdftitle={\TheTitle},
  pdfauthor={\TheAuthors}
}
\fi

\begin{document}

\maketitle

% REQUIRED
\begin{abstract}
The paper adapts the \acl{LDDMM} framework for image registration to the indirect setting where a template is registered against a target that is given through indirect noisy observations. The registration uses diffeomorphisms that transform the template through a (group) action. These diffeomorphisms are generated by solving a flow equation that is defined by a velocity field with certain regularity. The theoretical analysis includes a proof that indirect image registration has solutions (existence) that are stable and that converge as the data error tends so zero, so it becomes a well-defined regularization method. The paper concludes with examples of indirect image registration in 2D tomography with very sparse and/or highly noisy data.
\end{abstract}

% REQUIRED

% REQUIRED

\section{Introduction}
Image registration (matching) is the task of transforming a template image so that it matches a target image. This arises in many fields, such as quality control in industrial manufacturing \cite{DeStGa13}, various applications in remote sensing \cite{DaSaSh10}, face recognition \cite{ZiFl03,SaGuCa15}, robotic navigation \cite{BoOrOl08}, and medical imaging \cite{SoDaPa13,ScHePaBr16,ViMaKlMuSt16}, etc. The variant that is considered here is indirect image registration, \ie, when the template image is registered against a target that is known only through indirect noisy observations, such as in tomographic imaging. It makes use of the \acf{LDDMM} framework for diffeomorphic image registration and thereby extends the indirect image registration scheme in \cite{OkChDoRaBa16}, which uses (small) linearized deformations. 

An important development in regularization theory is ongoing where reconstruction and feature extraction steps are pursued simultaneously, such as joint segmentation and reconstruction. Indirect image registration can be seen as part of this development where the feature in question is the ``shape'' of the structures in the image and their temporal variability. This is highly relevant for spatiotemporal imaging, see \cite{OkChDoRaBa16} for a survey and other use cases of indirect image registration.

There is an extensive literature on image registration where a template is registered by means of a diffeomorphism against a target so that their  ``shapes''  agree, see \cite{SoDaPa13} for a nice survey. Image registration is then recast as the problem of finding a suitable diffeomorphism that deforms the template into the target image \cite{ChRaMi96}. The underlying assumption is that the target image is contained in the orbit of the template under the group action of diffeomorphisms. This principle can be stated in a very general setting where diffeomorphism acts on various image features, like  landmark points, curves, surfaces, scalar images, or even vector/tensor valued images \cite{Yo10}. It also adapts readily to the indirect image registration setting and this is worked out in detail for the case when one seeks to register scalar images.

\section{Overview of paper and specific contributions}
The main contribution is to develop a variational framework for indirect image registration (\cref{sec:IndirectImageReg}) where a template is registered against a target that is known only through indirect noisy observations (data). This is done by appropriately adapting the \ac{LDDMM} framework where the template is deformed by diffeomorphisms that act on images through a group action as explained in \cref{sec:ShapeTheory}. This section also contains an overview of the \ac{LDDMM} theory.

An important theoretical topic is to investigate to what extent this indirect registration scheme is a regularization. Existence of solutions along with stability and convergence results are stated in \cref{sec:RegularizingProperty}. In conclusion, the \ac{LDDMM} based indirect registration scheme is formally a well-defined regularization method. On the other hand, as most of the other variational schemes for image registration, one cannot expect to have uniqueness due to lack of convexity. This theoretical investigation is followed by explicit calculations of the derivative and gradient of the objective functional in the variational formulation of the indirect registration problem (\cref{sec:Gradient}). Numerical implementation, which relies on theory of \acl{RKHS} (\cref{sec:RKHS}) is outlined in \cref{sec:NumericalImpl}.

\Cref{sec:2DCT} contains some numerical experiments from tomography that show performance of indirect image registration. The results support the claim that shape information contained in the template has a strong regularizing effect, and especially so for noisy highly under-sampled inverse problems. Furthermore, the experiments also suggest that the prior shape information does not have to be that accurate, which is important for the used cases mentioned in \cite{OkChDoRaBa16}. None of these claims are however proved formally by mathematical theorems. 

Finally, \cref{sec:Discussion} discusses possible extensions of indirect image registration to the case when the template also needs to be recovered. It also outlines how indirect image registration relates to joint reconstruction and motion estimation in spatiotemporal imaging.

\section{Inverse problems, ill-posedness and variational regularization}\label{sec:InverseProblems}
The purpose of this section is to set some notations and concepts used throughout the paper.

\paragraph{Image reconstruction}
The goal in image reconstruction is to estimate some spatially distributed quantity (\emph{image}) from indirect noisy observations (\emph{measured data}). Stated mathematically, the aim is to reconstruct an \emph{image} $\truesignal \in \RecSpace$ from \emph{data} $\data \in\DataSpace$ where
\begin{equation}\label{eq:InvProb}  
     \data = \ForwardOp(\truesignal) + \noisedata. 
\end{equation}
Here, $\RecSpace$ (\emph{reconstruction space}) is the vector space of all possible images, so it is a suitable Hilbert space of functions defined on a fixed domain $\domain \subset \Real^n$. Next, $\DataSpace$ (\emph{data space}) is the vector space of all possible data, which for digitized data is a subset of $\Real^m$. Furthermore, $\ForwardOp \colon \RecSpace \to \DataSpace$ (\emph{forward operator}) models how a given image gives rise to data in absence of noise and measurement errors, and $\noisedata \in \DataSpace$ (\emph{data noise component}) is a sample of a $\DataSpace$--valued random element whose probability distribution is (\emph{data noise model}) assumed to be explicitly expressible in terms of $\ForwardOp(\truesignal) $.

\paragraph{Ill-posedness}
A naive approach for reconstructing the true (unknown) image $\truesignal$ is to solve the equation $\ForwardOp(\signal)= \data$. Often there are no solutions to this equation since measured data is not in the range of $\ForwardOp$ for a variety of reasons (noise, modeling errors, \etc).  This is commonly addressed by relaxing the notion of a solution by considering 
\begin{equation}\label{eq:LeastSquares}
  \min_{\signal \in\RecSpace} \DataDisc\bigl( \ForwardOp(\signal), \data \bigr).
\end{equation}  
The mapping $\DataDisc \colon \DataSpace \times \DataSpace \to \Real$ (\emph{data discrepancy functional}) quantifies the data misfit and a natural candidate is to choose it as a suitable affine transformation of the negative log likelihood of data. In such case, solving \cref{eq:LeastSquares} amounts to finding \ac{ML} solutions, which works well when \cref{eq:LeastSquares} has a unique solution (\emph{uniqueness}) that depends continuously on data (\emph{stability}). This is however not the case when  \cref{eq:InvProb} is ill-posed, in which case one needs to use regularization that introduces stability, and preferably uniqueness, by making use of prior knowledge about $\truesignal$.

\paragraph{Variational regularization}
The idea here is to add a penalization term to the objective functional in \cref{eq:LeastSquares} resulting in a variational problem of the type
\begin{equation}\label{eq:VarReg}
  \min_{\signal \in\RecSpace} \Bigl[ \mu \RegFunc(\signal) + \DataDisc\bigl( \ForwardOp(\signal), \data \bigr) \Bigr] 
  \quad\text{for some given $\mu\geq 0$.}
\end{equation}  
Such regularization methods have gained much attention lately, and especially so in imaging \cite{ScGrGrHaLe09}. The functional $\RegFunc \colon \RecSpace \to \Real$ introduces stability, and preferably also uniqueness, often by encoding some priori known regularity property of $\truesignal$, \eg, assuming $\RecSpace \subset \LpSpace^2(\domain,\Real)$ and taking the $\LpSpace^2$-norm of the gradient magnitude (Dirichlet energy) is known to produce smooth solutions whereas taking the $\LpSpace^1$-norm of the gradient magnitude (\acl{TV}) yields solutions that preserve edges while smooth variations may be suppressed \cite{BuOs13}.

\section{Reproducing kernel Hilbert spaces}\label{sec:RKHS}
The diffeomorphic deformations constructed in the \ac{LDDMM} framework (\cref{sec:LDDMMFlows}) will make use of velocity fields that at each time point are contained in a \acf{RKHS}. The gradient computations in \cref{sec:Gradient} also rely on this assumption. The short introduction to the theory of \ac{RKHS} provided here gives the necessary background needed in subsequent sections.

The theory of \aclp{RKHS} was initialized in the 1940s for spaces of scalar-valued functions \cite{Ar50}, which was later extended to spaces of functions with values in locally convex topological spaces \cite{Sc64}. It has lately gained significant interest due to applications in machine learning \cite{ScSm01,BeTh04, Mu12}. The starting point is an abstract Hilbert space $\VectorSpace$ whose elements are functions defined on a fixed domain $\domain \subset \Real^n$ that take values in a real Hilbert space $\VectorSpaceOther$. Such a space is a \ac{RKHS} if evaluation functionals $\delta^a_x \colon  \VectorSpace \to \Real$ defined as $\delta^a_x(\vfield):=\bigl\langle \vfield(x),a \bigr\rangle_{\VectorSpaceOther}$  are bounded for every $x \in \domain$ and $a \in  \VectorSpaceOther$. One way to construct a \ac{RKHS} is to specify a \emph{reproducing kernel function}. It maps a pair of points in $\domain$ into the Banach space
$\LinearSpace(\VectorSpaceOther,\VectorSpaceOther)$ of bounded linear operators on $\VectorSpaceOther$, \ie, an operator $\Koperator \colon \domain \times \domain \to \LinearSpace(\VectorSpaceOther,\VectorSpaceOther)$ such that
\begin{enumerate}[(i)]
\item $\Koperator(\Cdot,x)(a) \in \VectorSpace$ for all $x \in \domain$ and $a \in \VectorSpaceOther$.
\item The \emph{reproducing property} holds for $\Koperator$, \ie, if $x \in \domain$ then 
\begin{equation}\label{eq:IpEquality}
  \bigl\langle \vfield(x), a \bigr\rangle_{\VectorSpaceOther} 
  = \bigl\langle \vfield, \Koperator(\Cdot ,x)(a) \bigr\rangle_{\VectorSpace}
       \quad\text{for any $\vfield \in \VectorSpace$ and $a \in \VectorSpaceOther$.} 
\end{equation}    
\end{enumerate}

An important characterization is that a Hilbert space $\VectorSpace$ of $\VectorSpaceOther$-valued functions that is continuously embedded in $\Smooth(\domain,\VectorSpaceOther)$ is a \ac{RKHS} if and only if it has a continuous reproducing kernel $\Koperator \colon \domain \times \domain \to \LinearSpace(\VectorSpaceOther,\VectorSpaceOther)$. 

\paragraph{Square integrable vector fields}
%Vector fields on $n$-dimensional Euclidian space are $\Real^n$-valued functions defined on $\Real^n$. 
If $\VectorSpace$ is a Hilbert space of vector fields on $\Real^n$ that is admissible (see \cref{sec:LDDMMFlows}), then it is in particular continuously embedded in $\Smooth(\domain,\Real^n)$. Hence, it is a \ac{RKHS} if and only if it has a continuous positive definite reproducing kernel $\Koperator \colon \domain\times \domain \to \LinearSpace(\Real^{n}, \Real^{n})$, which in turn can be represented by a continuous positive definite function $\kernel \colon \domain \times \domain \to \Matrix_{+}^{n \times n}$. Here, $\Matrix^{n,m}$ denotes the vector space of all $(n \times m)$ matrices and $\Matrix^{n,m}_+$ denotes those matrices that are positive definite. Assuming in addition that  $\VectorSpace$ is a \ac{RKHS} that is continuously embedded in $L^2(\domain,\Real^n)$, then there exists an continuous imbedding $\iota \colon \VectorSpace \to \LpSpace^2(\domain, \Real^n)$. Hence, by Riesz theorem there exist $\iota^* \colon \LpSpace^2(\domain, \Real^n) \to \VectorSpace$ such that 
\begin{equation}\label{eq:Inclusion}
  \bigl\langle \iota(\vfield), \vfieldother \bigr\rangle_{\LpSpace^2(\domain, \Real^n)} = 
  \bigl\langle \vfield, \iota^*(\vfieldother) \bigr\rangle_{\VectorSpace}
  \quad\text{for all $ \vfield, \vfieldother \in \VectorSpace$}.
\end{equation}  
A natural task is to further examine the relation between the $\LpSpace^2$-inner product and the \ac{RKHS} inner product on $\VectorSpace$. As we next show, one can prove that 
\begin{equation}\label{eq:VIP}
  \langle \vfieldother, \vfield \rangle_{\LpSpace^2} 
   = \Bigl\langle \vfield, \int_{\domain} \Koperator(\Cdot, x)\bigl( \vfieldother(x) \bigr)\dint x \Bigr\rangle_{\VectorSpace}.
\end{equation}

To prove \cref{eq:VIP}, observe first that by \cref{eq:Inclusion} 
\[  \bigl\langle \iota(\vfield), \iota(\vfieldother) \bigr\rangle_{\LpSpace^2(\domain, \Real^n)} = 
  \Bigl\langle \vfield, \iota\bigl(\iota^*(\vfieldother)\bigr) \Bigr\rangle_{\VectorSpace}
  \quad\text{for any $\vfield, \vfieldother \in \VectorSpace$.}
\]
Hence, \cref{eq:VIP} follows directly from 
\begin{equation}\label{eq:InclusionAdjoint}
  \iota\bigl(\iota^*(\vfieldother)\bigr) = \int_{\domain} \Koperator(\Cdot, x)\bigl( \vfieldother(x) \bigr)\dint x.
\end{equation}
To prove \cref{eq:InclusionAdjoint}, observe first that it is sufficient to prove it point wise, \ie, to show that 
\begin{equation}\label{eq:InclusionAdjointPointwise}
  \Bigl\langle \iota\bigl(\iota^*(\vfieldother)\bigr)(y), a \Bigr\rangle_{\Real^n} 
  = \int_{\domain} \Bigl\langle \Koperator(y, x)\bigl( \vfieldother(x) \bigr), a \Bigr\rangle_{\Real^n} \dint x
  \quad\text{holds for any $y \in \domain$ and $a \in \Real^n$.}
\end{equation}
The equality in \cref{eq:InclusionAdjointPointwise} follows from the next calculation:
\begin{align*}
  \Bigl\langle \iota\bigl(\iota^*(\vfieldother)\bigr)(y), a \Bigr\rangle_{\Real^n}
  &= \bigl[ \text{By \cref{eq:IpEquality}} \bigr]
    =  \Bigl\langle \iota^*\bigl(\iota(\vfieldother)\bigr), \Koperator(\Cdot ,y)(a) \Bigr\rangle_{\VectorSpace} 
  \! = \! \bigl[ \text{By \cref{eq:Inclusion}} \bigr] \!\!
    =  \Bigl\langle \iota(\vfieldother), \iota\bigl( \Koperator(\Cdot ,y)(a) \bigr) \Bigr\rangle_{\LpSpace^2} \\
  &= \int_{\domain} \Bigl\langle \iota(\vfieldother)(x), \iota\bigl( \Koperator(\Cdot,y)(a) \bigr)(x) \Bigr\rangle_{\Real^n}\dint x
  = \int_{\domain} \Bigl\langle \vfieldother(x), \Koperator(x,y)(a) \Bigr\rangle_{\Real^n}\dint x.
\end{align*}

\acreset{LDDMM}
\section{Shape theory}\label{sec:ShapeTheory}
The overall aim is to develop a quantitative framework for studying shapes and their variability. The approach considered here is based on deformable templates and it can be traced back to work by D'Arcy Thompson in beginning of 1900's \cite{Ar45}. Shape theory based on deformable templates is now an active field of research \cite{GrMi07,Yo10,TrYo11,Mi1TrYo5}.

The starting point is to specify a \emph{shape space}, which here refers to a set whose elements are image features with shapes that are to be analysed. There is no formal definition of what constitutes an image feature, but intuitively it is a representation of the image that retains information relevant for its interpretation. Hence, the choice of shape space is highly task dependent and examples are landmark points, curves, surfaces, and scalar/vector/tensor valued images. 
This is followed by specifying a \emph{set of deformations} whose elements map the shape space into itself and thereby model shape changes. Clearly, the identity mapping preserves all aspects of shape. And in this context, it is therefore the ``smallest'' deformation. A key step will be to introduce a metric on the set of deformations that induces a shape similarity measure in shape space.

\subsection{Shape space}\label{sec:ShapeSpace}
There are many image features important for interpretation \cite{AwHa16}, see also \cite{Li13,Li15} for an axiomatic characterization of features natural for visual perception. For (indirect) image registration, it is important that such features are deformable. Namely, it should be feasible to act on it by means of a deformation. 

This paper considers grey-scale images, more precisely $\RecSpace = \SBV(\domain,\Real) \bigcap \LpSpace^{\infty}(\domain,\Real)$. The requirement that elements in $\RecSpace$ are in $\LpSpace^{\infty}(\domain,\Real)$ (essentially bounded) is reasonable for images since grey-scale values are bounded. Furthermore, $\SBV(\domain,\Real)$ denotes the set of real-valued functions of special bounded variation over $\domain \subset \Real^n$. This space was first introduced in \cite{AmDe88} (see also \cite{DeAm89}) and its formal definition is somewhat involved, see \cite[Definition~22 on p.~141]{Vi09}. Intuitively, these are functions of bounded variation whose singular part of the distributional derivative is supported by an $(n-1)$-dimensional rectifiable set in $\domain$. In particular, such  functions are not necessarily smooth. The latter is important since images may contain edges. Occasionally $\RecSpace$ will be equipped with an $\LpSpace^2$--inner product.

\subsection{Set of deformations}
Deformations are operations that transform elements in a shape space. In our setting, deformations is represented by mappings from the image domain $\domain \subset \Real^n$ into $\Real^n$ along with an action describing how they deform deformable objects. The set of deformations should be rich enough to capture the shape variability arising in the application, yet restricted enough to allow for a coherent mathematical theory. 

Considering only rigid body motions is often too limited since many applications involve non-rigid shape deformations. Furthermore, composing two deformations should yield a valid deformation, \ie, the set $\LieGroup$ of deformations closed under composition. Here, the identity mapping becomes the natural ``zero'' deformation. Furthermore, it should also be possible to reverse a deformation, \ie, $\LieGroup$ closed under inversion. Taken together, this implies that $\LieGroup$ \emph{forms a group under the group law given by composition of functions}. The group structure also implies that a deformation transforms an element in the shape space by means of a group action. Finally, often it also makes sense to assume that deformable objects do not tear, fold or collapse the image, \ie, the deformation preserving the topology. 

$\Smooth^p$-diffeomorphisms (see below) form a group of non-rigid transformations satisfying the above requirements. These mappings can also act on real valued functions as in \cref{sec:GroupAction}.
\[ 
  \DiffG^p(\Real^n) := 
    \Bigl\{ 
      \diffeo \in \Smooth^p(\Real^n,\Real^n) : 
      \text{$\diffeo$ is bijective with $\diffeo^{-1} \in \Smooth^p(\Real^n,\Real^n)$}
    \Bigr\}.
\]
However, formulating and proving mathematical results relevant for (indirect) image registration, and the need to perform computations with such objects, require one to further restrict the set of diffeomorphisms, \ie,  considering suitable sub-groups of $\DiffG^p(\Real^n)$.

\subsection{Large diffeomorphic deformations}\label{sec:LargeDeformations}
A natural approach is to consider deformations given by additively perturbing the identity map with a vector field that is sufficiently regular \cite{OkChDoRaBa16}. The resulting set of deformations is closed under composition, \ie, it forms a semi-group. Furthermore, such a deformation can be seen as linear approximation of a diffeomorphism, but it is not necessarily invertible unless the aforementioned vector field is sufficiently small and regular \cite[Proposition 8.6]{Yo10}. Hence, this is not a suitable framework for diffeomorphic image registration where the template and target have large differences. 

One approach to address the above issue of invertibility was presented in \cite{ChRaMi96}. The idea is to consider transformations given as a composition of infinitesimally small linearised deformations. More precisely, diffeomorphic large deformations are obtained  by integrating the identity map along a velocity field (curve of vector fields). If properly designed, the velocity field induces an isotopy (a curve of diffeomorphisms). Thereby, the set of velocity fields parametrises the set of deformations and a natural requirement is that velocity fields under consideration have integrable trajectories. This provides a framework for generating a rich sets of deformations, which then can be adapted to specific image registration problems by appropriately choosing the velocity fields. 

\emph{The indirect image registration method in this paper is based on \ac{LDDMM},} see \cref{rem:AltLDDMM} for a brief survey of alternative frameworks for large diffeomorphic deformations. In \ac{LDDMM}, the velocity field is square integrable and its corresponding vector field at each time is contained in a fixed Hilbert space of vector fields that is an admissible \ac{RKHS}. The deformations are generated by integrating the identity map along trajectories of such velocity fields and these form a sub-group of $\DiffG^1(\Real^n)$, see \cref{sec:LDDMMFlows} for more details. As shown there, this group can furthermore be equipped with a metric, which by a group action (\cref{sec:GroupAction}) induces a metric on the shape space that can be used to quantify shape similarity (\cref{sec:IndirectImageReg}).

\subsubsection{The \ac{LDDMM} framework}\label{sec:LDDMMFlows}
The starting point in the \ac{LDDMM} framework for large diffeomorphic deformations is to define the vector space of velocity fields that generate the diffeomorphisms. An important point is to show that these diffeomorphisms for a group. Another is to define a right invariant metric so that the group is a complete metric space \wrt this metric. The metric is later used to induce a metric on the shape space, which in turn is used as a regularizer in indirect image registration (\cref{sec:IndirectImageReg}). A key part of showing that indirect image registration is a well-defined regularization method (\cref{sec:RegularizingProperty}) is to relate convergence of velocity fields to convergence of corresponding diffeomorphisms. 

\paragraph{Admissible vector fields and the flow equation}
The idea in \ac{LDDMM} is to generate large diffeomorphic deformations by solving a flow equation. This is only possible when the velocity field enjoys certain regularity that ensures the flow equation has a unique solution.
More precisely, fix a domain $\domain \subset \Real^n$ and let $\LieAlgebra\subset \VectorFields_0^p(\domain,\Real^n)$ denote a $p$--\emph{admissible} vector space of vector fields, \ie, $\LieAlgebra$ can be continuously embedded into $\VectorFields_0^p(\domain,\Real^n)$ under the topology induced by the $\Vert \Cdot \Vert_{p,\infty}$--norm. Furthermore, let $\FlowSpace{q}$ denote the vector space of velocity fields that are $\LpSpace^q$-integrable in time and belong to $\LieAlgebra$ at any time point, \ie, 
\begin{equation}\label{eq:FlowSpace}
  \FlowSpace{q} := 
     \Bigl\{ 
       \velocityfield \colon [0,1] \times \domain \to \Real^n : \velocityfield(t,\Cdot) \in \LieAlgebra
       \text{ and $\Vert \velocityfield \Vert_{\FlowSpace{q}} < \infty$}
     \Bigr\}
\end{equation} 
with 
\[
\Vert \velocityfield \Vert_{\FlowSpace{q}} :=  
      \biggl( \int_0^1 \bigl\Vert \velocityfield(t,\Cdot) \bigr\Vert^q_{\LieAlgebra}\dint t \biggr)^{1/q}
    \quad\text{for an integer $q\geq 1$.}
\]
This is a normed vector space and if $\LieAlgebra$ is a Hilbert space, then $\FlowSpace{2}$ can also be equipped with a Hilbert space structure. Focus will henceforth be on the Hilbert space $\FlowSpace{2}$, but some statements will involve $\FlowSpace{1}$. The following theorem forms the basis of the \ac{LDDMM} approach, see \cite[Appendix~C]{Yo10} for its proof.
\begin{theorem}[Existence of flows]\label{thm:FlowExistence}
Given $\velocityfield \in \FlowSpace{1}$, there exists a unique continuous curve $[0,1] \ni t \mapsto \gelement{0,t}{\velocityfield} \in \DiffG^p(\Real^n)$ that satisfies the following flow equation:
\begin{equation}\label{eq:FlowEq}
 \begin{cases} 
    \dfrac{d}{d t}\gelement{0,t}{\velocityfield}(x) = \velocityfield\bigl( t, \gelement{0,t}{\velocityfield}(x) \bigr) & \\[0.75em]
    \gelement{0,0}{\velocityfield}(x)=x. &  
   \end{cases} 
   \quad\text{for all $x \in \domain$.}
\end{equation}
%Furthermore, for each $t \in \gelement{0,t}{\velocityfield} \in\DiffG^p(\Real^n)$.
\end{theorem}
By \cref{thm:FlowExistence}, $\gelement{s,t}{\velocityfield} := \gelement{0,t}{\velocityfield} \circ (\gelement{0,s}{\velocityfield})^{-1}$ becomes a well defined element in $\DiffG^p(\Real^n)$ for any $s, t \in [0,1]$, and by \cite[Proposition~C.6]{Yo10} one gets 
\begin{equation}\label{eq:phiRelations}
    \gelement{s,t}{\velocityfield} = \gelement{\tau,t}{\velocityfield} \circ \gelement{s,\tau}{\velocityfield} 
    \quad\text{and}\quad
    (\gelement{s,t}{\velocityfield})^{-1} = \gelement{t,s}{\velocityfield} 
  \qquad\text{for any $0 \leq s,\tau,t \leq 1$.}
\end{equation}

\paragraph{Diffeomorphisms generated by flows}
Assume $\LieAlgebra\subset \VectorFields_0^p(\domain,\Real^n)$ is a $p$-admissible Banach space and define \begin{equation}\label{eq:DeforSet}
 \DiffeoGroup := \Bigl\{ \gelement{0,1}{\velocityfield} \colon \domain \to \Real^n : 
      \text{$\gelement{0,1}{\velocityfield}$ solves \cref{eq:FlowEq} with 
      $\velocityfield \in \FlowSpace{1}$} 
   \Bigr\}.
\end{equation}
As already concluded, $\DiffeoGroup \subset \DiffG^p(\domain)$ is a sub-group. Next, the norm on $\FlowSpace{1}$ can be used to define a (right invariant) metric $\LieGroupMetric \colon \DiffeoGroup \times \DiffeoGroup \to \Real_+$ on $\DiffeoGroup$ as 
\begin{equation}\label{eq:DiffMetric1}
  \LieGroupMetric(\diffeo,\diffeoother) := \!\!
  \inf_{\substack{\velocityfield \in \FlowSpace{1} \\ \diffeoother = \diffeo \circ \gelement{0,1}{\velocityfield}}}\, 
     \Vert \velocityfield \Vert_{\FlowSpace{1}}
  \quad\text{for $\diffeo,\diffeoother \in \DiffeoGroup$.}
\end{equation}
%In the following, we will use the notation 
%\begin{equation}\label{eq:DiffNorm}
%  \LieGroupNorm{\diffeo}{\LieAlgebra} := \LieGroupMetric(\Id,\diffeo)
%  \quad\text{for $\diffeo \in \DiffeoGroup$.}
%\end{equation}
With this construction, $(\DiffeoGroup,\LieGroupMetric)$ is a complete metric space whenever $\LieAlgebra$ is $p$-admissible \cite[Theorem~8.15]{Yo10}. The right invariance of the metric is important since
\[ 
    \LieGroupMetric(\diffeo,\diffeoother) = \LieGroupMetric(\diffeo \circ \varphi,\diffeoother \circ \varphi)
    \quad\text{for any $\varphi \in \DiffeoGroup$.}
\]
Furthermore, with this metric one can prove existence of a minimizing velcocity field between two diffeomorphisms in $\DiffeoGroup$. More precisely, if $\LieAlgebra$ is $p$-admissible, then by \cite[Theorems~8.18 and 8.20]{Yo10}, for any $\diffeo,\diffeoother \in \DiffeoGroup$ there exists a velocity field $\velocityfield \in \FlowSpace{2}$ such that 
  \begin{equation}\label{eq:DiffMetric2}
    \LieGroupMetric(\diffeo,\diffeoother) = \Vert \velocityfield\Vert_{\FlowSpace{2}}
      \quad\text{and}\quad \diffeoother = \diffeo \circ \gelement{0,1}{\velocityfield}.
\end{equation}
Hence, when computing the metric in \cref{eq:DiffMetric1}, the infimum over $\FlowSpace{1}$ can be replaced by a minimum over the Hilbert space $\FlowSpace{2}$.

Developing a theory for image registration based on arbitrary admissible spaces is difficult, mainly due to issues related to assigning topological and smooth structures on $\DiffeoGroup$. See \cref{sec:TopoGDiff} for a discussion. This paper considers admissible Hilbert space of vector fields $\LieAlgebra$ constructed from \ac{RKHS} theory (\cref{sec:RKHS}) that are continuously embedded in $\LpSpace^2(\domain,\Real^n)$. The specific application to tomography in \cref{sec:2DCT} makes use of a diagonal Gaussian kernel.
\begin{remark}
It is possible to use other \ac{RKHS} kernels than diagonal ones as long as the kernel is represented by a function continuously differentiable in both variables up to order $p$, in which case $\LieAlgebra \subset \VectorFields^p(\Real^n, \Real^n)$ \cite[Theorem~9]{Vi09}. See \cite{MiGl14} for admissible vector fields constructed using other non-scalar matrix-valued kernels. An alternative way to construct admissible vector spaces is by introducing a norm defined through a differential (or pseudo-differential) operator $L$:
\[ \Vert \vfield \Vert_{\LieAlgebra}^2 
      := \int_{\Real^n} \bigl\langle L(\vfield)(x), \vfield(x) \bigr\rangle_{\Real^n}\,\dint x. 
\]
As an example, choosing $L := (\Id-\Laplace)^s$ with $\Laplace$ as the Laplacian operator yields the Sobolev space $\LieAlgebra=\Sobolev^s(\Real^n,\Real^n)$. See \cite[Section~9.2]{Yo10} for more on building admissible vector spaces from operators, and in particular \cite[Theorem 9.12]{Yo10} for the relation to \ac{RKHS} theory.
\end{remark}

\paragraph{Strong and weak convergence results}
The aim here is to relate convergence of vector fields to convergence of corresponding diffeomorphisms. The next result shows that weak convergence of velocity fields implies uniform convergence of corresponding diffeomorphism on every compact set. This is an important part for proving the results in \cref{sec:RegularizingProperty}. The precise statement is given below, see \cite[Theorem~3 on p.~12]{Vi09} for the proof.
\begin{theorem}[Uniform convergence of compact sets]\label{thm:Thm3Vi09}
Let $\LieAlgebra \subset \VectorFields^p(\Real^n, \Real^n)$ be $p$--admissible with $p \geq 1$.  If $\velocityfield_k \to \velocityfield$ in the $\FlowSpace{1}$-topology (strong convergence of velocity fields), then $\Diff^j(\gelement{0,t}{\velocityfield_k}) \to \Diff^j(\gelement{0,t}{\velocityfield})$ uniformly on any compact subset of $\domain$ for any $t \in [0,1]$ and $j \leq p$. Likewise, if $\velocityfield_k \rightharpoonup \velocityfield$ in the $\FlowSpace{2}$-topology (weak convergence of velocity fields), then $\Diff^j(\gelement{0,t}{\velocityfield_k}) \to \Diff^j(\gelement{0,t}{\velocityfield})$ uniformly on any compact subset of $\domain$ for any $t \in [0,1]$ and $j \leq p-1$. 
\end{theorem}
The case $p=1$ is of specific interest:
\begin{enumerate}
\item $\velocityfield_k \to \velocityfield$ in $\FlowSpace{1}$, then $\Diff(\gelement{0,t}{\velocityfield_k}) \to \Diff(\gelement{0,t}{\velocityfield})$ uniformly on every compact set. 
\item $\velocityfield_k \rightharpoonup \velocityfield$ in $\FlowSpace{2}$, then $\gelement{0,t}{\velocityfield_k} \to \gelement{0,t}{\velocityfield}$ uniformly on every compact set.
\end{enumerate}
\begin{remark}\label{rem:ExpMap}
The \emph{flow map} is defined as the mapping $\FlowSpace{1} \ni \velocityfield \mapsto  \gelement{0,1}{\velocityfield} \in \DiffeoGroup$ and by \cref{thm:Thm3Vi09}, it is well-defined and continuous.
This mapping is also called the ``exponential map'' (not to be confused with the exponential map of a Riemannian metric). This exponential map is the analogue of the exponential map of a finite dimensional Lie group. One can  argue that the group $\DiffG^{\infty}(\domain)$ is an infinite dimensional Lie group with Lie algebra the Lie algebra $\LieAlgebra$ provided that $\domain$ is compact. However, the nice relations between a finite dimensional Lie group and its Lie algebra, cease to exist in the infinite dimensional setting. For instance the exponential map fails to be one-to-one or surjective near the identity, see \cite[Section~1.3.6]{Ba97} and  \cref{sec:TopoGDiff} for further discussion on these matters.  
\end{remark}

\subsection{Group actions}\label{sec:GroupAction}
Elements in  the group $\DiffeoGroup$ of diffeomorphisms in \cref{eq:DeforSet} can act on images $\RecSpace \subset \LpSpace^2(\Real^n, \Real)$, which are real-valued functions defined on some fixed bounded domain $\domain \subset\Real^n$. The group action defines the operator
\begin{equation}\label{eq:DeforOp}
  \DeforOpG \colon \DiffeoGroup \times \RecSpace \to \RecSpace.  
\end{equation}
In imaging, two group actions are natural:
\begin{description}
\item[Geometric deformation:]
  Deforms images without changing grey-scale values. This choice is suitable for inverse problems 
  where shape, not texture, is the main image feature.
  \begin{equation}\label{eq:Geometric}
   \DeforOpG(\diffeo, \template)  := \template \circ \diffeo^{-1} 
     \quad\text{for $\template \in \RecSpace$ and $\diffeo \in \DiffeoGroup$.}  
   \end{equation}
\item[Mass-preserving deformation:]
  Deformation changes intensities but ensures mass is preserved. This choice is suitable for inverse problems 
  where intensities are allowed to change while preserving mass.
  \begin{equation}\label{eq:MassPreserving}
    \DeforOpG(\diffeo, \template)  := 
        \bigl\vert \Diff(\diffeo^{-1}) \bigr\vert \, (\template \circ \diffeo^{-1}) 
     \quad\text{for $\template \in \RecSpace$ and $\diffeo \in \DiffeoGroup$.}    
   \end{equation}
   In the above, $\bigl\vert \Diff(\diffeo^{-1}) \bigr\vert$ denotes the Jacobian determinant of $\diffeo^{-1}$. 
\end{description}

\section{Indirect image registration}\label{sec:IndirectImageReg}
Consider the inverse problem in \cref{eq:InvProb} where the shape space $\RecSpace$ has elements representing grey-scale images defined over some fixed image domain $\domain \subset \Real^n$. Next, assume a priori that the true (unknown) target image $\truesignal \in \RecSpace$ in \cref{eq:InvProb} can be written as an admissible deformation of a given \emph{shape template} $\template \colon \domain \to \Real$, \ie, 
\begin{equation}\label{eq:ExactShapeAss}
  \truesignal = \DeforOpG(\truediffeo,\template)
     \quad\text{on $\domain$ for some $\truediffeo \in \DiffeoGroup$.}
\end{equation}   
The metric \cref{eq:DiffMetric1} can, through the action of the group $\DiffeoGroup$ on the shape space $\RecSpace$, be used to define a shape similarity measure between objects in $\RecSpace$. To define the latter, fix an element in shape space, the \emph{template} $\template \in \RecSpace$, and introduce the \emph{shape functional} $\ShapeFunc(\Cdot, \template) \colon \RecSpace \to \Real_+$ as
\begin{equation}\label{eq:ExactShapeFunc}
 \ShapeFunc(\signal, \template) :=  \!\!
 \inf_{\substack{\diffeo \in \DiffeoGroup \\ \signal = \DeforOpG(\diffeo, \template) }} \!\!\!\!
 \LieGroupMetric(\diffeo,\Id)^2
 \quad\text{for $\signal \in \RecSpace$.}
\end{equation}
It is now possible to register $\template$ onto an indirectly observed target $\truesignal$ by solving the following variational problem:
\begin{equation}\label{eq:VarRegExactShape}
 \begin{cases}
\displaystyle{\inf_{\signal\in \RecSpace}} \, \Bigl[ 
     \gamma \ShapeFunc(\signal, \template) 
     + \mu\RegFunc(\signal) 
     +  \DataDisc\bigl( \ForwardOp(\signal), \data \bigr) \Bigr] & \\[1em]
  \signal = \DeforOpG(\diffeo,\template)
     \quad\text{on $\domain$ for some $\diffeo \in \DiffeoGroup$.} &
 \end{cases}
\end{equation}    
In the above, $\DataDisc \colon \DataSpace \times \DataSpace \to \Real$ is the data discrepancy functional introduced in \cref{eq:LeastSquares}, $\RegFunc \colon \RecSpace \to \Real$ is the \emph{regularity functional} that encodes further regularity properties of the target $\truesignal$ that are known before hand, and  $\ShapeFunc(\Cdot , \template) \colon \RecSpace \to \Real_+$ is the \emph{shape functional} defined in \cref{eq:ExactShapeFunc}. Finally, $\gamma,\mu \geq 0$ are regularization parameters in which $\gamma$ regulates the influence of the a priori shape information and $\mu$ regulates the a priori regularity information.
  
The constraint in \cref{eq:VarRegExactShape} simply states that the solution must be contained in the orbit of $\template$ under the group action. Furthermore,
\begin{equation}\label{eq:ShapeRegOnOrbit}
  \ShapeFunc( \Cdot, \template) \circ \DeforOpG(\Cdot,\template) = \LieGroupMetric(\Id,\Cdot)^2
  \quad\text{on $\DiffeoGroup$.}
\end{equation}     
Hence, \cref{eq:VarRegExactShape} can be reformulated as 
\begin{equation}\label{eq:OptimG}
 \inf_{\diffeo \in  \DiffeoGroup} \Big[  
     \gamma \LieGroupMetric(\Id,\diffeo)^2
    + \MatchingFunctionalX  \circ \DeforOpG(\diffeo, \template) \Bigr]
\end{equation}
where $\MatchingFunctionalX  \colon \RecSpace \to \Real$ is given by 
\begin{equation}\label{eq:MatchingFunctionalX}  
\MatchingFunctionalX (\signal) :=  \mu \RegFunc(\signal) + \DataDisc\bigl( \ForwardOp(\signal), \data \bigr)
    \quad\text{for $\signal \in \RecSpace$.} 
\end{equation}

\subsection{Reformulating the variational problem}\label{sec:Reformulation}
If $\DiffeoGroup$ is a Riemannian manifold, then one could solve \cref{eq:OptimG} using a Riemannian gradient method \cite{RiWi12}. It is however often more convenient to work with operators defined on vector spaces and $\DiffeoGroup$ lacks a natural vector space structure, \eg, the point-wise sum of two diffeomorphisms is not necessarily a diffeomorphism. This poses both mathematical and numerical difficulties. 

On the other hand, diffeomorphisms in $\DiffeoGroup$ are generated by velocity fields in $\FlowSpace{1}$ 
through \eqref{eq:FlowEq} and by \cref{eq:DiffMetric2}, it is sufficient to compute $\LieGroupMetric(\Id,\Cdot)^2$ for diffeomorphisms generated by velocity fields in $\FlowSpace{2}$. Hence, there is a natural vector space associated to $\DiffeoGroup$, namely $\FlowSpace{2}$, and a straightforward re-formulation of \cref{eq:OptimG} as an optimization over $\FlowSpace{2}$ yields 
\begin{equation}\label{eq:OptimV}
 \inf_{\velocityfield \in  \FlowSpace{2}} \Big[  
     \gamma \Vert \velocityfield \Vert_{\FlowSpace{2}}^2
     + \MatchingFunctionalX  \circ \DeforOpX(\velocityfield,\template) \Bigr]
\end{equation}
where $\DeforOpX \colon \FlowSpace{2} \times \RecSpace \to \RecSpace$ is the \emph{deformation operator} that is defined as
\begin{equation}\label{eq:DeforOpX}
  \DeforOpX(\velocityfield,\template) := \DeforOpG (\gelement{0,1}{\velocityfield},\template)
     \quad\text{where $\gelement{0,1}{\velocityfield} \in \DiffeoGroup$ solves \cref{eq:FlowEq}.}
\end{equation}

If $\LieAlgebra$ is admissible, then \cref{eq:OptimG} and \cref{eq:OptimV} are equivalent as shown in \cite[Theorem~11.2 and Lemma~11.3]{Yo10}, \ie, 
\begin{equation}\label{eq:Reformulation}
   \truevelocityfield \in\FlowSpace{2} \text{ solves \cref{eq:OptimV}}
   \iff
   \signal^{\ast} := \DeforOpX(\truevelocityfield,\template) \text{ solves \cref{eq:VarRegExactShape}.}
\end{equation}   
Note also that $\LieAlgebra$ is often infinite dimensional Hilbert space, in which case \cref{eq:OptimV} is a minimization over the infinite dimensional Hilbert space $\FlowSpace{2}$.

\subsubsection{\acs{PDE} constrained formulation}\label{sec:PDEFormulations}
Note that evaluating $\velocityfield \mapsto \DeforOpX(\velocityfield,\template)$ requires solving the \ac{ODE} in \cref{eq:FlowEq}, so the variational problem in \cref{eq:OptimV} is an \ac{ODE} constrained optimization problem:
\begin{equation}\label{eq:LDDMM:ODE}
  \begin{cases}
  \displaystyle{\min_{\velocityfield \in \FlowSpace{2}}} 
    \biggl[ \gamma \Vert \velocityfield \Vert_{\FlowSpace{2}}^2 
        + \MatchingFunctionalX  \circ \DeforOpG\bigl( \diffeo(1,\Cdot),\template\bigr)  
    \biggr]
    & \\[1em]
  \displaystyle{\frac{d}{dt}} \diffeo(t,\Cdot) = \velocityfield\bigl( t, \diffeo(t,\Cdot) \bigr)   
  \quad\text{on $\domain$ and $t \in [0,1]$,} & \\[0.5em] 
  \diffeo(0,\Cdot)= \Id \quad\text{on $\domain$.} &
  \end{cases}  
\end{equation}
For image registration the above can also be formulated as solving a \ac{PDE} constrained optimization problem with a time dependent image, see, \eg, \cite[eq.~(1)]{HoJoSaStNi12}. As to be shown next, such a reformulation easily adapts to the indirect image registration setting. 

Re-formulating \cref{eq:LDDMM:ODE} as a \ac{PDE} constrained minimization is based on observing that \cref{eq:FlowEq} is strictly related (via the method of characteristics) to a continuity equation. More precisely, start by considering the time derivative of 
\begin{equation}\label{eq:LDDMM:Orbit}
  \signal(t,x) :=  \DeforOpG\bigl(\diffeo(t,\Cdot), \template\bigr)(x) 
  \quad\text{for $x \in \domain$ and $t \in [0,1]$.}
\end{equation}
Since $\diffeo(t,\Cdot)$ is given by $\velocityfield  \in \FlowSpace{2}$ through \cref{eq:FlowEq}, it should be possible to express the 
time derivative of the right-hand-side of \cref{eq:LDDMM:Orbit} entirely in terms of $\signal$ and $\velocityfield $. 
Furthermore, 
\[  \signal(0,\Cdot)=\template
   \quad\text{and}\quad
   \signal(1,\Cdot)=\DeforOpG\bigl(\diffeo(1,\Cdot),\template \bigr)
   \quad\text{on $\domain$.}
\]
Hence, \cref{eq:LDDMM:ODE} can be re-stated as a \ac{PDE} constrained minimization with
objective functional
\begin{equation}\label{eq:LDDMM:PDE:Objective}
    \velocityfield  \mapsto \gamma \Vert \velocityfield  \Vert_{\FlowSpace{2}}^2
             + \MatchingFunctionalX  \bigl( \signal(1, \Cdot) \bigr).
\end{equation}
Note that $\signal(1, \Cdot)$ above depends on $\velocityfield$ since it depends on $\diffeo(1,\Cdot)$ through 
\cref{eq:LDDMM:Orbit} and $\diffeo(1,\Cdot)$ depends on $\velocityfield$ by the \ac{ODE} constraint  in 
\cref{eq:LDDMM:ODE}.
The precise form for the \ac{PDE} depends on the choice of group action, see \cref{sec:GroupAction} for a list of some natural group actions.

\paragraph{Geometric group action}
Let $\signal \colon [0,1] \times \domain \to \Real$ be given by \cref{eq:LDDMM:Orbit} and consider 
the group action is given by \cref{eq:Geometric}. Then, 
\begin{equation}\label{eq:LDDMM:Geometric:FP}
  \signal\bigl(t, \diffeo(t,\Cdot)\bigr) = \template
  \quad\text{on $\domain$ for $t \in [0,1]$ where $\diffeo$ solves \cref{eq:FlowEq}.}
\end{equation}
Differentiating \cref{eq:LDDMM:Geometric:FP} \wrt time $t$ and using \cref{eq:FlowEq} yields
\begin{multline*}
\partial_t \signal\bigl(t, \diffeo(t,\Cdot)\bigr) 
   + \biggl\langle 
         \grad \signal\bigl(t,\diffeo(t,\Cdot)\bigr), 
         \dfrac{d \diffeo(t,\Cdot)}{d t}
      \biggr\rangle_{\!\!\Real^n} \\
= \partial_t \signal\bigl(t, \diffeo(t,\Cdot)\bigr) 
   + \Bigl\langle 
         \grad \signal\bigl(t,\diffeo(t,\Cdot)\bigr), 
         \velocityfield \bigl(t,\diffeo(t,\Cdot)\bigr) 
      \Bigr\rangle_{\Real^n}\!\!\!\! = 0.
\end{multline*}
Since the above holds on $\domain$, it is equivalent to
\begin{equation}\label{eq:GroupAction:Geom:PDE}
  \partial_t \signal(t, \Cdot) + \Bigl\langle \grad \signal(t,\Cdot), \velocityfield (t,\Cdot) \Bigr\rangle_{\Real^n} 
  \!\!\!\! = 0
  \quad\text{on $\domain$ for $t \in [0,1]$.}  
\end{equation}
Hence, using the geometric group action in \cref{eq:OptimV} yields the following \ac{PDE} constrained formulation: 
\begin{equation}\label{eq:LDDMM:Geom:PDE}
\begin{cases}
  \displaystyle{\min_{\velocityfield  \in  \FlowSpace{2}}} \Big[  
     \gamma \Vert \velocityfield  \Vert_{\FlowSpace{2}}^2
             + \MatchingFunctionalX  \bigl( \signal(1, \Cdot) \bigr)  
    \Bigr] & \\[1em]
  \partial_t \signal(t, \Cdot) + \Bigl\langle \grad \signal(t,\Cdot), \velocityfield (t,\Cdot) \Bigr\rangle_{\Real^n}\!\!\!\! = 0 
  \quad\text{on $\domain$ and $t \in [0,1]$,} & \\[0.5em] 
  \signal(0,\Cdot) = \template \quad\text{on $\domain$.}
\end{cases}
\end{equation}

\paragraph{Mass-preserving group action}
Let $\signal \colon [0,1] \times \domain \to \Real$ be given by \cref{eq:LDDMM:Orbit} and consider 
the group action is given by \cref{eq:MassPreserving}. Then, 
\begin{equation}\label{eq:LDDMM:MassPreserving:FP}
  \Bigl\vert \Diff\bigl(\diffeo(t,\Cdot) \bigr)\Bigr\vert \, \signal\bigl(t, \diffeo(t,\Cdot)\bigr) 
    = \template 
  \quad\text{on $\domain$ and $t \in [0,1]$ with $\diffeo$ solving \cref{eq:FlowEq}.}
\end{equation}
Differentiating \cref{eq:LDDMM:MassPreserving:FP} \wrt time $t$ yields
\begin{displaymath}
\partial_t \signal\bigl(t, \diffeo(t,\Cdot)\bigr) 
   + \nabla\cdot\biggl(
         \signal(t, \Cdot)\, 
         \dfrac{d \diffeo(t,\Cdot)}{d t}
      \biggr) 
= \partial_t \signal\bigl(t, \diffeo(t,\Cdot)\bigr) 
   + \Div\Bigl( 
         \signal(t, \Cdot)\, 
         \velocityfield \bigl(t,\diffeo(t,\Cdot)\bigr)
      \Bigr) = 0.
\end{displaymath}
The last equality above makes use of \cref{eq:FlowEq} and the definition of the divergence operator. 
Since the above holds on $\domain$, it is equivalent to
\begin{equation}\label{eq:GroupAction:MassPreserving:PDE}
  \partial_t \signal(t, \Cdot) +  \Div\bigl(\signal(t, \Cdot)\, \velocityfield (t, \Cdot) \bigr)= 0
  \quad\text{on $\domain$ for $t \in [0,1]$.}  
\end{equation}
Hence, using the mass-preserving group action in \cref{eq:OptimV} yields the following \ac{PDE} constrained formulation: 
\begin{equation}\label{eq:LDDMM:MassPreserving:PDE}
\begin{cases}
  \displaystyle{\min_{\velocityfield  \in  \FlowSpace{2}}} \Big[  
     \gamma \Vert \velocityfield  \Vert_{\FlowSpace{2}}^2
                + \MatchingFunctionalX  \bigl( \signal(1, \Cdot) \bigr)  
    \Bigr] & \\[1em]
  \partial_t \signal(t, \Cdot) +  \Div\bigl(\signal(t, \Cdot)\, \velocityfield (t, \Cdot) \bigr) = 0 
  \quad\text{on $\domain$ and $t \in [0,1]$,} & \\[0.5em] 
  \signal(0,\Cdot) = \template \quad\text{on $\domain$.}
\end{cases}
\end{equation}

\begin{remark}
\Acp{PDE} are frequently used for defining similarity measures between images. As an example, registration by optical flow handles shapes as boundaries of objects which are then treated as fluids \cite{BePeSc15}. See also \cite{SoDaPa13} for further examples from medical imaging. 

In our setting, the \ac{PDE} constraint encodes \emph{both} the \ac{ODE} in \cref{eq:FlowEq} and the diffeomorphic group action. 
An advantage with the \ac{ODE} constrained formulation is that these two components, the generative model for the diffeomorphisms and their action on images, are explicit whereas in the \ac{PDE} constrained formulation they are ``hidden'' in the \ac{PDE}. Another potential advantage relates to computational feasibility. Within the \ac{LDDMM} theory, there are several  numerical methods based on various characterizations of minimizers to \eqref{eq:OptimV}, see \eg, \cite{Yo10}. 
On the other hand, the \ac{PDE} constrained formulation may be more suitable for considering a setting where the velocity fields are non-smooth. As outlined in \cref{sec:LDDMMFlows}, admissibility was a key assumption for ensuring that \cref{eq:FlowEq} is uniquely solvable. In such case, the velocity field is Lipschitz \wrt to space uniformly in time. 
An alternative to the \ac{ODE} based arguments in \cref{sec:LDDMMFlows} is to use Cauchy-Lipschitz theory and classical \ac{PDE} arguments to prove existence and uniqueness for \cref{eq:FlowEq}. These \ac{PDE} based techniques extend to certain non-smooth cases, \eg, \cite{PeLi89,Am04} proves existence and uniqueness of a certain solution to \cref{eq:FlowEq} (regular Lagrangian flow) even in the case for Sobolev and BV vector fields, see \cite{Co17} for further details.
Finally, the \ac{PDE} constrained formulation may also be better suited in applications where the actual image trajectory is of interest alongside the final deformed image.
\end{remark}

\section{Regularizing properties}\label{sec:RegularizingProperty}
The goal here is to investigate the regularizing properties of the variational reconstruction scheme in \cref{eq:VarRegExactShape} for solving \eqref{eq:InvProb} under the simplification where $\mu=0$. This involves determining wether \cref{eq:VarRegExactShape} has solutions, if they are unique, and assessing stability and convergence properties of these solutions. 

Following regularization theory, a reconstruction scheme is a regularization if existence holds along with stability and convergence \cite[Chapter~3]{ScGrGrHaLe09} (see also the notion of ``well-defined regularization method'' in \cite{Gr10}). Hence, even though formal uniqueness is a desirable property, it is not required from a regularization scheme. The starting point is thus to state continuity of the deformation operator \wrt to the velocity field (\cref{sec:ConvTemplate}). This is followed by arguments showing that regularizing properties for indirect registration can be transferred to regularizing properties of \cref{eq:VarRegExactShape} (\cref{sec:SimpReg}). The analysis of indirect registration is concluded by studying existence (\cref{sec:Existence}), uniqueness (\cref{sec:Unique}), and stability \& convergence (\cref{sec:StabConv}). 

\subsection{Convergence of deformed templates}\label{sec:ConvTemplate}
The aim here is to couple weak convergence of velocity fields in $\FlowSpace{2}$ to $\LpSpace^2$--convergence of corresponding deformed templates in $\RecSpace = \SBV(\domain,\Real) \bigcap \LpSpace^{\infty}(\domain,\Real)$.
\begin{theorem}\label{thm:ConvDeforTemplateG}
Assume $\domain \subset \Real^n$ is a bounded domain, $\LieAlgebra \subset \Smooth^1_0(\domain,\Real^n)$ is an admissible Hilbert space, and $\RecSpace = \SBV(\domain,\Real) \bigcap \LpSpace^{\infty}(\domain,\Real)$. Furthermore, let $\template \in \RecSpace$ and consider a 
sequence $\{ \velocityfield_k \}_k \subset \FlowSpace{2}$ that converges weakly to some $\velocityfield \in \FlowSpace{2}$, \ie, $\velocityfield_k  \rightharpoonup \velocityfield$ in $\FlowSpace{2}$. Then, 
\begin{equation}\label{eq:ConvDeforTemplateG}
  \DeforOpG(\gelement{0,1}{\velocityfield_k} , \template) \to \DeforOpG(\gelement{0,1}{\velocityfield} , \template)
    \quad\text{in $\LpSpace^2(\domain,\Real)$.}
\end{equation}
In the above, $\DeforOpG \colon \DiffeoGroup \times \RecSpace \to \RecSpace$ is given by \cref{eq:Geometric} for geometric group action and by \cref{eq:MassPreserving} for mass preserving group action.
\end{theorem}
\begin{proof}
The claim \cref{eq:ConvDeforTemplateG} for the case when $\DeforOpG$ is given by \cref{eq:Geometric} follows directly from \cite[Theorem~20 on p.~59]{Vi09}. Consider now the case when $\DeforOpG$ is given by \cref{eq:MassPreserving}. From \cref{thm:Thm3Vi09} we know that $\velocityfield_k  \rightharpoonup \velocityfield$ in $\FlowSpace{2}$ implies that $\Diff\bigl( (\gelement{0,1}{\velocityfield_k})^{-1}\bigr)$ converges uniformly to $\Diff\bigl( (\gelement{0,1}{\velocityfield})^{-1}\bigr)$ on compact subsets of $\domain$. Combined with previous argument proves the claim in \cref{eq:ConvDeforTemplateG} also when $\DeforOpG$ is given by \cref{eq:MassPreserving}. This concludes the proof.
\end{proof}
\begin{remark}
The above theorem handles the case of a non-differentiable template. If the template $\template$ is differentiable, then one does not have to introduce $\SBV$-functions and the corresponding result follows from the more widely know result \cite[Theorem~8.11]{Yo10}.
\end{remark}
It is also possible to state \cref{thm:ConvDeforTemplateG} in terms of the deformation operator given in \cref{eq:DeforOpX}. The corresponding claim, which is given in \cref{cor:ConvDeforTemplateV}, follows directly from \cref{thm:Thm3Vi09}, which states that weak convergence of velocity fields in $\FlowSpace{2}$ implies uniform convergence of corresponding diffeomorphisms on compact subsets. 
\begin{corollary}\label{cor:ConvDeforTemplateV}
Consider the deformation operator $\DeforOpX \colon \FlowSpace{2} \times \RecSpace \to \RecSpace$ defined by \cref{eq:DeforOpX} and assume the assumptions in \cref{thm:ConvDeforTemplateG} hold. Then $\DeforOpX(\velocityfield_k , \template) \to \DeforOpX(\velocityfield, \template)$ in $\LpSpace^2(\domain,\Real)$. 
\end{corollary}

\subsection{Reformulation}\label{sec:SimpReg}
With the a priori assumption in \cref{eq:ExactShapeAss}, the inverse problem in \eqref{eq:InvProb} can be re-phrased as the indirect image registration problem of recovering the velocity field $\truevelocityfield \in \FlowSpace{2}$ given data $\data \in \DataSpace$ and template $\template \in \RecSpace$ such that 
\begin{equation}\label{eq:InvProbV}
     \data = \ForwardOp\bigl( \DeforOpX(\truevelocityfield,\template) \bigr)+ \noisedata. 
\end{equation}
In the above, $\DeforOpX(\Cdot,\template) \colon \FlowSpace{2} \to \RecSpace$ is given by \cref{eq:DeforOpX} whereas $\ForwardOp \colon \RecSpace \to \DataSpace$ and $\noisedata \in \DataSpace$ come from \cref{eq:InvProb}. The variational scheme for solving \eqref{eq:InvProbV} is given in \cref{eq:OptimV}. Hence, a natural question is whether regularizing properties for \cref{eq:OptimV} carry over to \cref{eq:VarRegExactShape}.

By \cref{eq:Reformulation}, any solution for \cref{eq:OptimG} yields a solution for \cref{eq:VarRegExactShape} whenever $\MatchingFunctionalX \colon \RecSpace \to \Real$ is given as in \cref{eq:MatchingFunctionalX}. Next, by \cite[Lemma~11.3]{Yo10}, \cref{eq:OptimG} is equivalent to \cref{eq:OptimV}. Hence, existence and uniqueness of solutions for \cref{eq:OptimV} implies the same for \cref{eq:VarRegExactShape}. Next, if $\signal \to \DataDisc\bigl( \ForwardOp(\signal), \data \bigr)$ is continuous, then by \cref{cor:ConvDeforTemplateV}, stability and convergence for \cref{eq:OptimV} implies the same for \cref{eq:VarRegExactShape}. Hence, under the priori assumption in \cref{eq:ExactShapeAss}, regulaizing properties of \cref{eq:OptimV} carry over to \cref{eq:VarRegExactShape}. 

Finally, we consider the somewhat simplified case when $\mu=0$ in \cref{eq:VarRegExactShape}, so the corresponding image registration scheme in \cref{eq:OptimV} that will be analyzed is to minimize the objective functional $\ObjectiveV_{\data,\lambda}  \colon \FlowSpace{2} \to \Real$ given by 
\begin{equation}\label{eq:ObjectiveGoalFunctionalV}
 \ObjectiveV_{\data,\gamma}(\velocityfield) := 
   \gamma \Vert \velocityfield \Vert_{\FlowSpace{2}}^2 
   + \DataDisc\Bigl( \ForwardOp\bigl(\DeforOpX(\velocityfield,\template)\bigr), \data \Bigr)
   \quad\text{for $\velocityfield \in \FlowSpace{2}$}  
\end{equation}
with $\DeforOpX \colon \FlowSpace{2} \times \RecSpace \to \RecSpace$ given by \cref{eq:DeforOpX} and $\DataDisc \colon \DataSpace \times \DataSpace \to \Real^+$ is the data discrepancy functional introduced in \cref{eq:LeastSquares}.

\subsection{Existence}\label{sec:Existence}
Existence of solutions is a basic property, since otherwise the reconstruction scheme is of limited usefulness for reconstruction. 
\begin{theorem}[Existence]\label{thm:Existence}
Assume $\domain \subset \Real^n$ is a bounded domain, $\LieAlgebra \subset \Smooth^1_0(\domain,\Real^n)$ is an admissible Hilbert space, $\template \in \SBV(\domain,\Real) \bigcap \LpSpace^{\infty}(\domain,\Real)$, and both $\DataDisc\bigl( \ForwardOp(\Cdot), \data \bigr), \RegFunc(\Cdot) \colon \RecSpace \to \Real$ are lower semi-continuous. Then, $\ObjectiveV_{\data,\lambda}  \colon \FlowSpace{2} \to \Real$ in \cref{eq:ObjectiveGoalFunctionalV} has a minimizer, \ie, the variational problem in \cref{eq:OptimV} has a solution expressible as $\DeforOpX(\velocityfield,\template) \in \RecSpace$ for some $\velocityfield \in \FlowSpace{2}$.
\end{theorem}
\begin{proof}   
First, $\LieAlgebra$ is separable so $\FlowSpace{2}$ is separable, which in turn implies that bounded balls are weakly compact. Next, let $\{ \velocityfield_k \}_k \subset \FlowSpace{2}$ be a minimizing sequence. Then, there exists a subsequence that weakly converges to an element $\velocityfield \in \FlowSpace{2}$, \ie, $\velocityfield_k  \rightharpoonup \velocityfield$ in $\FlowSpace{2}$. Now, if the objective functional in \cref{eq:ObjectiveGoalFunctionalV} is lower semi-continuous in $\FlowSpace{2}$ \wrt the weak topology and if, then $\velocityfield$ is also minimizer. This proves existence of solutions for \cref{eq:OptimV}. 

From the above, existence of solutions for \cref{eq:OptimV} holds whenever the functional in \cref{eq:ObjectiveGoalFunctionalV} is lower semi-continuous on $\FlowSpace{2}$ \wrt the weak topology. This requires making use of the admissibility of $\LieAlgebra$ and the assumption that $\template \in \SBV(\domain,\Real) \bigcap \LpSpace^{\infty}(\domain,\Real)$. By \cref{thm:ConvDeforTemplateG}
\[ \velocityfield_k  \rightharpoonup \velocityfield \text{ in $\FlowSpace{2}$} 
   \quad\Longrightarrow\quad
   \template \circ \gelement{1,0}{\velocityfield_k} \to \template \circ \gelement{1,0}{\velocityfield} \text{ in $\LpSpace^2(\domain,\Real)$.}  
\]
Furthermore, $\gelement{1,0}{\velocityfield_k}$ and $\Diff(\gelement{1,0}{\velocityfield_k})$ converge uniformly on $\domain$ to $\gelement{1,0}{\velocityfield}$ and $\Diff(\gelement{1,0}{\velocityfield})$, respectively. In particular, $\velocityfield \mapsto \DeforOpX(\velocityfield,\template)$ is continuous in the weak topology on $\FlowSpace{2}$ for the group actions listed in \cref{sec:GroupAction}. Next, $\signal \mapsto \DataDisc\bigl( \ForwardOp(\signal), \data \bigr)$ is by assumption lower semi-continuous on $\RecSpace$, so $\velocityfield \mapsto \MatchingFunctionalX  \circ \DeforOpX(\velocityfield,\template)$ is lower semi-continuous in the weak topology of $\FlowSpace{2}$, \ie, the objective functional in \cref{eq:OptimV} lower semi-continuous.
This concludes the proof.
\end{proof}  

\subsection{Uniqueness}\label{sec:Unique}
Another desirable property is uniqueness, which in this context means \cref{eq:OptimV} has a unique solution for given data. Unfortunately, $\velocityfield \mapsto \DeforOpX(\velocityfield,\template)$ is not necessarily convex, so the objective functional in \cref{eq:OptimV} is not necessarily convex even when both $\signal \to \RegFunc(\signal)$ and $\signal \to \DataDisc\bigl( \ForwardOp(\signal), \data \bigr)$ are strictly convex. Hence, \cref{eq:OptimV} may not have a unique solution. 

As with all reconstruction methods that involve solving non-convex optimization problems, there is always the issue of getting stuck in local extrema. One option to address this is to further restrict the set $\LieAlgebra$ of velocity fields, but it is highly non-trivial to work out conditions on $\LieAlgebra$ that would guarantee uniqueness. 

\subsection{Stability and convergence}\label{sec:StabConv}
Stability essentially means that reconstructions are continuous \wrt data and convergence refers to the case when data in \cref{eq:InvProb} tends to ideal data, \ie, as $\data \to \ForwardOp(\truesignal)$ in $\DataSpace$. Both are essential parts of a regularization scheme and this section proves stability and convergence results. Preferably such results are complemented with convergence rates and stability estimates, but they are left for the future.

The starting point is to show that minimizing $\ObjectiveV_{\data,\gamma}  \colon \FlowSpace{2} \to \Real$ in \cref{eq:ObjectiveGoalFunctionalV} defines a reconstruction scheme for \cref{eq:InvProbV} that is stable in the weak sense. The precise formulation reads as follows:
\begin{theorem}[Stability]\label{thm:Stability}
In addition to the assumptions in \cref{thm:Existence}, let $\data \in \DataSpace$ be fixed, $\RecSpace =  \LpSpace^2(\domain,\Real)$, the forward operator $\ForwardOp \colon \RecSpace \to \DataSpace$ is continuous, and the data discrepancy $\DataDisc\bigl( \ForwardOp(\Cdot), \data \bigr) \colon \RecSpace \to \Real^+$ is  continuous. Finally, assume furthermore that $\{ \data_k \}_k \subset \DataSpace$ where $\data_k \to \data$. 
If $\velocityfield_k \in \FlowSpace{2}$ as a minimizer of $\velocityfield \mapsto \ObjectiveV_{\data_k,\gamma}(\velocityfield)$ for some $k$ and $\gamma\geq 0$, then there exists a sub-sequence of $\{ \velocityfield_k \}_k$ that converges weakly towards a minimizer of $\ObjectiveV_{\data,\gamma}$ in \cref{eq:ObjectiveGoalFunctionalV}.
\end{theorem}
\begin{proof}   
First, since the data discrepancy is positive and $\velocityfield_k \in \FlowSpace{2}$ as a minimizer of $\velocityfield \mapsto \ObjectiveV_{\data_k,\gamma}(\velocityfield)$ in \cref{eq:ObjectiveGoalFunctionalV}, one has 
\[  
    \Vert \velocityfield_k \Vert_{\FlowSpace{2}}^2
    \leq \frac{1}{\gamma}\ObjectiveV_{\data_k,\gamma}(\velocityfield_k) 
    \leq \frac{1}{\gamma}\ObjectiveV_{\data_k,\gamma}(\velocityfieldzero) 
    \quad\text{for each $k$ and $\gamma\geq 0$.}
\]
In the above, $\velocityfieldzero \in \FlowSpace{2}$ denotes the zero-velocity field. This velocity field has zero $\FlowSpace{2}$-norm. Next, both 
the forward operator and the data discrepancy are continuous, so 
\[  \ObjectiveV_{\data_k,\gamma}(\velocityfieldzero) = 
      \DataDisc\Bigl( \ForwardOp\bigl( \DeforOpX(\velocityfieldzero,\template) \bigr), \data_k \Bigr)
      \to 
      \DataDisc\Bigl( \ForwardOp\bigl( \DeforOpX(\velocityfieldzero,\template) \bigr), \data \Bigr)       
      \quad\text{as $\data_k \to \data$ in $\DataSpace$.}
\]       
Hence, $\ObjectiveV_{\data_k,\gamma}(\velocityfieldzero)$ is a bounded sequence and therefore $\{ \velocityfield_k \}_k \subset \FlowSpace{2}$ is bounded. Hence, it has a subsequence that converges weakly to some $\widehat{\velocityfield} \in \FlowSpace{2}$. The continuity of the forward operator along with \Cref{cor:ConvDeforTemplateV} implies
\[ \ForwardOp\bigl( \DeforOpX(\velocityfield_k,\template) \bigr) \to 
   \ForwardOp\bigl( \DeforOpX(\widehat{\velocityfield},\template) \bigr)  
   \quad\text{in $\DataSpace$ as $\velocityfield_k \rightharpoonup \widehat{\velocityfield}$ in $\FlowSpace{2}$.}
\]
Since the data discrepancy is continuous, 
\[  
   \DataDisc\Bigl( \ForwardOp\bigl( \DeforOpX(\velocityfield_k,\template) \bigr), \data_k \Bigr)
    \to 
   \DataDisc\Bigl( \ForwardOp\bigl( \DeforOpX(\widehat{\velocityfield},\template) \bigr), \data \Bigr)
    \quad\text{as $\data_k \to \data$ in $\DataSpace$.}
\]    
Finally, $\Vert \widehat{\velocityfield} \Vert_{\FlowSpace{2}} \leq \liminf_{k} \Vert \velocityfield_k \Vert_{\FlowSpace{2}}$ and $\velocityfield_k$ is a minimizer of $\ObjectiveV_{\data_k,\gamma}$, so 
\[   \ObjectiveV_{\data,\gamma}(\widehat{\velocityfield}) 
     \leq \liminf_k  \ObjectiveV_{\data_k,\gamma}(\velocityfield_k) \leq
     \liminf_k  \ObjectiveV_{\data_k,\gamma}(\velocityfield)
     \quad\text{for any $\velocityfield \in \FlowSpace{2}$.}
\]
Hence, $\ObjectiveV_{\data,\gamma}(\widehat{\velocityfield}) \leq \ObjectiveV_{\data,\gamma}(\velocityfield)$ as $\data_k \to \data$, but this holds for any $\velocityfield \in \FlowSpace{2}$, so $\widehat{\velocityfield}$ is a minimizer of $\ObjectiveV_{\data,\gamma}$. 

In conclusion, the reconstructed velocity fields depend continuously on data in the weak sense, which in turn concludes the proof.
\end{proof}
\begin{remark}
A natural open question is to study whether the same result holds in the strong sense, \ie, whether there exists a sub-sequence of $\{ \velocityfield_k \}_k$ that converges strongly to a minimizer of $\ObjectiveV_{\data,\gamma}$.
\end{remark}

Next is to consider convergence, again in the weak sense.
\begin{theorem}[Convergence]\label{thm:Convergence}
Assume the conditions in \cref{thm:Stability} holds and there exists some $\truevelocityfield \in \FlowSpace{2}$ such that 
$\ForwardOp\bigl(\DeforOpX(\truevelocityfield,\template)\bigr) = \data$. Furthermore, let $\{ \data_k \}_k \subset \DataSpace$ be sequence of data and $\delta_k$ is a sequence of positive real numbers converging to $0$ such that 
\[ \DataDisc( \data_k, \data) \leq \delta_k^2 \quad\text{for each $k$.} \]
Finally, assume that there is a parameter choice rule $\gamma \colon \Real \to \Real$ such that 
\begin{equation}\label{eq:ParamSelectionRule}
 \gamma(\delta) \to 0 
   \quad\text{and}\quad
   \dfrac{\delta^2}{\gamma(\delta)} \to 0 
   \quad\text{when $\delta \to 0$,}
\end{equation}
and let  $\velocityfield_k \in \FlowSpace{2}$ be a minimizer of $\ObjectiveV_{\data_k,\gamma(\delta_k)}  \colon \FlowSpace{2} \to \Real$ given by \cref{eq:ObjectiveGoalFunctionalV}. Then, there exists a sub-sequence of $\{ \velocityfield_k \}_k$ that converges weakly towards an element in the pre-image of $\data$ under the mapping $\ForwardOp\bigl(\DeforOpX(\Cdot,\template)\bigr)$.
\end{theorem}
\begin{proof}   
By assumption, there is a $\truevelocityfield \in \FlowSpace{2}$ such that $\ForwardOp\bigl(\DeforOpX(\truevelocityfield,\template)\bigr) = \data$. Then, 
\[ \Vert \velocityfield_k \Vert_{\FlowSpace{2}}^2 
   \leq \frac{1}{\gamma_k}  \ObjectiveV_{\data_k,\gamma_k}(\velocityfield_k) 
   \leq \frac{1}{\gamma_k}  \ObjectiveV_{\data_k,\gamma_k}(\truevelocityfield) 
\quad\text{where $\gamma_k := \gamma(\delta_k)$.}
\]
The last inequality above follows from the assumption that $\velocityfield_k$ is a minimizer of $\velocityfield \mapsto \ObjectiveV_{\data_k,\gamma_k}(\velocityfield)$. Moreover, \cref{eq:ParamSelectionRule} gives 
\begin{align*}
  \frac{1}{\gamma_k}  \ObjectiveV_{\data_k,\gamma_k}(\truevelocityfield) &= 
    \frac{1}{\gamma_k}  \Bigl[ \DataDisc(\data,\data_k) + \gamma_k \Vert  \truevelocityfield \Vert_{\FlowSpace{2}}^2 \Bigr] \\
    &\leq \frac{{\delta_k}^2}{\gamma_k} + \Vert  \truevelocityfield \Vert_{\FlowSpace{2}}^2
    \to \Vert  \truevelocityfield \Vert_{\FlowSpace{2}}^2 
    \quad\text{as $\delta_k \to 0$.}
\end{align*}    
Hence, $\{ \velocityfield_k \}_k \subset \FlowSpace{2}$ is bounded, so it has a sub-sequence that converges weakly to $\widetilde{\velocityfield} \in \FlowSpace{2}$. Since both the forward operator and data discrepancy are continuous, so 
\begin{equation}\label{eq:Conv1}
\DataDisc\Bigl( \ForwardOp\bigl(\DeforOpX(\velocityfield_k,\template)\bigr),\data_k \Bigr) \to 
   \DataDisc\Bigl( \ForwardOp\bigl(\DeforOpX(\widetilde{\velocityfield},\template)\bigr),\data \Bigr)
    \quad\text{as $\delta_k \to 0$.}   
\end{equation}
Furthermore, 
\begin{equation}\label{eq:Conv2}
\DataDisc\Bigl( \ForwardOp\bigl(\DeforOpX(\velocityfield_k,\template)\bigr),\data_k \Bigr) \
\leq
\ObjectiveV_{\data_k,\gamma}(\velocityfield_k)
\leq
\ObjectiveV_{\data_k,\gamma}(\truevelocityfield) 
= \DataDisc( \data,\data_k) + \gamma_k \bigl\Vert \truevelocityfield \bigr\Vert_{\FlowSpace{2}}^2 
\quad\text{for any $k$.}
\end{equation}
However, $\DataDisc( \data,\data_k) \to 0$ as $\delta_k \to 0$ and $\gamma_k \to 0$ by \cref{eq:ParamSelectionRule}, so \cref{eq:Conv1} and \cref{eq:Conv2} implies that $\ForwardOp\bigl(\DeforOpX(\widetilde{\velocityfield},\template)\bigr)=\data$, which implies that $\{ \velocityfield_k \}_k \subset \FlowSpace{2}$ has a sub-sequence that converges weakly as data error goes to zero to $\truevelocityfield \in \FlowSpace{2}$ where $\ForwardOp\bigl(\DeforOpX(\truevelocityfield,\template)\bigr) = \data$. This concludes the proof of convergence.
\end{proof}

\section{Derivative and gradient computations}\label{sec:Gradient}
To goal here is to compute the derivative and gradient of the objective functional $\ObjectiveV  \colon \FlowSpace{2} \to \Real$ in \cref{eq:OptimV}, \ie, so
\begin{equation}\label{eq:GenericGoalFunctionalV}
 \ObjectiveV(\velocityfield) := 
   \gamma \Vert \velocityfield \Vert_{\FlowSpace{2}}^2 
   + \MatchingFunctionalX  \circ \DeforOpX(\velocityfield,\template)
   \quad\text{for $\velocityfield \in \FlowSpace{2}$}  
\end{equation}
where $\MatchingFunctionalX  \colon \RecSpace \to \Real$ given by \cref{eq:MatchingFunctionalX} and $\DeforOpX \colon \FlowSpace{2} \times \RecSpace \to \RecSpace$ given by \cref{eq:DeforOpX}. Since $\ObjectiveV$ is defined on the Hilbert space $\FlowSpace{2}$, the gradient associated with its G\^ateaux derivative will be based on the Hilbert structure of $\FlowSpace{2}$, which is the natural inner-product space for the optimization problem \cref{eq:OptimV}. The derivation will assume a differentiable template, but one can in fact work with templates in $\SBV(\domain,\Real) \bigcap \LpSpace^{\infty}(\domain,\Real$ that are non-smooth as shown in \cite[Chapter~2]{Vi09}.

The starting point is to compute the derivative of the deformation operator $\DeforOpX$ (\cref{sec:DeforOpDiff}). Next step is to use this in computing the gradient of the registration functional $\MatchingFunctionalX  \circ \DeforOpX(\Cdot,\template) \colon \FlowSpace{2} \to \Real$. This is done in \cref{sec:RegFunc} for an abstract registration functional $\MatchingFunctionalX$, the case of the $2$-norm is worked out in \cref{sec:L2regFunc}. The gradient of the $\Vert \Cdot \Vert_{\FlowSpace{2}}^2$ is straightforward (\cref{sec:ShapeFuncGrad}) and all these pieces are put together in \cref{sec:GradObjFunc} to obtain expressions for the gradient of the objective functional in \cref{eq:GenericGoalFunctionalV}. The explicit expressions for the gradient assume $\LieAlgebra$ is an admissible \ac{RKHS} with a continuous positive definite reproducing kernel $\Koperator \colon \domain \times \domain \to \LinearSpace(\Real^{n}, \Real^{n})$ represented by a function $\kernel \colon \domain \times \domain \to \Matrix_{+}^{n \times n}$ (see \cref{sec:RKHS}). 

\subsection{The deformation operator}\label{sec:DeforOpDiff}
The aim here is to compute the derivative of the deformation operator 
$\DeforOpX \colon \FlowSpace{2} \times \RecSpace \to \RecSpace$
in \cref{eq:DeforOpX}. The expressions will depend on the group action, see \cref{sec:DeforOpGeomAction} for the case with geometric group action and \cref{sec:DeforOpMassPre} for the case with mass-preserving group action. The starting point however is to derive the derivative of diffeomorphisms in $\DiffeoGroup$ \wrt the underlying velocity field.

\subsubsection{Derivative of diffeomorphisms \wrt the velocity field}
Each diffeomorphism in $\DiffeoGroup$ depends on an underlying velocity field in $\FlowSpace{2}$. The derivative of such a diffeomorphism \wrt the underlying velocity field can be explicitly computed and is given by the following theorem. 
\begin{proposition}\label{thm:DerivDiffeoV}
Let $\velocityfield,\velocityfieldother \in \FlowSpace{2}$ and let $\gelement{s,t}{\velocityfield} \in \DiffeoGroup$ be a a solution to \cref{eq:FlowEq}. Then,
\begin{align}
  \label{eq:PhiDeriv}
  \frac{d}{d \epsilon} \bigl( \gelement{s,t}{\velocityfield + \epsilon\velocityfieldother} \bigr) \Bigl\vert_{\epsilon=0} (x)
    &= \int_{s}^{t} \Diff\bigl( \gelement{\tau,t}{\velocityfield} \bigr)\bigl( \gelement{s,\tau}{\velocityfield}(x) \bigr) 
             \Bigl( \velocityfieldother\bigl(\tau, \gelement{s,\tau}{\velocityfield}(x) \bigr) \Bigr) \dint \tau 
    \\    
    \label{eq:PhiInvDeriv}
    \frac{d}{d \epsilon} \bigr( (\gelement{s,t}{\velocityfield + \epsilon\velocityfieldother})^{-1} \bigr)\Bigl\vert_{\epsilon=0} (x)
    &=  - \int_{s}^{t} \Bigl[ \Diff\bigl(\gelement{s,\tau}{\velocityfield}\bigr)\bigl( \gelement{t,s}{\velocityfield}(x)\bigr) \Bigr]^{-1}
           \Bigl( \velocityfieldother\bigl(\tau,\gelement{t,\tau}{\velocityfield}(x)\bigr) \Bigr) \dint \tau
\end{align}
for $x \in \domain$ and $0 \leq s,t \leq 1$. 
\end{proposition}
\begin{proof}
The first equality \cref{eq:PhiDeriv} follows directly from \cite[Lemma~11.5]{GrMi07} (see also \cite[Theorem~8.10]{Yo10}) and 
the second, \cref{eq:PhiInvDeriv}, follows directly from \cite[Lemma~12.8]{GrMi07} (see 
also \cite[eq.~(10.16)]{Yo10}) and \cref{eq:phiRelations}.
\end{proof}

\subsubsection{Geometric group action}\label{sec:DeforOpGeomAction}
For the (left) geometric group action in \cref{eq:Geometric}, the deformation operator becomes 
\begin{equation}\label{eq:DeforOpGeom}
  \DeforOpX(\velocityfield,\template) = \template \circ (\gelement{0,1}{\velocityfield})^{-1} 
    = \template \circ \gelement{1,0}{\velocityfield}
\end{equation}  
where the last equality follows from \cref{eq:phiRelations}. Its G\^ateaux derivative is given by the next result.
\begin{theorem}\label{thm:DiffDeforOpGeom}
Let $\LieAlgebra$�be admissible and $\template \in \RecSpace$ is differentiable. The deformation operator 
$\DeforOpX(\Cdot,\template) \colon \FlowSpace{2} \to \RecSpace$ in \cref{eq:DeforOpGeom} is then G\^ateaux differentiable 
at $\velocityfield \in \FlowSpace{2}$ and its G\^ateaux derivative is 
\begin{equation}\label{eq:DiffDeforOpGeom}
  \GDiff \DeforOpX(\velocityfield,\template)(\velocityfieldother)(x) 
  = - \int_{0}^{1} 
      \biggl\langle 
      \grad (\template \circ \gelement{t,0}{\velocityfield})\bigl( \gelement{1,t}{\velocityfield}(x)\bigr),
           \velocityfieldother\bigl(t,\gelement{1,t}{\velocityfield}(x) \Bigr) 
      \biggr\rangle_{\!\!\Real^n}\!\!\!\! \dint t                        
\end{equation}  
for $x \in \domain$ and $\velocityfieldother \in \FlowSpace{2}$. 
\end{theorem}

\begin{proof}
The G\^ateaux derivative is a linear mapping
$\GDiff \DeforOpX(\velocityfield,\template) \colon \FlowSpace{2} \to  \RecSpace$
given as 
\[  \GDiff \DeforOpX(\velocityfield, \template)(\velocityfieldother)(x) := 
       \frac{d}{d \epsilon} \DeforOpX(\velocityfield+ \epsilon\velocityfieldother,\template)(x)\Bigl\vert_{\epsilon=0}
       = \frac{d}{d \epsilon} \Bigl( \template \circ (\gelement{0,1}{\velocityfield+ \epsilon\velocityfieldother})^{-1} \Bigr)(x)\Bigl\vert_{\epsilon=0}
\]       
for $\velocityfieldother \in \FlowSpace{2}$ and $x \in \domain$. Since $\template \in \RecSpace \subset \LpSpace^{2}(\domain,\Real)$ is differentiable, the chain rule yields 
\begin{equation}\label{eq:WDiff}  
  \GDiff \DeforOpX(\velocityfield, \template)(\velocityfieldother)(x) = 
      \biggl\langle 
      \grad \template \bigl(\gelement{1,0}{\velocityfield}(x)\bigr),
      \frac{d}{d \epsilon} \bigl( (\gelement{0,1}{\velocityfield+ \epsilon\velocityfieldother})^{-1} \bigr)\Bigl\vert_{\epsilon=0} (x)     
      \biggr\rangle_{\!\!\Real^n}
       \quad\text{for $x \in \domain$.}      
\end{equation}
The second term in the scalar product on the right hand side of \cref{eq:WDiff} is the derivative of a flow with 
respect to variations in the associated  field. The following equation now follows 
from \cref{eq:PhiInvDeriv} in \cref{thm:DerivDiffeoV}.
\begin{equation}\label{eq:DiffDeforOpGeom_eq1}
  \GDiff \DeforOpX(\velocityfield,\template)(\velocityfieldother)(x) = 
      - \int_{0}^{1} 
      \biggl\langle 
      \grad \template \bigl(\gelement{1,0}{\velocityfield}(x)\bigr),
      \Bigl[ \Diff\bigl(\gelement{0,t}{\velocityfield}\bigr)\bigl( \gelement{1,0}{\velocityfield}(x) \bigr) \Bigr]^{-1}
           \Bigl( \velocityfieldother\bigl(t,\gelement{1,t}{\velocityfield}(x)\bigr) \Bigr) 
      \biggr\rangle_{\!\!\Real^n}\!\!\!\! \dint t. 
\end{equation}

To prove \cref{eq:DiffDeforOpGeom}, consider first the chain rule: 
\begin{displaymath}
  \Diff( \gelement{0,t}{\velocityfield}\circ \gelement{1,0}{\velocityfield} )(x) = 
        \Diff(\gelement{0,t}{\velocityfield})\bigl( \gelement{1,0}{\velocityfield}(x) \bigr) \circ \Diff(\gelement{1,0}{\velocityfield})(x).
\end{displaymath}     
Next, $\gelement{0,t}{\velocityfield}\circ \gelement{1,0}{\velocityfield}= \gelement{1,t}{\velocityfield}$, so  
\begin{displaymath}
  \Diff(\gelement{0,t}{\velocityfield})\bigl( \gelement{1,0}{\velocityfield}(x) \bigr) =
      \Diff( \gelement{1,t}{\velocityfield} )(x) 
         \circ \bigr[ \Diff(\gelement{1,0}{\velocityfield})(x)\bigl]^{-1},
\end{displaymath}    
which in turn implies that, 
\begin{equation}\label{eq:SimplerDiffEq}
  \Bigl[  \Diff(\gelement{0,t}{\velocityfield})\bigl( \gelement{1,0}{\velocityfield}(x) \bigr) \Bigr]^{-1} 
    =  \Diff(\gelement{1,0}{\velocityfield})(x) \circ \Bigl[ \Diff( \gelement{1,t}{\velocityfield} )(x) \Bigr]^{-1}.
\end{equation}
Inserting \cref{eq:SimplerDiffEq} into \cref{eq:DiffDeforOpGeom_eq1} yields
\begin{align*}
  \GDiff \DeforOpX(\velocityfield, \template)(\velocityfieldother)(x) 
   & = - \int_{0}^{1} 
      \biggl\langle 
      \grad \template \bigl(\gelement{1,0}{\velocityfield}(x)\bigr),
        \Diff(\gelement{1,0}{\velocityfield})(x) \circ \Bigl[ \Diff( \gelement{1,t}{\velocityfield} )(x) \Bigr]^{-1}
        \Bigl( \velocityfieldother\bigl(t,\gelement{1,t}{\velocityfield}(x)\bigr) \Bigr) 
      \biggr\rangle_{\!\!\Real^n}\!\!\!\! \dint t \\
   &  = - \int_{0}^{1} 
      \biggl\langle 
        \bigr[ \Diff(\gelement{1,0}{\velocityfield})(x)\bigl]^{\ast} \Bigl( \grad \template \bigl(\gelement{1,0}{\velocityfield}(x)\bigr) \Bigr),
        \Bigl[ \Diff( \gelement{1,t}{\velocityfield} )(x) \Bigr]^{-1}
        \Bigl( \velocityfieldother\bigl(t,\gelement{1,t}{\velocityfield}(x)\bigr) \Bigr) 
      \biggr\rangle_{\!\!\Real^n}\!\!\!\! \dint t \\      
 &  = - \int_{0}^{1} 
      \biggl\langle 
      \grad (\template \circ \gelement{1,0}{\velocityfield})(x),
      \bigl[ \Diff( \gelement{1,t}{\velocityfield} )(x) \bigr]^{-1} 
           \Bigl( \velocityfieldother\bigl(t,\gelement{1,t}{\velocityfield}(x)\bigr) \Bigr) 
      \biggr\rangle_{\!\!\Real^n}\!\!\!\! \dint t \\
 &  = - \int_{0}^{1} 
      \biggl\langle 
      \grad (\template \circ \gelement{t,0}{\velocityfield})\bigl( \gelement{1,t}{\velocityfield}(x)\bigr),
           \velocityfieldother\bigl(t,\gelement{1,t}{\velocityfield}(x) \Bigr) 
      \biggr\rangle_{\!\!\Real^n}\!\!\!\! \dint t.
      \end{align*}                  
The last equality above follows from
\begin{displaymath}
\grad (\template \circ \gelement{1,0}{\velocityfield})(x) = 
  \grad( \template \circ \gelement{t,0}{\velocityfield} \circ \gelement{1,t}{\velocityfield})(x) 
  = \bigr[ \Diff(\gelement{1,t}{\velocityfield})(x)\bigl]^{\ast} 
       \Bigl( \grad ( \template \circ \gelement{t,0}{\velocityfield} )\bigl( \gelement{1,t}{\velocityfield}(x) \bigr)
       \Bigr).       
\end{displaymath}
This concludes the proof of \cref{thm:DiffDeforOpGeom}.
\end{proof}

\subsubsection{Mass-preserving group action}\label{sec:DeforOpMassPre}
Under the (left) mass-preserving group action in \cref{eq:MassPreserving}, the deformation operator becomes 
\begin{equation}\label{eq:DeforOpMP}
  \DeforOpX(\velocityfield, \template) 
    = \bigl\vert \Diff \bigl((\gelement{0,1}{\velocityfield})^{-1} \bigr) \bigr\vert\, \bigl( \template \circ (\gelement{0,1}{\velocityfield})^{-1} \bigr) 
    = \bigl\vert \Diff (\gelement{1,0}{\velocityfield}) \bigr\vert\, \bigl( \template \circ \gelement{1,0}{\velocityfield} \bigr)
\end{equation}
where the last equality follows from \cref{eq:phiRelations}. Its G\^ateaux derivative is given in the next result.
\begin{theorem}\label{thm:DiffDeforOpMP}
If $\template \in \RecSpace$ is differentiable, then $\DeforOpX(\Cdot,\template) \colon \FlowSpace{2} \to \RecSpace$ in \cref{eq:DeforOpMP} is G\^ateaux differentiable 
at $\velocityfield \in \FlowSpace{2}$ and its G\^ateaux derivative is 
\begin{equation}\label{eq:DiffDeforOpMP}
  \GDiff \DeforOpX(\velocityfield, \template)(\velocityfieldother)(x) = 
  \bigl\vert\Diff \bigl(\gelement{1,0}{\velocityfield}\bigr)(x)\bigr\vert 
    \Div\bigl(\template(\Cdot) h(\Cdot)\bigr) \circ \gelement{1,0}{\velocityfield} (x)
%  -\int_{0}^{1}  
%     \biggl( \Div \Bigl[ 
%         \Diff\bigl( \gelement{t,0}{\velocityfield} \bigr)\bigl( \gelement{0,t}{\velocityfield} \bigr)  
%         \Bigl( \velocityfieldother\bigl(t, \gelement{0,t}{\velocityfield} \bigr) \Bigr) \template
%       \Bigr] \circ \gelement{1,0}{\velocityfield} 
%     \biggr)(x) \bigl\vert\Diff (\gelement{1,0}{\velocityfield})(x)\bigr\vert \dint t
\quad\text{for $x \in \domain$ and $\velocityfieldother \in \FlowSpace{2}$,}
  \end{equation}  
where
\begin{equation}\label{eq:h_simplify}
 h(x) := -\int_{0}^{1} 
     \Diff\bigl( \gelement{t,0}{\velocityfield} \bigr)\bigl( \gelement{0,t}{\velocityfield}(x) \bigr)  
     \Bigl( \velocityfieldother\bigl(t, \gelement{0,t}{\velocityfield}(x) \bigr) \Bigr)  
   \dint t.
\end{equation} 
\end{theorem}
\begin{proof}
The G\^ateaux derivative is the linear mapping
$\GDiff \DeforOpX(\velocityfield, \template) \colon \FlowSpace{2} \to  \RecSpace$
given as 
\begin{displaymath}
  \GDiff \DeforOpX(\velocityfield, \template)(\velocityfieldother)(x) := 
       \frac{d}{d \epsilon} \DeforOpX(\velocityfield+ \epsilon\velocityfieldother,\template)(x)\Bigl\vert_{\epsilon=0}
       = \frac{d}{d \epsilon} \bigl\vert\Diff \bigl(\gelement{1,0}{\velocityfield+ \epsilon\velocityfieldother}\bigr)(x)\bigr\vert \template \circ \gelement{1,0}{\velocityfield+ \epsilon\velocityfieldother}(x)\Bigl\vert_{\epsilon=0}
\end{displaymath}       
for $\velocityfieldother \in \FlowSpace{2}$ and $x \in \domain$. Now, by \cref{eq:PhiDeriv}
\begin{align*}
\frac{d}{d \epsilon} \gelement{1,0}{\velocityfield + \epsilon\velocityfieldother} \Bigl\vert_{\epsilon=0} (x)
  &= \int_{1}^{0} \Diff\bigl( \gelement{t,0}{\velocityfield} \bigr)\bigl( \gelement{1,t}{\velocityfield}(x) \bigr) 
             \Bigl( \velocityfieldother\bigl(t, \gelement{1,t}{\velocityfield}(x) \bigr) \Bigr) \dint t 
\\
&= -\int_{0}^{1} 
       \Diff\bigl( \gelement{t,0}{\velocityfield} \bigr)
       \bigl( \gelement{0,t}{\velocityfield} \circ\gelement{1,0}{\velocityfield}(x) \bigr) 
       \Bigl( \velocityfieldother\bigl(t, \gelement{0,t}{\velocityfield} \circ \gelement{1,0}{\velocityfield}(x) \bigr) \Bigr) 
     \dint t 
\\
&= \biggl[ -\int_{0}^{1} 
      \Diff\bigl( \gelement{t,0}{\velocityfield} \bigr)
      \bigl( \gelement{0,t}{\velocityfield}(\Cdot) \bigr)  
      \Bigl( \velocityfieldother\bigl(t, \gelement{0,t}{\velocityfield}(\Cdot) \bigr) \Bigr)  
    \dint t
    \biggr] \circ
    \gelement{1,0}{\velocityfield}(x) 
 = h(\Cdot) \circ\gelement{1,0}{\velocityfield}(x),
\end{align*}
where the last equation uses the definition of $h(\Cdot)$ by \cref{eq:h_simplify}.

Taking a Taylor expansion of $x \mapsto \gelement{1,0}{\velocityfield + \epsilon\velocityfieldother}(x)$ at $\epsilon=0$ gives  
\begin{equation}\label{eq:taylor_expansion_1}
  \gelement{1,0}{\velocityfield + \epsilon\velocityfieldother}(x)
  = \gelement{1,0}{\velocityfield}(x)  + \epsilon \, (h \circ\gelement{1,0}{\velocityfield})(x) + o(\epsilon)
\end{equation}
which in turn implies 
\begin{displaymath} 
\bigl\vert\Diff \bigl(\gelement{1,0}{\velocityfield+ \epsilon\velocityfieldother}\bigr)(x)\bigr\vert 
= \bigl\vert\Diff \bigl(\gelement{1,0}{\velocityfield}\bigr)(x)\bigr\vert 
   + \epsilon \bigl\vert\Diff \bigl(\gelement{1,0}{\velocityfield}\bigr)(x)\bigr\vert \bigl(\Div (h) \circ \gelement{1,0}{\velocityfield} \bigr)(x) 
   + o(\epsilon).
\end{displaymath} 
Taking the derivative \wrt $\epsilon$ and evaluating it at $\epsilon=0$ yields 
\begin{equation}\label{eq:taylor_expansion_3}
\frac{d}{d \epsilon}\bigl\vert\Diff \bigl(\gelement{1,0}{\velocityfield+ \epsilon\velocityfieldother}\bigr)(x)\bigr\vert \Bigl\vert_{\epsilon=0} 
= \bigl\vert\Diff \bigl(\gelement{1,0}{\velocityfield}\bigr)(x) \bigr\vert 
   \bigl(\Div (h) \circ \gelement{1,0}{\velocityfield} \bigr)(x).
\end{equation}

The chain rule now implies that 
\begin{equation}\label{eq:taylor_expansion4}
  \frac{d}{d \epsilon} \Bigl( \template \circ \gelement{1,0}{\velocityfield+ \epsilon\velocityfieldother}(x) \Bigr) \Bigl\vert_{\epsilon=0} = 
  \Bigl\langle 
    \grad \template\bigl(\gelement{1,0}{\velocityfield}(x)\bigr),
    ( h \circ\gelement{1,0}{\velocityfield})(x)
  \Bigr\rangle_{\Real^n}.
%= \grad \template\bigl(\gelement{1,0}{\velocityfield}(x)\bigr) ^Th \circ\gelement{1,0}{\velocityfield}(x).
\end{equation}
Hence
\begin{align*}
\GDiff \DeforOpX(\velocityfield, \template)(\velocityfieldother)(x) 
&= \bigl\vert\Diff \bigl(\gelement{1,0}{\velocityfield}\bigr)(x)\bigr\vert 
     \bigl( \Div (h) \circ \gelement{1,0}{\velocityfield}(\Cdot) \template(\Cdot) \circ \gelement{1,0}{\velocityfield} \bigr)(x) \\
&\quad + \bigl\vert\Diff \bigl(\gelement{1,0}{\velocityfield}\bigr)(x)\bigr\vert 
  \Bigl\langle 
    \grad \template\bigl(\gelement{1,0}{\velocityfield}(x)\bigr),
    ( h \circ\gelement{1,0}{\velocityfield})(x)
  \Bigr\rangle_{\Real^n}
\\
&= \bigl\vert\Diff \bigl(\gelement{1,0}{\velocityfield}\bigr)(x)\bigr\vert 
    \Div\bigl(\template(\Cdot) h(\Cdot)\bigr) \circ \gelement{1,0}{\velocityfield} (x).
\end{align*}
This proves \cref{eq:DiffDeforOpMP}.
\end{proof}

\subsection{Registration functionals in the abstract setting}\label{sec:RegFunc}
The aim is to compute the derivative and gradient of a \emph{registration functional}
$\MatchingFunctionalV_{\template} \colon \FlowSpace{2} \to \Real$ of the form in \cref{eq:OptimV}, \ie, 
\begin{equation}\label{eq:MatchFunc}
   \MatchingFunctionalV_{\template}(\velocityfield) := \MatchingFunctionalX  \circ \DeforOpX(\velocityfield, \template)
   \quad\text{for $\velocityfield \in \FlowSpace{2}$}  
\end{equation}  
where $\MatchingFunctionalX  \colon \RecSpace \to \Real$ is G\^ateaux differentiable, and $\DeforOpX$ is given by \cref{eq:DeforOpX}. 

By the chain rule one immediately obtains an expression for the G\^ateaux derivative of the registration functional $\MatchingFunctionalV_{\template}$ in \cref{eq:MatchFunc}:
\begin{equation}\label{eq:ChainRuleEq}
  \GDiff \MatchingFunctionalV_{\template}(\velocityfield)(\velocityfieldother) = 
      \GDiff \MatchingFunctionalX \bigl( \DeforOpX(\velocityfield, \template) \bigr) 
      \bigl( \GDiff \DeforOpX(\velocityfield, \template)(\velocityfieldother) \bigr)
    \quad\text{for $\velocityfield, \velocityfieldother \in \FlowSpace{2}$.}
\end{equation}

Furthermore, the expression for its $\FlowSpace{2}$--gradient is derived as
\begin{equation}\label{eq:MatchFuncGradientRel2}
  \bigl\langle \grad \MatchingFunctionalV_{\template}(\velocityfield), \velocityfieldother \bigr\rangle_{\FlowSpace{2}}
       = \Bigl\langle \grad \MatchingFunctionalX \bigl( \DeforOpX(\velocityfield, \template) \bigr), 
          \GDiff \DeforOpX(\velocityfield, \template)(\velocityfieldother) \Bigr\rangle_{\RecSpace}.
\end{equation}
In the above, $\GDiff \DeforOpX(\velocityfield, \template) \colon \FlowSpace{2} \to \RecSpace$ is a linear map that has a $\FlowSpace{2}$--adjoint 
$\GDiff \DeforOpX(\velocityfield, \template)^{\ast} \colon \RecSpace \to \FlowSpace{2}$, so 
\begin{equation}\label{eq:MatchFuncGradientRel3}
  \bigl\langle \grad \MatchingFunctionalV_{\template}(\velocityfield), \velocityfieldother \bigr\rangle_{\FlowSpace{2}}
       = \Bigl\langle \GDiff \DeforOpX(\velocityfield, \template)^{\ast}\bigl[ \grad \MatchingFunctionalX \bigl( \DeforOpX(\velocityfield, \template) \bigr], \velocityfieldother \Bigr\rangle_{\FlowSpace{2}}.
\end{equation}
Hence, the $\FlowSpace{2}$--gradient of $\MatchingFunctionalV_{\template}$ in \cref{eq:MatchFunc} is given by  
\begin{equation}\label{eq:MatchFuncGradientRel4}
  \grad \MatchingFunctionalV_{\template}(\velocityfield)
       = \GDiff \DeforOpX(\velocityfield, \template)^{\ast}\bigl[ \grad \MatchingFunctionalX \bigl( \DeforOpX(\velocityfield, \template) \bigr].
\end{equation}

More explicit expressions requires choosing a specific group action and Hilbert space structures on $\RecSpace$ and $\LieAlgebra$ (for gradient expressions). This is done in the following sections.

\subsubsection{Geometric group action}
The deformation operator is here given by the geometric group action in \cref{eq:DeforOpGeom}. Furthermore, consider the case when $\RecSpace$ has the $\LpSpace^2$--Hilbert space structure and $\LieAlgebra$ is a \ac{RKHS}. Under these conditions, it is possible to 
provide explicit expressions for the derivative and gradient of $\MatchingFunctionalV_{\template}$.
\begin{theorem}\label{thm:DiffMatchFuncGeom}
Let the registration functional $\MatchingFunctionalV_{\template} \colon \FlowSpace{2} \to \Real$ be given as in \cref{eq:MatchFunc} where $\template \in \RecSpace$ is differentiable and $\MatchingFunctionalX  \colon \RecSpace \to \Real$ is G\^ateaux differentiable on $\RecSpace$.
Furthermore, assume $\DiffeoGroup$ acts on $\RecSpace$ by means of the (left) geometric group action \cref{eq:Geometric}. Then, the G\^ateaux derivative of $\MatchingFunctionalV_{\template}$ is given as \cref{eq:ChainRuleEq}, where $\GDiff \DeforOpX(\velocityfield, \template)(\velocityfieldother) \colon \domain \to \Real$ is given as \cref{eq:DiffDeforOpGeom}. 

Furthermore, if $\RecSpace \subset \LpSpace^2(\domain, \Real)$ and $\LieAlgebra$ is a \ac{RKHS} with a reproducing kernel represented by a symmetric and positive definite function $\kernel \colon \domain \times \domain \to \Matrix_{+}^{n \times n}$ (see \cref{sec:RKHS}), then 
the $\FlowSpace{2}$--gradient is 
\begin{equation}\label{eq:GradMatchingFunctionalVRHKS}
  \grad \MatchingFunctionalV_{\template}(\velocityfield)(t,x) = 
      -  \int_{\domain} 
              \bigl\vert \Diff(\gelement{t,1}{\velocityfield})(y) \bigr\vert \grad\MatchingFunctionalX \bigl( \DeforOpX(\velocityfield, \template) \bigr)\bigl( \gelement{t,1}{\velocityfield}(y) \bigr)
              \kernel(x,y) \cdot 
              \grad(\template \circ \gelement{t,0}{\velocityfield})(y) 
              \dint y
\end{equation}    
for $x \in \domain$ and $0\leq t \leq 1$.
\end{theorem}
\begin{proof}
The deformation operator $\DeforOpX(\Cdot,\template)$ is given by \cref{eq:DeforOpGeom} for geometric group action and its derivative is given by \cref{eq:DiffDeforOpGeom} in \cref{thm:DiffDeforOpGeom}. The gives the statement. To derive the expression in \cref{eq:GradMatchingFunctionalVRHKS} for the gradient, one needs to use the assumption that $\RecSpace \subset \LpSpace^2(\domain, \Real)$. An expression for $\GDiff \DeforOpX(\velocityfield, \template)(\velocityfieldother)$ is given by \cref{eq:DiffDeforOpGeom}, which then can be inserted into \cref{eq:MatchFuncGradientRel2}:
\begin{align*}
\GDiff \MatchingFunctionalV_{\template}(\velocityfield)(\velocityfieldother)  
  &= \biggl\langle 
    \grad\MatchingFunctionalX \bigl( \DeforOpX(\velocityfield, \template) \bigr)(\Cdot), 
   - \int_{0}^{1} 
      \Bigl\langle 
         \grad (\template \circ \gelement{t,0}{\velocityfield})\bigl( \gelement{1,t}{\velocityfield}(\Cdot) \bigr),
         \velocityfieldother\bigl(t,\gelement{1,t}{\velocityfield}(\Cdot)\bigr)  
      \Bigr\rangle_{\Real^n}\dint t
    \biggr\rangle_{\!\!\LpSpace^2(\domain,\Real)}     
\\
  &= - \int_{0}^1 \biggl\langle 
      \grad\MatchingFunctionalX \bigl( \DeforOpX(\velocityfield, \template) \bigr)(\Cdot) 
        \grad (\template \circ \gelement{t,0}{\velocityfield})\bigl( \gelement{1,t}{\velocityfield}(\Cdot) \bigr),  
      \velocityfieldother\bigl(t,\gelement{1,t}{\velocityfield}(\Cdot)\bigr)  
    \biggr\rangle_{\!\!\LpSpace^2(\domain,\Real^n)}\dint t.
\end{align*}   
The last equality makes use of the fact that the inner product in $\LpSpace^2(\domain,\Real^n)$ (square 
integrable $\Real^n$-valued functions) is expressible as the inner product in $\Real^n$ followed by the 
inner product in $\LpSpace^2(\domain,\Real)$ (square integrable real valued functions).  Note also
that
\begin{align*}
  x &\mapsto 
  \grad\MatchingFunctionalX \bigl( \DeforOpX(\velocityfield, \template) \bigr)(x) 
        \grad (\template \circ \gelement{t,0}{\velocityfield})\bigl( \gelement{1,t}{\velocityfield}(x) \bigr),  \\
  x &\mapsto  \velocityfieldother\bigl(t,\gelement{1,t}{\velocityfield}(x)\bigr)
\end{align*}
are both mappings from $\domain$ into $\Real^n$, so the integrand in the last expression, which is given 
by the $\LpSpace^2(\domain,\Real^n)$ inner product, is well defined. 
 
Now, introduce the variable $y := \gelement{1,t}{\velocityfield}(x)$, so $x = \gelement{t,1}{\velocityfield}(y)$ by \cref{eq:phiRelations}, 
and use the fact that the inner product is symmetric. This gives
\begin{displaymath}
\GDiff \MatchingFunctionalV_{\template}(\velocityfield)(\velocityfieldother) = 
  - \int_{0}^1 \Bigl\langle \velocityfieldother(t,\Cdot), 
    \bigl\vert \Diff(\gelement{t,1}{\velocityfield})(\Cdot) \bigr\vert 
    \grad\MatchingFunctionalX \bigl( \DeforOpX(\velocityfield, \template) \bigr)\bigl( \gelement{t,1}{\velocityfield}(\Cdot) \bigr) 
    \grad(\template \circ \gelement{t,0}{\velocityfield})(\Cdot) 
    \Bigr\rangle_{\LpSpace^2(\domain,\Real^n)} \dint t.   
\end{displaymath}
Finally, $\LieAlgebra$ is a \ac{RKHS} with reproducing kernel $\Koperator \colon \domain \times \domain \to \LinearSpace(\Real^n,\Real^n)$, so \cref{eq:VIP} implies that 
\begin{displaymath}  
\GDiff \MatchingFunctionalV_{\template}(\velocityfield)(\velocityfieldother) 
   = - \int_0^1 \biggl\langle \velocityfieldother(t,\Cdot), \int_{\domain} \Koperator(\Cdot, x)\bigl( \tilde{\velocityfield}(t,x) 
   \bigr)\dint x \biggr\rangle_{\!\!\LieAlgebra} \dint t
\end{displaymath}
where $\tilde{\velocityfield}(t, \Cdot) \colon \domain \to \Real^n$ is defined as 
\begin{displaymath}   
\tilde{\velocityfield}(t,x) := 
   \bigl\vert \Diff(\gelement{t,1}{\velocityfield})(x) \bigr\vert 
    \grad\MatchingFunctionalX \bigl( \DeforOpX(\velocityfield, \template) \bigr)\bigl( \gelement{t,1}{\velocityfield}(x) \bigr) 
    \grad(\template \circ \gelement{t,0}{\velocityfield})(x).
 \end{displaymath}
It is now possible to read off the expression for the $\FlowSpace{2}$--gradient of $\MatchingFunctionalV_{\template}$:
\begin{displaymath}  
    \grad \MatchingFunctionalV_{\template}(\velocityfield)(t,x) 
      = -  \int_{\domain} \Koperator(x,y)\bigl( \tilde{\velocityfield}(t,y) \bigr) \dint y,
\end{displaymath}
and inserting the matrix valued function $\kernel \colon \domain \times \domain \to \Matrix_{+}^{n \times n}$ 
representing the reproducing kernel yields \cref{eq:GradMatchingFunctionalVRHKS}. This concludes the proof of \cref{thm:DiffMatchFuncGeom}.
\end{proof}

\subsubsection{Mass-preserving group action}
The deformation operator is here given by the mass-preserving group action in \cref{eq:DeforOpMP}. Furthermore, consider the case when 
$\RecSpace$ has the $\LpSpace^2$--Hilbert space structure and $\LieAlgebra$ is a \ac{RKHS}. Under these conditions, it is possible to 
provide explicit expressions for the derivative and gradient of $\MatchingFunctionalV_{\template}$.
\begin{theorem}\label{thm:DiffMatchFuncMP}
Let the assumptions in \cref{thm:DiffMatchFuncGeom} hold with the (left) mass-preserving group action \cref{eq:MassPreserving} instead of the geometric one. Then, the G\^ateaux derivative of $\MatchingFunctionalV_{\template}$ is given as in \cref{eq:ChainRuleEq} with $\GDiff \DeforOpX(\velocityfield, \template)(\velocityfieldother) \colon \domain \to \Real$ is given as \cref{eq:DiffDeforOpMP}.
Furthermore, the corresponding $\FlowSpace{2}$--gradient is  
  \begin{equation}\label{eq:GradMatchingFunctionalVRHKS_mp}
    \grad \MatchingFunctionalV_{\template}(\velocityfield)(t,x) = 
      \int_{\domain}             
           \bigl\vert\Diff(\gelement{t,0}{\velocityfield})(y)\bigr\vert (\template \circ \gelement{t,0}{\velocityfield})(y) \kernel(x,y) 
           \cdot
           \grad\Bigl( \grad\MatchingFunctionalX \bigl( \DeforOpX(\velocityfield, \template) \bigr)\circ \gelement{t,1}{\velocityfield}\Bigr)(y)
            \dint y
  \end{equation}    
  for $x \in \domain$ and $0\leq t \leq 1$.
\end{theorem}
\begin{proof}
Under the mass-preserving group action, the deformation operator $\DeforOpX(\Cdot,\template)$ is given by \cref{eq:DeforOpMP} and its derivative is calculated by \cref{eq:DiffDeforOpMP} in \cref{thm:DiffDeforOpMP}. By chain rule, the first statement is immediately proved.  

To prove \cref{eq:GradMatchingFunctionalVRHKS_mp}, start by inserting the expression for $\GDiff \DeforOpX(\velocityfield, \template)(\velocityfieldother)$ into \cref{eq:MatchFuncGradientRel2}:
\begin{align*}
  \GDiff \MatchingFunctionalV_{\template}(\velocityfield)(\velocityfieldother)
     &= \int_{0}^{1}\biggl\langle \grad \MatchingFunctionalX \bigl( \DeforOpX(\velocityfield, \template) \bigr), 
 - \bigl\vert\Diff \bigl(\gelement{1,0}{\velocityfield}\bigr)\bigr\vert \Div \biggl(\Diff\bigl( \gelement{t,0}{\velocityfield} \bigr)\bigl( \gelement{0,t}{\velocityfield} \bigr)  \Bigl( \velocityfieldother\bigl(t, \gelement{0,t}{\velocityfield} \bigr) \Bigr) \template\biggr) \circ \gelement{1,0}{\velocityfield}\biggr\rangle_{\!\!\LpSpace^2(\domain,\Real)}\!\!\! \dint t \\
%   & \quad    = \int_{0}^{1}\biggl\langle \grad \MatchingFunctionalX \bigl( \DeforOpX(\velocityfield, \template) \bigr) \circ \gelement{0,1}{\velocityfield},  -  \Div \biggl(\Diff\bigl( \gelement{t,0}{\velocityfield} \bigr)\bigl( \gelement{0,t}{\velocityfield} \bigr)  \Bigl( \velocityfieldother\bigl(t, \gelement{0,t}{\velocityfield} \bigr) \Bigr) \template\biggr) \biggr\rangle_{\LpSpace^2(\domain,\Real)}\,\,\,\dint t \\
%   & \quad    = \int_{0}^{1}\biggl\langle \template \grad\Bigl(\grad \MatchingFunctionalX \bigl( \DeforOpX(\velocityfield, \template) \bigr) \circ \gelement{0,1}{\velocityfield}\Bigr),  \Diff\bigl( \gelement{t,0}{\velocityfield} \bigr)\bigl( \gelement{0,t}{\velocityfield} \bigr)  \Bigl( \velocityfieldother\bigl(t, \gelement{0,t}{\velocityfield} \bigr) \Bigr)  \biggr\rangle_{\LpSpace^2(\domain,\Real)}\,\,\,\dint t \\
%   & \quad    = \int_{0}^{1}\biggl\langle  \bigl\vert\Diff\bigl(\gelement{t,0}{\velocityfield}\bigr)\bigr\vert \template \circ \gelement{t,0}{\velocityfield} \grad\Bigl(\grad \MatchingFunctionalX \bigl( \DeforOpX(\velocityfield, \template) \bigr) \circ \gelement{0,1}{\velocityfield}\Bigr) \circ \gelement{t,0}{\velocityfield}, \Diff\bigl( \gelement{t,0}{\velocityfield} \bigr)  \Bigl( \velocityfieldother\bigl(t, \Cdot \bigr) \Bigr)  \biggr\rangle_{\LpSpace^2(\domain,\Real)}\,\,\,\dint t \\
   &= \biggl\langle \int_{0}^{1} \bigl\vert\Diff\bigl(\gelement{t,0}{\velocityfield}\bigr)\bigr\vert \template \circ \gelement{t,0}{\velocityfield} \grad\Bigl(\grad \MatchingFunctionalX \bigl( \DeforOpX(\velocityfield, \template) \bigr)\circ \gelement{t,1}{\velocityfield}\Bigr) \dint t , \velocityfieldother\bigl(t, \Cdot \bigr) \biggr\rangle_{\!\!\LpSpace^2(\domain,\Real^n)}.
\end{align*}
Now, \cref{eq:GradMatchingFunctionalVRHKS_mp} follows from combining the above with \cref{eq:VIP} and reading off the gradient term.
\end{proof}

\subsection{$\LpSpace^{2}$-based registration functionals}\label{sec:L2regFunc}
Consider \cref{eq:InvProb} where $\RecSpace \subset \LpSpace^{2}(\domain,\Real)$ and 
$\DataSpace \subset \LpSpace^{2}(\datadomain,\Real)$ are Hilbert spaces 
endowed with the $\LpSpace^{2}$--Hilbert space structure. Here, $\datadomain$ denotes 
a smooth manifold providing coordinates for data. Finally, assume that $\ForwardOp$ is
G\^ateaux differentiable and $\noisedata$ is independent of $\ForwardOp(\truesignal)$ 
and Gaussian. Given these assumptions, the natural registration functional is 
\begin{equation}\label{eq:MatchingFuncL2Indirect}
 \MatchingFunctionalX (\signal) := \bigl\Vert \ForwardOp(\signal) - \data \bigr\Vert_{\DataSpace}^2
   = \int_{\datadomain} \bigl( \ForwardOp(\signal)(y) - \data(y) \bigr)^2 \dint y
 \quad\text{for $\signal \in \RecSpace$,}
\end{equation}
for given data $\data\in \DataSpace$. The corresponding registration functional $\MatchingFunctionalV_{\template} \colon \FlowSpace{2} \to \Real$ is then 
\begin{equation}\label{eq:EnergyImageMatchingIndirect}
  \MatchingFunctionalV_{\template}(\velocityfield) := 
       \Bigl\Vert 
         \ForwardOp\bigl( \DeforOpX(\velocityfield, \template) \bigr) - \data 
       \Bigr\Vert_{\LpSpace^{2}(\datadomain,\Real)}^2.
\end{equation}
Note also that the G\^ateaux derivative of $\MatchingFunctionalX $ and its corresponding gradient are given by  
\begin{equation}\label{eq:Diff:MatchingFuncL2Indirect}
\begin{split}
 \GDiff \MatchingFunctionalX (\signal)(\signalother) 
   &= 2 \Bigl\langle 
          \GDiff \ForwardOp(\signal)^{\ast}\bigl( \ForwardOp(\signal)-\data \bigr),
    \signalother
  \Bigr\rangle_{\RecSpace} \\
 \grad \MatchingFunctionalX (\signal) 
   &= 2 \GDiff \ForwardOp(\signal)^{\ast}\bigl( \ForwardOp(\signal)-\data \bigr)
\end{split}
\qquad \text{for $\signal,\signalother \in \RecSpace$.} 
\end{equation}
Here, ``$\ast$'' denotes the Hilbert space adjoint and $\GDiff \ForwardOp$ is the G\^ateaux derivative of 
$\ForwardOp$. When $\ForwardOp$ is linear, then $\GDiff \ForwardOp(\signal)=\ForwardOp$
in \cref{eq:Diff:MatchingFuncL2Indirect}.
Next, compute the G\^ateaux derivative of $\MatchingFunctionalV_{\template}$ and its corresponding gradient under the two group actions.

\subsubsection{Geometric group action}
Here, the registration functional is given by \cref{eq:EnergyImageMatchingIndirect} with the deformation operator given as in 
\cref{eq:DeforOpGeom}, \ie, 
\begin{equation}\label{eq:L2IndirectGeom}
  \MatchingFunctionalV_{\template}(\velocityfield) 
       = \int_{\datadomain} \Bigr( \ForwardOp\bigl(\template \circ \gelement{1,0}{\velocityfield}\bigr)(x) - \data(x) \Bigr)^2\dint x.
\end{equation}
\begin{corollary}\label{cor:DiffL2IndirectGeom}
Let the assumptions in \cref{thm:DiffMatchFuncGeom} hold and let
$\MatchingFunctionalV_{\template} \colon \FlowSpace{2} \to \Real$ be given as in \cref{eq:L2IndirectGeom}
with $\ForwardOp \colon \RecSpace \to \DataSpace$ G\^ateaux differentiable.
Then
\begin{equation}\label{eq:DiffL2IndirectGeom}
    \GDiff \MatchingFunctionalV_{\template}(\velocityfield)(\velocityfieldother)  
       = 2 \Bigl\langle  
             \GDiff\ForwardOp\bigl( \DeforOpX(\velocityfield, \template) \bigr)^{\ast}\bigl( \ForwardOp\bigl( \DeforOpX(\velocityfield, \template) - \data \bigr),
             \GDiff \DeforOpX(\velocityfield, \template)(\velocityfieldother)
           \Bigr\rangle_{\LpSpace^2(\domain,\Real)}
\quad\text{for $\velocityfield, \velocityfieldother \in \FlowSpace{2}$,}
\end{equation}  
 where $\GDiff \DeforOpX(\velocityfield, \template)(\velocityfieldother) \colon \domain \to \Real$ is given as \cref{eq:DiffDeforOpGeom}. 

Furthermore, if $\RecSpace \subset \LpSpace^2(\domain, \Real)$ and 
$\LieAlgebra$ is a \ac{RKHS} with a reproducing kernel represented by the 
symmetric and positive definite function $\kernel \colon \domain \times \domain \to \Matrix_{+}^{n \times n}$
(see \cref{sec:RKHS}), then the corresponding $\FlowSpace{2}$--gradient is  
\begin{multline}\label{eq:GradMatchingFunctionalVImageMatchingIndirectRHKS} 
    \grad \MatchingFunctionalV_{\template}(\velocityfield)(t,x) \\ 
    = -  \int_{\domain} \bigl\vert \Diff(\gelement{t,1}{\velocityfield})(y) \bigr\vert 
              \GDiff\ForwardOp\bigl( \DeforOpX(\velocityfield, \template) \bigr)^{\ast}\bigl( \ForwardOp\bigl( \DeforOpX(\velocityfield, \template) - \data \bigr)\bigl( \gelement{t,1}{\velocityfield}(y) \bigr) 
         \kernel(x,y) \cdot 
              \grad(\template \circ \gelement{t,0}{\velocityfield})(y) 
             \dint y  
\end{multline}
for $x \in \domain$ and $0 \leq t \leq 1$.    
\end{corollary}
\begin{proof}
The proof of \cref{eq:DiffL2IndirectGeom} follows directly from inserting the expression for 
$\GDiff \MatchingFunctionalX $ in \cref{eq:Diff:MatchingFuncL2Indirect} into \cref{eq:ChainRuleEq}. Similarly, 
\cref{eq:GradMatchingFunctionalVImageMatchingIndirectRHKS} follows directly from inserting the expression for 
$\grad \MatchingFunctionalX $ in \cref{eq:Diff:MatchingFuncL2Indirect} into \cref{eq:GradMatchingFunctionalVRHKS}.
\end{proof}

\subsubsection{Mass-preserving group action}
The registration functional is here given by \cref{eq:EnergyImageMatchingIndirect} with the deformation operator given as in 
\cref{eq:DeforOpMP}, \ie, 
\begin{equation}\label{eq:L2IndirectMP}
  \MatchingFunctionalV_{\template}(\velocityfield) 
       =  \int_{\datadomain} \Bigr( \ForwardOp\Bigl(\bigl\vert \Diff (\gelement{1,0}{\velocityfield}) \bigr\vert\,  \template \circ \gelement{1,0}{\velocityfield}\Bigr)(x) - \data(x) \Bigr)^2\dint x.
\end{equation}

\begin{corollary}\label{cor:DiffL2IndirectMP}
Let the assumptions in \cref{cor:DiffL2IndirectGeom} hold but with mass-preserving group action instead of the geometric one.  Then, $\velocityfieldother \mapsto \GDiff \MatchingFunctionalV_{\template}(\velocityfield)(\velocityfieldother)$ is given as in \cref{eq:DiffL2IndirectGeom} but with $\GDiff \DeforOpX(\velocityfield, \template)(\velocityfieldother) \colon \domain \to \Real$ is given as \cref{eq:DiffDeforOpMP}.
Furthermore, the corresponding $\FlowSpace{2}$--gradient is  
\begin{multline}\label{eq:GradMatchingFunctionalVImageMatchingIndirectRHKS_mp}
    \grad \MatchingFunctionalV_{\template}(\velocityfield)(t,x) 
       = 2\int_{\domain} 
           \Bigl\vert\Diff\bigl(\gelement{t,0}{\velocityfield}\bigr)(y)\Bigr\vert \, 
           (\template \circ \gelement{t,0}{\velocityfield})(y) \kernel(x,y) 
           \cdot \\
           \grad\Bigl( \Bigl[\GDiff\ForwardOp\bigl( \DeforOpX(\velocityfield, \template) \bigr)^{\ast}\Bigl( \ForwardOp\bigl( \DeforOpX(\velocityfield, \template) \bigr)- \data \Bigr) \Bigr]\circ \gelement{t,1}{\velocityfield}(y)\Bigr) \dint y    
\end{multline}
for $x \in \domain$ and $0 \leq t \leq 1$.    
\end{corollary}
The proof is analogous to the proof of \cref{cor:DiffL2IndirectGeom}.

\subsection{The shape functional}\label{sec:ShapeFuncGrad}
Define $\Lambda \colon \FlowSpace{2} \to \Real_+$ as 
$\Lambda(\velocityfield) := \gamma \, \LieGroupMetric(\Id,\gelement{0,1}{\velocityfield})^2$.
Then, \cref{eq:ExactShapeFunc,eq:ShapeRegOnOrbit,eq:DeforOpX}, yields 
\begin{equation}\label{eq:DefMetricOp}  
  \Lambda(\velocityfield) = \gamma \LieGroupMetric(\Id,\gelement{0,1}{\velocityfield})^2
  = \gamma \Vert \velocityfield \Vert_{\FlowSpace{2}}^2 \quad\text{for $\velocityfield \in \FlowSpace{2}$.} 
\end{equation}  
The G\^ateaux derivative and associated $\FlowSpace{2}$--gradient of $\Lambda$ is then
\begin{equation}\label{eq:VNormGrad}
\begin{split}
  & \GDiff \Lambda(\velocityfield)(\velocityfieldother) = 2 \gamma \langle \velocityfield, \velocityfieldother \rangle_{\FlowSpace{2}}
   \\  
  &  \grad \Lambda(\velocityfield) = 2 \gamma \velocityfield 
\end{split}    
    \qquad\text{for $\velocityfield, \velocityfieldother \in \FlowSpace{2}$.}
\end{equation}

\subsection{Gradient of the objective functional}\label{sec:GradObjFunc}
The goal here is to provide expressions for the gradient of objective functional in \cref{eq:GenericGoalFunctionalV} for the two group actions, geometric and mass-preserving.

\subsubsection{Geometric group action}
The following generalization of \cite[Theorem~16.2]{GrMi07} provides an explicit expression for the gradient of the objective functional in \cref{eq:GenericGoalFunctionalV}.
\begin{corollary}\label{corr:GoalFunctionalGradient}
Let the assumptions in \cref{thm:DiffMatchFuncGeom} hold. 
Then, the $\FlowSpace{2}$--gradient of $\velocityfield \mapsto \ObjectiveV (\velocityfield)$ in \cref{eq:GenericGoalFunctionalV} is
\begin{equation}\label{eq:GradGoalFunctionalVRHKS}
    \grad \ObjectiveV (\velocityfield)(t,x) 
    =  2 \gamma \velocityfield(t,x)  
      -  \int_{\domain} \bigl\vert \Diff(\gelement{t,1}{\velocityfield})(y) \bigr\vert
              \grad\MatchingFunctionalX \bigl( \DeforOpX(\velocityfield, \template) \bigr)\bigl( \gelement{t,1}{\velocityfield}(y) \bigr)\kernel(x,y) \cdot 
              \grad(\template \circ \gelement{t,0}{\velocityfield})(y) 
              \dint y
\end{equation}
for $x \in \domain$ and $0\leq t \leq 1$.
\end{corollary}
This is an immediate consequence of \cref{thm:DiffMatchFuncGeom} and \cref{eq:VNormGrad}.
%\begin{proof}
%The proof of this corollary is quite straightforward, which is omitted here.
%We know that 
%$\ObjectiveV  = \Lambda + \MatchingFunctionalV_{\template}$ with $\Lambda$ given as in \cref{eq:DefMetricOp} and 
%$\MatchingFunctionalV_{\template}$ given as in \cref{eq:MatchFunc}, so by linearity 
%\begin{displaymath}
%  \grad \ObjectiveV (\velocityfield) =  \grad \Lambda(\velocityfield) + \grad \MatchingFunctionalV_{\template}(\velocityfield)
%       \quad\text{for $\velocityfield \in \FlowSpace{2}$.}
%\end{displaymath}
%$\LieAlgebra$ is, by assumption, a \ac{RKHS} so we can use \cref{eq:GradMatchingFunctionalVRHKS} in 
%\cref{thm:DiffMatchFuncGeom} to obtain an explicit expression for the gradient 
%$\grad \MatchingFunctionalV_{\template} \colon \FlowSpace{2} \to \FlowSpace{2}$. Likewise, \cref{eq:VNormGrad}
%gives an explicit expression for the gradient 
%$\grad \Lambda \colon \FlowSpace{2} \to \FlowSpace{2}$. Inserting these expressions into the above 
%equality yields the equality in \cref{eq:GradGoalFunctionalVRHKS}.
%\end{proof}

\subsubsection{Mass-preserving group action}
The following result is the version of \cref{corr:GoalFunctionalGradient} under the mass-preserving group action where the deformation operator is given as in \cref{eq:DeforOpMP}.
\begin{corollary}\label{corr:GoalFunctionalGradient_mp}
Let the assumptions in \cref{thm:DiffMatchFuncMP} hold. 
Then, the $\FlowSpace{2}$--gradient of $\velocityfield \mapsto \ObjectiveV (\velocityfield)$ in \cref{eq:GenericGoalFunctionalV} is
\begin{equation}\label{eq:GradGoalFunctionalVRHKS_mp}
    \grad \ObjectiveV (\velocityfield)(t,x) 
    =  2 \gamma \velocityfield(t,x)  
   + \int_{\domain} 
           \bigl\vert\Diff\bigl(\gelement{t,0}{\velocityfield}\bigr)(y)\bigr\vert \, 
           (\template \circ \gelement{t,0}{\velocityfield})(y) \kernel(x,y) 
           \cdot
           \grad\Bigl( \grad\MatchingFunctionalX \bigl( \DeforOpX(\velocityfield, \template) \bigr)\circ \gelement{t,1}{\velocityfield}\Bigr)(y) \dint y
\end{equation}
for $x \in \domain$ and $0\leq t \leq 1$.
\end{corollary}
This is an immediate consequence of \cref{thm:DiffMatchFuncMP} and \cref{eq:VNormGrad}.
%\begin{proof}
%The $\FlowSpace{2}$--gradient of $\velocityfield \mapsto \ObjectiveV (\velocityfield)$ in \cref{eq:GenericGoalFunctionalV} is given as
%\begin{displaymath}
%\grad \ObjectiveV (\velocityfield)(t,x)  =  2 \gamma \velocityfield(t,x) + \grad \MatchingFunctionalV_{\template}(\velocityfield)(t,x)
%\end{displaymath}
%for $\velocityfield \in \FlowSpace{2}$ and $x \in \domain$, $0\leq t \leq 1$, where $\grad \MatchingFunctionalV_{\template}(\velocityfield)$ is the $\FlowSpace{2}$--gradient of $\velocityfield \mapsto \MatchingFunctionalV_{\template}(\velocityfield)$.
%The claim in \cref{eq:GradGoalFunctionalVRHKS_mp} now follows directly from \cref{thm:DiffMatchFuncMP}.
%\end{proof}

\section{Numerical implementation}\label{sec:NumericalImpl}
The focus here is to describe the numerical methods used for solving \cref{eq:OptimV}. The corresponding  implementation is available used in this paper is available from \href{https://github.com/chongchenmath/odl_lddmm}{https://github.com/chongchenmath/odl\_lddmm} and it makes use of the Operator Discretization Library (\href{http://github.com/odlgroup/odl}{http://github.com/odlgroup/odl}).

\subsection{Optimization strategies}\label{sec:OptimizationStrategy}
\Cref{sec:Gradient} gives explicit expressions for the $\FlowSpace{2}$--gradient of the objective functional in \cref{eq:OptimV}. Hence, one can use any optimization algorithm that makes use of gradient information, like gradient descent. The numerical implementation of such a gradient descent scheme is outlined in \cref{subsec:GradientDescent}.

An alternative approach is to consider the gradient in the Hamiltonian form, which directly incorporates the dimension reduction that arises from the projection onto a finite dimensional subset of the \ac{RKHS} $\LieAlgebra$, see \cite[section~11.6.2]{Yo10} (in particular algorithm~4 on \cite[p.~275]{Yo10}) for further details. Yet another approach is the shooting method outlined in \cite[section~11.6.4]{Yo10}. Here, one makes use of the characterization of a minimizer based on the momentum geodesic equations. The numerical implementation can make use of the natural finite dimensional counterpart obtained by projecting onto a finite dimensional subset of the \ac{RKHS} $\LieAlgebra$. 

Even though the above approaches may require less iterates than gradient descent, they have an important drawback. They rely on discretizing the vector fields in the \ac{RKHS} $\LieAlgebra$ by control points that do not remain stationary over iterates, so the control points do not necessarily remain on a fixed regular grid. Thus, evaluating the velocity field at some point in time by convolving against a time-dependent kernel function in the \ac{RKHS} cannot easily make use of efficient \ac{FFT} based schemes. This becomes in particular troublesome when the shape space is scalar valued functions defined in 2D/3D. In contrast, the  gradient descent scheme can be formulated so that the \ac{RKHS} kernel used to evaluate the vector fields is given by a fixed set of control points. For this reason, the gradient descent scheme competes favourably regarding computational complexity against the shooting method.

\subsection{Gradient descent}\label{subsec:GradientDescent}
The gradient descent scheme for solving \cref{eq:OptimV} is a first-order iterative optimization algorithm given as 
\begin{equation}\label{eq:gradientflow}
  \velocityfield^{k+1} = \velocityfield^k -  \stepsize\grad \ObjectiveV (\velocityfield^k).
\end{equation}
Here, $\stepsize>0$ is the step-size and $\grad \ObjectiveV $ is the $\FlowSpace{2}$--gradient of the objective functional $\ObjectiveV $ in \cref{eq:GenericGoalFunctionalV}.

\Cref{corr:GoalFunctionalGradient} gives an expression for $\grad \ObjectiveV $ in \cref{eq:GradGoalFunctionalVRHKS} where the reproducing kernel function $\kernel \colon \domain \times \domain \to \Matrix_{+}^{n \times n}$ is evaluated on points that do not move as iterates proceed. Choosing a translation invariant kernel and points on a regular grid in $\domain$, allows one to use computationally efficient \acs{FFT}-based schemes for computing the velocity field at each iterate (see \cite[section~11.7.2]{Yo10}). This is necessary for numerically evaluating $\grad \ObjectiveV (x)$ for $x \in \domain$ in the aforementioned grid and it is computationally more feasible than letting the kernel depend on points that move in time as in the shooting method. 

\paragraph{Computing diffeomorphic deformations}
Evaluating $\velocityfield \mapsto \grad \ObjectiveV (\velocityfield)$ in \cref{eq:gradientflow}  requires computing the diffeomorphic deformations $\gelement{t,0}{\velocityfield}$ and $\gelement{t,1}{\velocityfield}$. Since this is done repeatedly within an iterative scheme, one needs to have an efficient implementation for computing these diffeomorphisms. 

First recall that  $\gelement{s,t}{\velocityfield}$ solves the \ac{ODE} in \cref{eq:FlowEq} forward in time:
\begin{equation}\label{eq:BasicODEInv_0_2}
\begin{cases}
  \partial_t \diffeoother(t,x) = \velocityfield\bigl(t,\diffeoother(t,x)\bigr) & \\
  \diffeoother(s,x) = x & 
 \end{cases}
 \quad\text{for $x\in \domain$ and $0 \leq s, t \leq 1$} 
\end{equation}
with $s$ a fixed time point.
%Next, recall that $\gelement{1,t}{\velocityfield} = (\gelement{t,1}{\velocityfield} )^{-1}$ and $\gelement{1,t}{\velocityfield} = \gelement{0,t}{\velocityfield} \circ \gelement{1,0}{\velocityfield}$, so it can be written as the solution of the 
%\ac{ODE} in \cref{eq:BasicODEInv_0_3} backwards in time:
%\begin{equation}\label{eq:BasicODEInv_0_3}
%\begin{cases}
%  \partial_t \varphi(t,x) = \velocityfield\bigl(t,\varphi(t,x)\bigr) & \\
%  \varphi(1,x) = x & 
% \end{cases}
% \quad\text{for $x\in \domain$ and $0 \leq t \leq 1$.} 
%\end{equation}
Integrating \cref{eq:BasicODEInv_0_2} \wrt time $t$ gives
\begin{equation}
  \gelement{s,t}{\velocityfield}(x) = x + \int_s^t \velocityfield\bigl(\tau,\gelement{s,\tau}{\velocityfield}(x)\bigr)\dint \tau
  \label{eq:forwarddeformation} 
  %\\
%  \gelement{1,t}{\velocityfield}(x) &= x - \int_t^1 \velocityfield\bigl(\tau,\gelement{1,\tau}{\velocityfield}(x\bigr)\dint \tau 
%  \label{eq:backwardsdeformation}
\quad\text{for $0 \leq t \leq 1$ and $x \in \domain$.}
\end{equation}

A numerical implementation needs to discretize time, which can be done by sub-dividing the time interval $[0,1]$ uniformly into $N$ parts, \ie, 
$t_i = i/N$ for $i = 0, 1, \ldots, N$. Let $s=t_i$, $t=t_{i+1}$ and $t_{i-1}$, the expressions for small deformations $\gelement{t_i,t_{i+1}}{\velocityfield}$ and $\gelement{t_i,t_{i-1}}{\velocityfield}$ would be derived as the approximations
\begin{align}
\label{eq:forwarddeformation_small1}  \gelement{t_i,t_{i+1}}{\velocityfield} &\approx \Id + \frac{1}{N}\velocityfield(t_i, \cdot),\\
\label{eq:forwarddeformation_small2}  \gelement{t_i,t_{i-1}}{\velocityfield} &\approx \Id - \frac{1}{N}\velocityfield(t_i, \cdot).
\end{align}

Furthermore, \cref{eq:phiRelations} implies that $\gelement{t_i,0}{\velocityfield} = \gelement{t_{i-1},0}{\velocityfield} \circ \gelement{t_i,t_{i-1}}{\velocityfield}$, 
which is combined with \cref{eq:forwarddeformation_small2} yields the following approximation:
\begin{equation}\label{eq:ODEflow1b}
\gelement{t_i,0}{\velocityfield} 
  \approx  \gelement{t_{i-1},0}{\velocityfield} \circ \Bigl(\Id - \frac{1}{N}\velocityfield(t_i,\Cdot)\Bigr) 
  \quad\text{for $i = 1, \ldots, N$.}
\end{equation}
Similarly, \cref{eq:phiRelations} implies $\gelement{t_i,1}{\velocityfield} = \gelement{t_{i+1},1}{\velocityfield} \circ \gelement{t_i,t_{i+1}}{\velocityfield}$, 
which is combined with \cref{eq:forwarddeformation_small1} gives the following approximation:
\begin{equation}\label{eq:ODEflow2b}
\gelement{t_i,1}{\velocityfield} \approx  \gelement{t_{i+1},1}{\velocityfield} \circ \Bigl(\Id  + \frac{1}{N} \velocityfield(t_i,\Cdot)\Bigr) \quad\text{for $i = N-1, \ldots, 0$.}
\end{equation}

\begin{remark}\label{re:inverse_linear_deformation}
The approximation used for deriving \cref{eq:ODEflow1b} and \cref{eq:ODEflow2b} from \cref{eq:forwarddeformation_small2} and \cref{eq:forwarddeformation_small1}, respectively implies 
\begin{equation}\label{eq:inverse_relation_linear_deformation}
\Bigl(\Id  + \frac{1}{N} \velocityfield(t_i,\Cdot)\Bigr)^{-1} \approx \Id  - \frac{1}{N} \velocityfield(t_{i+1},\Cdot) \quad\text{for $i = 0, \ldots, N-1$.}
\end{equation}
\end{remark}

\paragraph{Gradient descent algorithm}
The approximations in the aforementioned paragraph play a central role in the gradient descent algorithm that implements \cref{eq:gradientflow} for minimizing for finding a local extrema to \cref{eq:OptimV}. 

The numerical implementation is outlined in \cref{alg:GradientDescentAlgorithm_1} for the case with geometric group action. 
The version for mass-preserving group action can be obtained from \cref{alg:GradientDescentAlgorithm_1} after modifying two steps. 
The first is to replace step~\ref{step:UpdateJacobian} with:
\begin{quote}
Update the Jacobian determinant $\bigl\vert \Diff\gelement{t_i,0}{\velocityfield^k}\bigr\vert$ based on \cref{eq:ODEflow1b}: 
\begin{displaymath}
\bigl\vert \Diff\gelement{t_i,0}{\velocityfield^k}\bigr\vert \gets  \Bigl(1 - \frac{1}{N} \Div\velocityfield^k(t_i,\Cdot)\Bigr)\bigl\vert \Diff\gelement{t_{i-1},0}{\velocityfield^k}\bigr\vert \circ \Bigl(\Id  - \frac{1}{N} \velocityfield^k(t_i,\Cdot)\Bigr)
\end{displaymath}
 for $i = 1, \ldots, N$, where $\bigl\vert \Diff\gelement{t_0,0}{\velocityfield^k}\bigr\vert = \bigl\vert \Diff\gelement{0,0}{\velocityfield^k}\bigr\vert = 1$.
\end{quote}
The second is to replace step~\ref{step:ComputeGradient} with:
\begin{quote}
Compute $\grad \ObjectiveV (\velocityfield^k)$ using \cref{eq:GradGoalFunctionalVRHKS_mp} (use \ac{FFT} based techniques for computing the kernel):
\begin{multline*}
    \grad \ObjectiveV (\velocityfield^k)(t_i,,\Cdot) 
    \gets  2 \gamma \velocityfield^k(t_i,,\Cdot)  \\
      +  \int_{\domain} \bigl\vert \Diff\bigl(\gelement{t_i,0}{\velocityfield^k}\bigr)(y) \bigr\vert \bigl(\template \circ \gelement{t_i,0}{\velocityfield^k}\bigr)(y) \kernel(\Cdot,y) \cdot \grad\Bigl(\grad\MatchingFunctionalX \bigl( \DeforOpX(\velocityfield^k, \template) \bigr)\circ \gelement{t_i,1}{\velocityfield^k}\Bigr)(y)
               \dint y
\end{multline*}
for $i =N, N-1, \ldots, 0$.
\end{quote}
\begin{remark}
A similar gradient descent scheme is proposed in \cite{BeMiTrYo05}. Here, the \ac{RKHS} structure that regulates the smoothness of the velocity fields is defined by a kernel function instead of by a differential operator. This leads to a different implementation. Moreover, the updating strategy in \cite{BeMiTrYo05} that corresponds to in lines 7-9 is: compute $\gelement{t_i,0}{\velocityfield^k}$, $\gelement{t_i,1}{\velocityfield^k}$, and $\Diff\gelement{t_i,1}{\velocityfield^k}$ first and then compute $\template \circ \gelement{t_i,0}{\velocityfield^k}$, $\grad\MatchingFunctionalX \bigl( \DeforOpX(\velocityfield^k, \template) \bigr)\circ \gelement{t_i,1}{\velocityfield^k}$ and $\bigl\vert \Diff\gelement{t_i,1}{\velocityfield^k}\bigr\vert$, respectively. Compared to the updating strategy in  \cref{alg:GradientDescentAlgorithm_1} (lines~7-9), the approach in \cite{BeMiTrYo05} is more costly in terms of computational effort and memory usage. The reason is that in each iteration \cref{alg:GradientDescentAlgorithm_1} just needs to update three scalar fields instead of two vector fields, and one tensor field, and then three scalar fields on the considered domain.
\end{remark}

\begin{algorithm}
\caption{Gradient descent scheme for minimizing $\ObjectiveV $ in \cref{eq:GenericGoalFunctionalV} with geometric group action}
\label{alg:GradientDescentAlgorithm_1}
\begin{algorithmic}[1]
\STATE \emph{Initialize}:
\STATE $k \gets 0$.
\STATE $t_i \gets i/N$ for $i = 0, 1, \ldots, N$.
\STATE $\velocityfield^k(t_i,\Cdot) \gets \velocityfield^0(t_i,\Cdot)$ where $\velocityfield^0(t_i,\Cdot)$ is a given initial vector field.
\STATE Error tolerance $\epsilon > 0$, step size $\stepsize > 0$, and maximum iterations $K > 0$.
\WHILE {$\bigl\vert \grad \ObjectiveV (\velocityfield^k) \bigr\vert > \epsilon$ and $k<K$}
\STATE
Compute $\template \circ \gelement{t_i,0}{\velocityfield^k}$ using \cref{eq:ODEflow1b} by
\begin{displaymath}
\template \circ \gelement{t_i,0}{\velocityfield^k} \gets \bigl(\template \circ \gelement{t_{i-1},0}{\velocityfield^k}\bigr) \circ \Bigl(\Id - \frac{1}{N}\velocityfield^k(t_i,\Cdot)\Bigr)
\end{displaymath}
for $i = 1, \ldots, N$, where $\template \circ \gelement{t_0,0}{\velocityfield^k} = \template \circ \gelement{0,0}{\velocityfield^k} = \template$.
\STATE\label{step:UpdateJacobian} Update the Jacobian determinant $\bigl\vert \Diff\gelement{t_i,1}{\velocityfield^k}\bigr\vert$ based on \cref{eq:ODEflow2b}: 
\begin{displaymath}
\bigl\vert \Diff\gelement{t_i,1}{\velocityfield^k}\bigr\vert \gets  \Bigl(1 + \frac{1}{N} \Div\velocityfield^k(t_i,\Cdot)\Bigr)\bigl\vert \Diff\gelement{t_{i+1},1}{\velocityfield^k}\bigr\vert \circ \Bigl(\Id  + \frac{1}{N} \velocityfield^k(t_i,\Cdot)\Bigr)
\end{displaymath}
 for $i = N-1, \ldots, 0$, where $\bigl\vert \Diff\gelement{t_{N},1}{\velocityfield^k}\bigr\vert = \bigl\vert \Diff\gelement{1,1}{\velocityfield^k}\bigr\vert = 1$.
\STATE Compute $\grad\MatchingFunctionalX \bigl( \DeforOpX(\velocityfield^k, \template) \bigr)\circ \gelement{t_i,1}{\velocityfield^k}$ using \cref{eq:ODEflow2b} by
\begin{displaymath}
\grad\MatchingFunctionalX \bigl( \DeforOpX(\velocityfield^k, \template) \bigr)\circ \gelement{t_i,1}{\velocityfield^k} \gets \Bigl(  \grad\MatchingFunctionalX \bigl( \DeforOpX(\velocityfield^k, \template) \bigr)\circ \gelement{t_{i+1},1}{\velocityfield^k}\Bigl) \circ \Bigl(\Id  + \frac{1}{N} \velocityfield^k(t_i,\Cdot)\Bigr)
\end{displaymath}
for $i = N-1, \ldots, 0$, where 
\begin{displaymath}
 \grad\MatchingFunctionalX \bigl( \DeforOpX(\velocityfield^k, \template) \bigr)\circ \gelement{t_N,1}{\velocityfield^k} = \grad\MatchingFunctionalX \bigl( \DeforOpX(\velocityfield^k, \template) \bigr)\circ \gelement{1,1}{\velocityfield^k} = \grad\MatchingFunctionalX \bigl( \DeforOpX(\velocityfield^k, \template) \bigr). 
 \end{displaymath}
\STATE\label{step:ComputeGradient} Compute $\grad \ObjectiveV (\velocityfield^k)$ using \cref{eq:GradGoalFunctionalVRHKS} (use \ac{FFT} based techniques for computing the kernel):
\begin{multline*}
    \grad \ObjectiveV (\velocityfield^k)(t_i,,\Cdot) 
    \gets  2 \gamma \velocityfield^k(t_i,,\Cdot)  \\
      -  \int_{\domain} 
              \grad\MatchingFunctionalX \bigl( \DeforOpX(\velocityfield^k, \template) \bigr)\bigl( \gelement{t_i,1}{\velocityfield^k}(y) \bigr)\kernel(\Cdot,y) \cdot 
              \grad(\template \circ \gelement{t_i,0}{\velocityfield^k})(y) 
             \bigl\vert \Diff(\gelement{t_i,1}{\velocityfield^k})(y) \bigr\vert \dint y
\end{multline*}
for $i =N, N-1, \ldots, 0$.
\STATE Update $\velocityfield^k$:
\begin{displaymath}
 \velocityfield^k(t_i,\Cdot)  \gets \velocityfield^k(t_i,\Cdot) -  \stepsize \grad \ObjectiveV (\velocityfield^k)(t_i,\Cdot)
\end{displaymath}
for $i =N, N-1, \ldots, 0$.
\STATE $k \gets k+1$.
\ENDWHILE
\RETURN $\template \circ \gelement{t_i,0}{\velocityfield^k}$ for $i = 1, \ldots, N$.
\end{algorithmic}
\end{algorithm}

\paragraph{Complexity analysis} 
The complexity analysis applies to \cref{alg:GradientDescentAlgorithm_1}. Since the main part (each iteration) of the algorithm is located to lines 7-11, the analysis focuses on those lines. 

Suppose that $\domain \subset \Real^2$ and  the size of the image to be reconstructed is $n\times n$. In line 7, one updates $\template \circ \gelement{t_i,0}{\velocityfield^k}$ for $i = 1, \ldots, N$. Moreover, each of them needs to be used to compute the gradient of the objective functional in line 10 at each time point. Hence they need to be stored at hand. Through simple analysis, in this step, the computational cost is $O(n^2N)$ and the space complexity is $O(n^2N)$ . Similarly, for lines 8 and 9, the Jacobian determinant $\bigl\vert \Diff\gelement{t_i,1}{\velocityfield^k}\bigr\vert$ and $\grad\MatchingFunctionalX \bigl( \DeforOpX(\velocityfield^k, \template) \bigr)\circ \gelement{t_i,1}{\velocityfield^k}$ need to be updated and then stored for $i = N-1, \ldots, 0$. Therefore, for each of the two steps, the computational cost is $O(n^2N)$ and the space complexity is $O(n^2N)$. Evidently, lines~10-11 can be combined into one step. Line~10 uses \ac{FFT} to compute the gradient of the objective functional at each time point. Hence the computational cost for this line is $O(Nn^3\log n)$. In line~11, a vector field is updated at each time point and this part requires twice the memory compared to a scalar field on 2D domain. In conclusion, the computational cost is $O(Nn^3\log n)$ and the memory requirements are $O(n^2N)$.

\section{Application to 2D tomography}\label{sec:2DCT}
Indirect registration with geometric group action is here applied to 2D parallel beam tomography with very sparse or highly noisy data. Although this is not a full evaluation, it does still illustrate the performance of indirect registration.

\subsection{The indirect registration problem}
Let $\RecSpace=\SBV(\domain,\Real) \bigcap \LpSpace^{\infty}(\domain,\Real)$ for a fixed image bounded image domain $\domain \subset \Real^2$. Elements in $\RecSpace$ represent 2D images and this space is equipped with an $\LpSpace^{2}$--inner product. Diffeomorphisms act on such elements through a geometric group action and the goal is to register a given template $\template \in \RecSpace$ against a target that observed indirectly as in \cref{eq:InvProb}. Here, $\data \in \DataSpace = \LpSpace^{2}(\datadomain,\Real)$ with $\datadomain$ denoting a fixed manifold of parallel lines in $\Real^2$ (parallel beam data). The template is assumed to be differentiable. 
The forward operator $\ForwardOp \colon \RecSpace \to \DataSpace$ is here the 2D ray/Radon transform, see \cite{NaWu01} for further details.

The indirect registration scheme is given as the solution to \cref{eq:OptimV} where $\MatchingFunctionalX  \colon \RecSpace \to \Real$ in 
\cref{eq:MatchingFunctionalX} is given as 
\begin{displaymath}
\MatchingFunctionalX (\signal) :=  \bigl\Vert \ForwardOp( \signal ) - \data \bigr\Vert_{\LpSpace^{2}(\datadomain,\Real)}^2
    \quad\text{for $\signal \in \RecSpace$.} 
\end{displaymath}
Hence, $\mu=0$ in \cref{eq:MatchingFunctionalX}, \ie, there is no additional regularization and \cref{eq:OptimV} reduces to 
\begin{equation}\label{eq:RegOptim}
   \min_{\velocityfield \in \FlowSpace{2}} \biggl[ \gamma \Vert \velocityfield \Vert_{\FlowSpace{2}}^2    
   + \Bigl\Vert \ForwardOp\bigl( \template \circ \gelement{0,1}{\velocityfield} \bigr) - \data \Bigr\Vert_{\LpSpace^{2}(\datadomain,\Real)}^2
   \biggr].
\end{equation}
In the above, $\gamma >0$ is a fixed regularization parameter that weights the need for minimal deformation against the need to register against the indirectly observed target. 

Next, consider a set $\LieAlgebra$ of vector fields that is a \ac{RKHS} with a reproducing kernel represented by symmetric and positive definite Gaussian function $\kernel \colon \domain \times \domain \to \Matrix_{+}^{2 \times 2}$. Then $\LieAlgebra$ is admissible and is continuously embedded in $\LpSpace^2(\domain,\Real^2)$. The Gaussian is parameterised by $\sigma >0$ as 
\begin{equation}\label{eq:KernelEq}
  \kernel(x,y) := 
      \exp\Bigl(-\dfrac{1}{2 \sigma^2} \Vert x-y \Vert_2 \Bigr)
      \begin{pmatrix} 
          1  & 0 \\
          0  & 1
      \end{pmatrix} 
\quad\text{for $x,y \in \Real^2$.}
\end{equation}
The kernel width $\sigma$ also acts as a regularization parameter.

\subsection{Phantoms and data acquisition protocol}
All images in the shape space $\RecSpace$, such as the target (phantom) and template, are supported in a fixed rectangular image domain $\domain = \lbrack -16, 16 \rbrack \times \lbrack -16, 16 \rbrack$. Data is obtained by first evaluating the 2D parallel beam ray transform on $\datadomain$, which consists of parallel lines whose directions are equally distributed within $[0^\circ, 180^\circ]$. Finally, additive Gaussian white noise at varying noise levels is added to data.

The noise level in data is quantified in terms of the \ac{SNR} expressed using the logarithmic decibel (dB) scale:
\begin{equation}\label{snr}
  \mathrm{SNR}(\data) = 
     10\log_{10} \Biggl[
     \frac{\bigl\Vert \data_{\text{ideal}} - \mu_{\text{ideal}} \bigr\Vert_{\DataSpace}^2}
            {\bigl\Vert \noisedata - \mu_{\text{noise}} \bigr\Vert_{\DataSpace}^2}
     \Biggr]
     \quad\text{for $\data = \data_{\text{ideal}} + \noisedata$.}
\end{equation}
In the above, $\data_{\text{ideal}} := \ForwardOp(\truesignal)$ is the noise-free component of data and $\noisedata \in \DataSpace$ is the noise component. Furthermore, $\mu_{\text{ideal}}$ is the average of $\data^{\text{ideal}}$ and $\mu_{\text{noise}}$ is the 
average of $\noisedata$. Note here that computing the data noise level requires access to the noise-free component of data, which is the case for simulated data.

\subsection{Reconstruction protocol}
Test suites~1 and 2 compares \cref{alg:GradientDescentAlgorithm_1} against usual reconstruction of the target from tomographic data. For the latter, \ac{FBP} and \ac{TV} methods are used. The rationale behind comparing against direct reconstruction is to assess the influence of a priori information contained in \cref{eq:ExactShapeAss}, which states that the unknown target can be obtained by a diffeomorphic deformation of a given template. Hence, to some extent the indirect registration problem can be seen as a shape reconstruction problem. Furthermore, the  corresponding regularization parameters are set separately each method and data noise level as to give the best visual agreement with the target. 

The reconstruction methods \ac{FBP} and \ac{TV} are the ones available in Operator Discretization Library (\href{http://github.com/odlgroup/odl}{http://github.com/odlgroup/odl}). 
\begin{itemize}
\item \ac{FBP} reconstruction is obtained using linear interpolation for the back projection and the Hamming filter at varying frequency scaling.
\item \ac{TV} reconstruction is obtained by solving \cref{eq:VarReg} with $\RegFunc$ as the total variation
functional: 
\begin{displaymath} 
\RegFunc(\signal) := \int_{\domain} \bigl\vert \grad \signal(x) \bigr\vert\dint x. 
\end{displaymath}
The non-smooth minimization \cref{eq:VarReg} is solved using the primal-dual method in \cite{chpo11}.
\end{itemize}

\subsection{Test suites and results}\label{sec:results}
The test suites seek to assess robustness and performance against different kind of objects, sensitivity of indirect registration against choice of regularization parameters $\gamma, \sigma$, and impact of using a template with different topology from target. The targets include both single- and multi-object images.

\paragraph{Test suite 1: Single-object indirect registration}
The aim here is to investigate the performance of indirect registration against highly noisy sparse tomographic data from a single object target. The template and target are chosen similar to the direct registration test in \cite[section~4.2]{BaJoMo15}. 

Images in shape space are digitized using $64 \times 64$ pixels. The parallel beam tomographic data consists of 92 parallel line projections along ten uniformly distributed directions in $[0^\circ, 180^\circ]$.  Hence, in a fully digitized setting, the inverse problem in \cref{eq:InvProb} amounts to solving for $4\,096$ unknowns from $920$ linear equations. Hence, this problem is highly under-determined and data is also highly noisy with a noise level (\ac{SNR}) of 4.87 dB.

When solving the indirect registration problem, the time interval $[0, 1]$ for the flow of diffeomorphisms is sampled uniformly at $N = 20$ time points. The gradient step size is set as $\stepsize = 0.02$ and the regularization parameters are $\gamma = 10^{-7}$ and $\sigma = 6.0$. The gradient descent is stopped after 200 iterations. Finally, the comparison is against direct reconstructions from \ac{FBP} (Hamming filter with frequency scaling $0.4$) and \ac{TV} (1000 iterations, regularization parameter $\mu= 3.0$). 

The results, which are shown in \cref{Test_suite_1:single_object}, show that indirect registration performs fine against a single-object target as long as the template has the topology (\eg, no holes).

\begin{figure}[h]
\centering
% First row
   \begin{minipage}[t]{0.32\textwidth}%
     \centering
     \includegraphics[trim=75 25 60 40, clip, width=\textwidth]{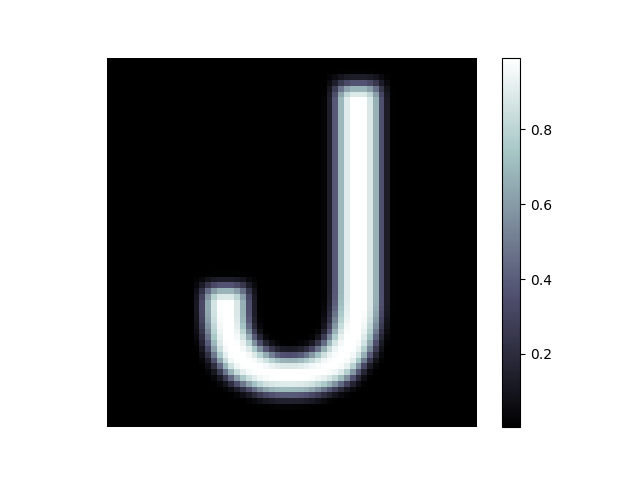}     
     \vskip-0.25\baselineskip
     Template
   \end{minipage}%
   \hfill
   \begin{minipage}[t]{0.32\textwidth}%
     \centering
    \includegraphics[trim=75 25 60 40, clip, width=\textwidth]{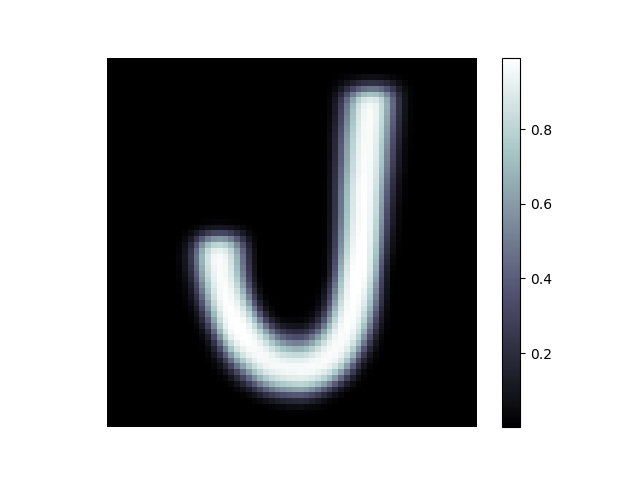}
     \vskip-0.25\baselineskip
     $t=0.25$
   \end{minipage}%
   \hfill
   \begin{minipage}[t]{0.32\textwidth}%
     \centering
     \includegraphics[trim=75 25 60 40, clip, width=\textwidth]{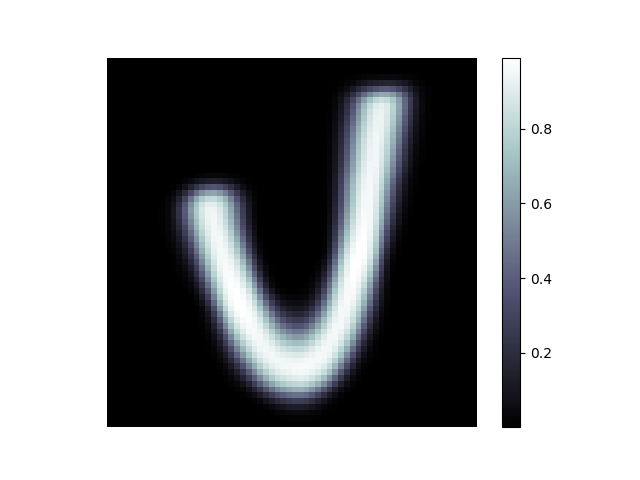}
     \vskip-0.25\baselineskip
     $t=0.5$
   \end{minipage}%
\par\medskip      
% Second row
   \begin{minipage}[t]{0.32\textwidth}%
     \centering
     \includegraphics[trim=75 25 60 40, clip, width=\textwidth]{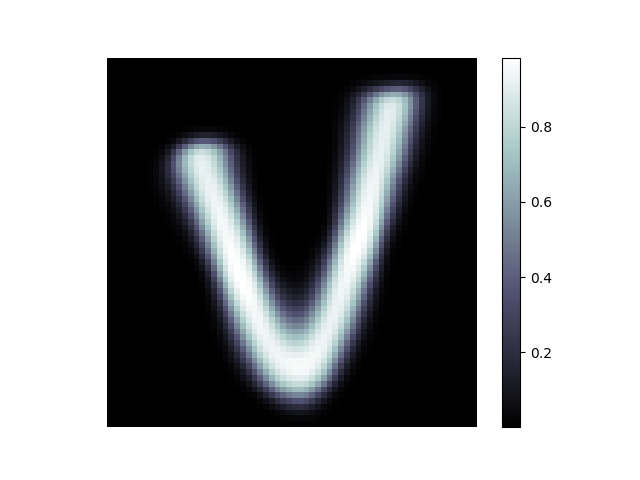}     
     \vskip-0.25\baselineskip
     $t=0.75$
   \end{minipage}%
   \hfill
   \begin{minipage}[t]{0.32\textwidth}%
     \centering
    \includegraphics[trim=75 25 60 40, clip, width=\textwidth]{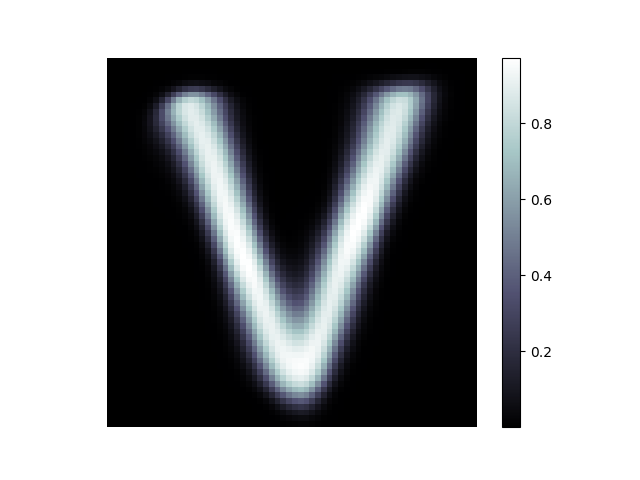}
     \vskip-0.25\baselineskip
     $t=1.0$
   \end{minipage}%
   \hfill
   \begin{minipage}[t]{0.32\textwidth}%
     \centering
     \includegraphics[trim=75 25 60 40, clip, width=\textwidth]{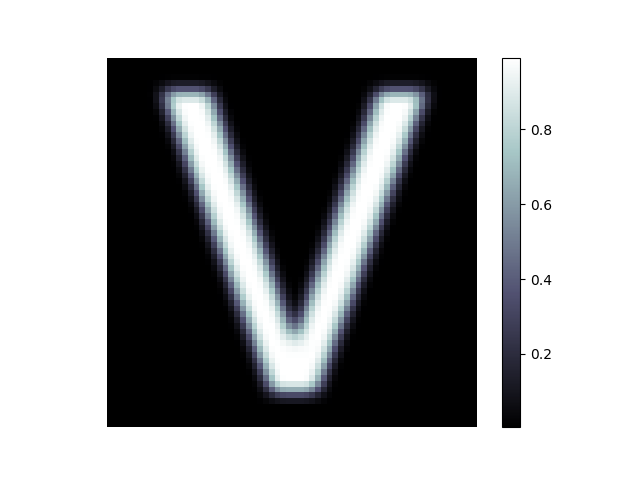}
     \vskip-0.25\baselineskip
     Target
   \end{minipage}%
\par\medskip      
% Third row
   \begin{minipage}[t]{0.32\textwidth}%
     \centering
     \includegraphics[trim=30 15 30 40, clip, width=\textwidth]{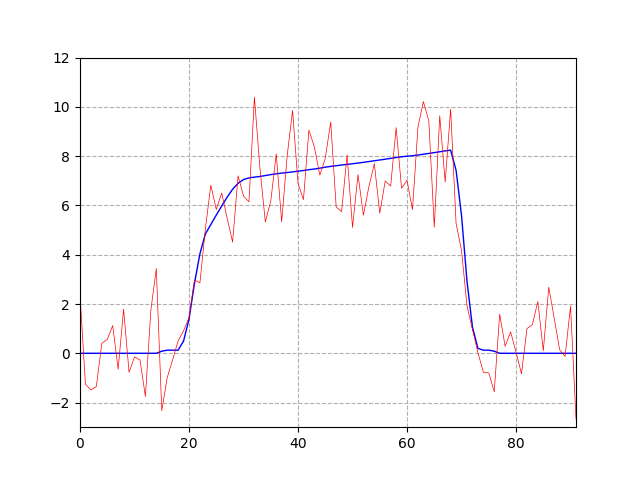}
     \vskip-0.25\baselineskip
     Data at $0^{\circ}$
   \end{minipage}%
   \hfill
   \begin{minipage}[t]{0.32\textwidth}%
     \centering
    \includegraphics[trim=75 25 60 40, clip, width=\textwidth]{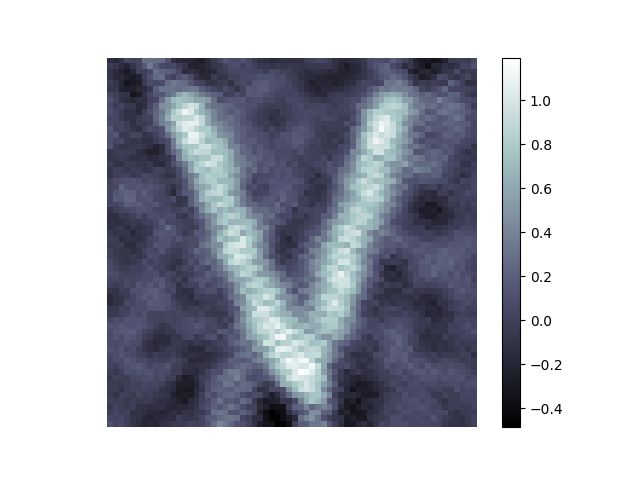}
     \vskip-0.25\baselineskip
     \ac{FBP}
   \end{minipage}%
   \hfill
   \begin{minipage}[t]{0.32\textwidth}%
     \centering
     \includegraphics[trim=75 25 60 40, clip, width=\textwidth]{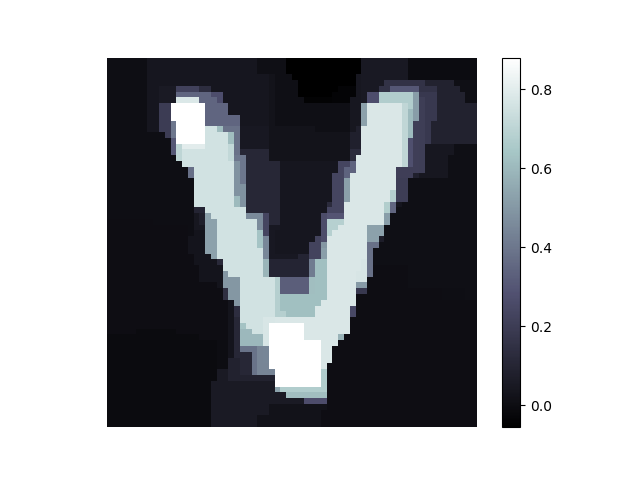}
     \vskip-0.25\baselineskip
     \ac{TV}
   \end{minipage}%
\caption{Test suite 1: Single-object indirect registration. The template (top left) is deformed by a diffeomorphism to match the target (rightmost, second row). Images labeled with time $t$ show the final flow of diffeomorphic deformations obtained from indirect registration with $t=1.0$ denoting the final (indirectly) registered template. Data is highly noisy and under-sampled, bottom leftmost image shows data at angle $0^\circ$ (blue smooth curve is noise-free data, red jagged curve is noisy data). The \ac{FBP} and \ac{TV} reconstructions are for comparison.}
\label{Test_suite_1:single_object}
\end{figure}

\paragraph{Test suite 2: Multi-object indirect registration}
The previous test used a target that only contains one single object. This test considers a multi-object target, namely a target consisting of six separately star-like objects with grey-values in $[0, 1]$, which is digitized using $438 \times 438$ pixels. The choice of target is similar to the direct registration test in \cite{BlPeBaBa15}. 

The data is highly sparse parallel beam tomographic projections taken along only six directions that are uniformly  distributed in $[0, 180^\circ]$. At each direction, one samples the ray transform in a set of uniformly spaced $620$ parallel lines. Hence, the inverse problem in \cref{eq:InvProb} amounts to reconstructing $191\,844$ unknowns from $3\,720$ linear equations. Besides being highly under-determined, data is also highly noisy with noise level (\ac{SNR}) of $4.75$\,dB.

As in test suite~1, the time interval $[0, 1]$ for the flow of diffeomorphisms is sampled uniformly at $N = 20$ time points, the gradient step size is set to $\stepsize = 0.04$, and the regularization parameter is $\gamma = 10^{-7}$. On the other hand, the other regularization parameter regulating the width of the \ac{RKHS} kernel is set as $\sigma = 2.0$ and the gradient descent was stopped after $200$ iterations. Finally, comparison is against direct reconstructions from \ac{FBP} (Hamming filter with frequency scaling $0.4$) and \ac{TV} (1000 iterations, regularization parameter $\mu= 1.0$). 

The results, which are shown in \cref{Test_suite_2:multi_object}, show that indirect registration performs fine also against a multi-object target as long as the template has the same number of objects.

\begin{figure}[h]
\centering
% First row
   \begin{minipage}[t]{0.32\textwidth}%
     \centering
     \includegraphics[trim=75 25 60 40, clip, width=\textwidth]{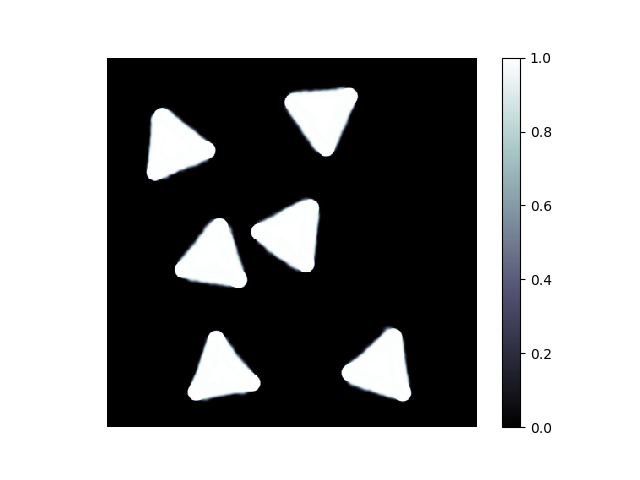}
     \vskip-0.25\baselineskip
     Template
   \end{minipage}%
   \hfill
   \begin{minipage}[t]{0.32\textwidth}%
     \centering
    \includegraphics[trim=75 25 60 40, clip, width=\textwidth]{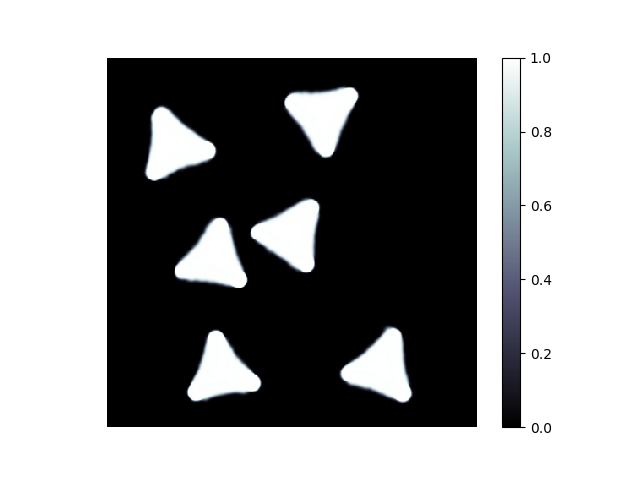}
     \vskip-0.25\baselineskip
     $t=0.25$
   \end{minipage}%
   \hfill
   \begin{minipage}[t]{0.32\textwidth}%
     \centering
     \includegraphics[trim=75 25 60 40, clip, width=\textwidth]{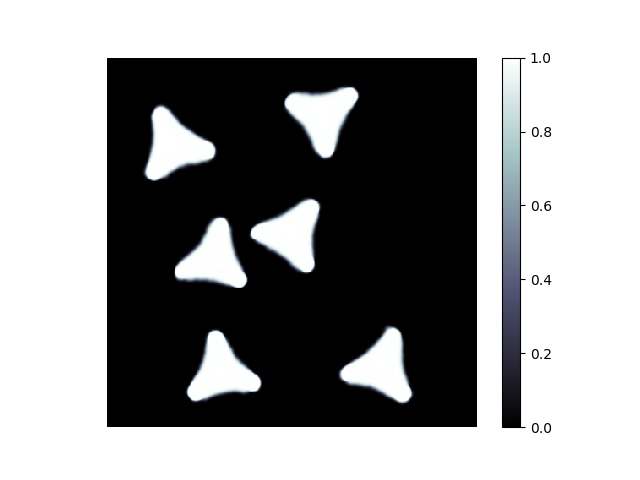}
     \vskip-0.25\baselineskip
     $t=0.5$
   \end{minipage}%
\par\medskip      
% Second row
   \begin{minipage}[t]{0.32\textwidth}%
     \centering
     \includegraphics[trim=75 25 60 40, clip, width=\textwidth]{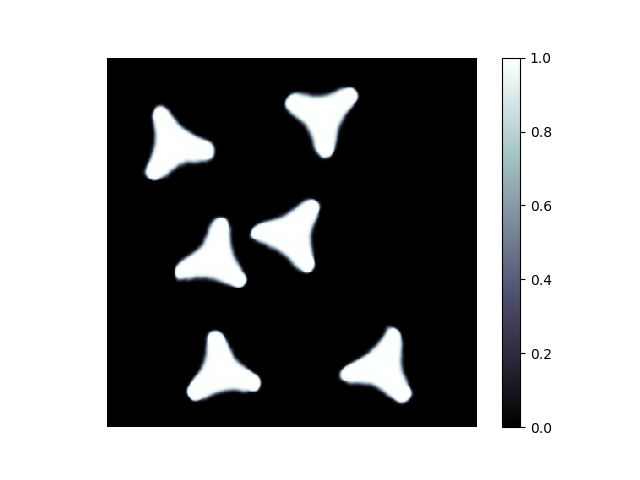}     
     \vskip-0.25\baselineskip
     $t=0.75$
   \end{minipage}%
   \hfill
   \begin{minipage}[t]{0.32\textwidth}%
     \centering
    \includegraphics[trim=75 25 60 40, clip, width=\textwidth]{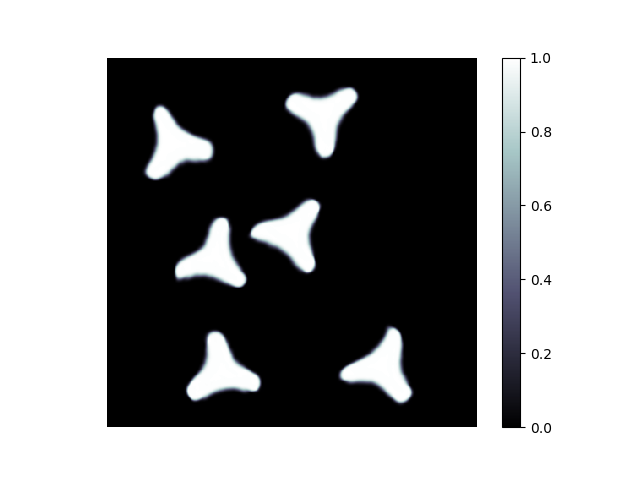}
     \vskip-0.25\baselineskip
     $t=1.0$
   \end{minipage}%
   \hfill
   \begin{minipage}[t]{0.32\textwidth}%
     \centering
     \includegraphics[trim=75 25 60 40, clip, width=\textwidth]{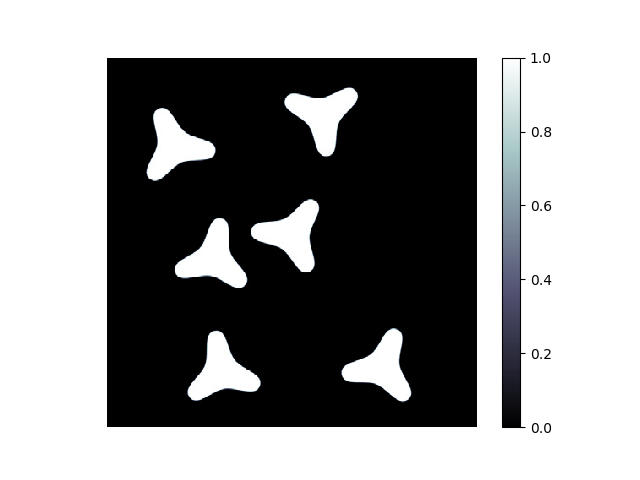}
     \vskip-0.25\baselineskip
     Target
   \end{minipage}%
\par\medskip      
% Third row
   \begin{minipage}[t]{0.32\textwidth}%
     \centering
     \includegraphics[trim=30 15 30 40, clip, width=\textwidth]{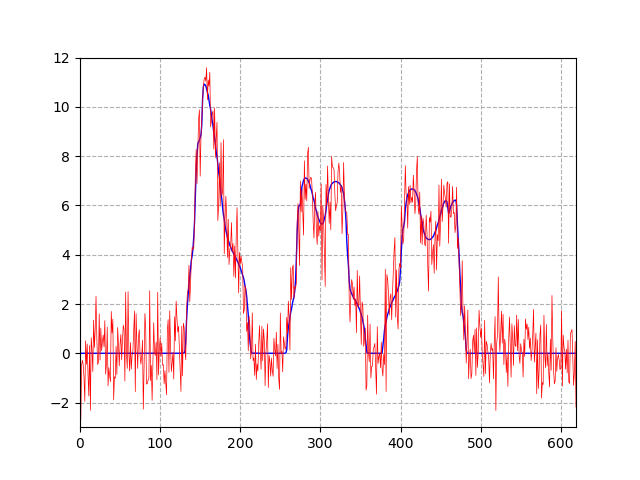}
     \vskip-0.25\baselineskip
     Data at $0^{\circ}$
   \end{minipage}%
   \hfill
   \begin{minipage}[t]{0.32\textwidth}%
     \centering
    \includegraphics[trim=75 25 60 40, clip, width=\textwidth]{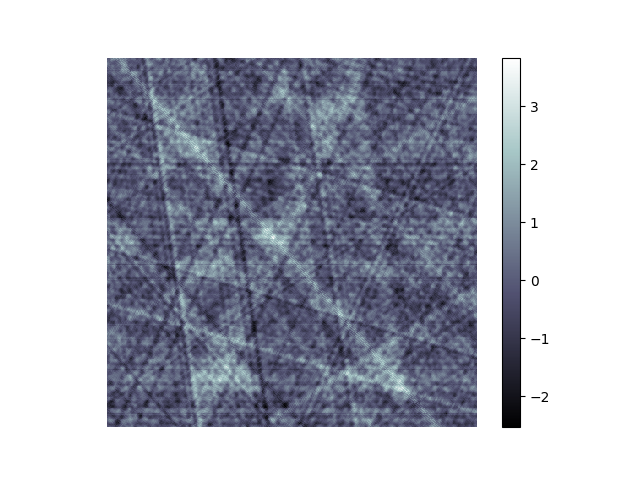}
     \vskip-0.25\baselineskip
     \ac{FBP}
   \end{minipage}%
   \hfill
   \begin{minipage}[t]{0.32\textwidth}%
     \centering
     \includegraphics[trim=75 25 60 40, clip, width=\textwidth]{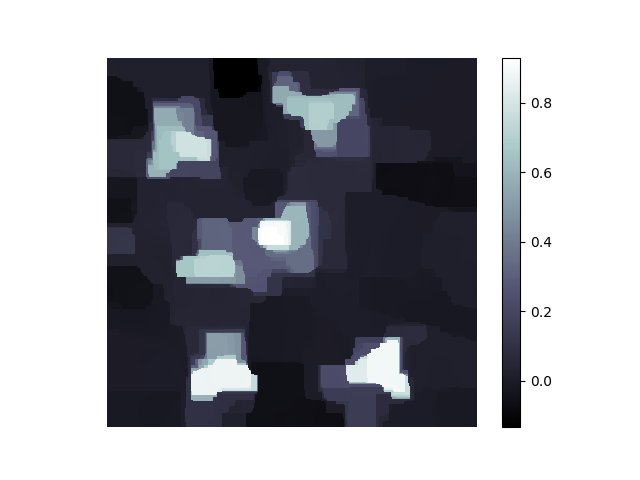}
     \vskip-0.25\baselineskip
     \ac{TV}
   \end{minipage}%
\caption{Test suite 2: Multi-object indirect registration. The template (top left) is a target consisting of six separately triangle-like objects. Images labeled with time $t$ show the final flow of diffeomorphic deformations obtained from indirect registration with $t=1.0$ denoting the final (indirectly) registered template. Data is highly noisy and under-sampled, bottom leftmost image shows data at angle $0^\circ$ (blue smooth curve is noise-free data, red jagged curve is noisy data). \Ac{FBP} and \ac{TV} reconstructions are shown for comparison.}
\label{Test_suite_2:multi_object}
\end{figure}

\paragraph{Test suite 3: Sensitivity \wrt choice of regularization parameters}

Choosing the regularization parameters is a well-known issue in most regularization schemes. Even though there is some theory for how to do this, in practice it is often chosen using heuristic methods and especially so when data is highly noisy and/or under-sampled. A natural question is therefore to empirically investigate the sensitivity of the indirect registration against variations in the regularization parameters $\gamma$ and $\sigma$ that enter in \cref{eq:RegOptim} and \cref{eq:KernelEq}, respectively.

Here the target is the well-known Shepp-Logan phantom consisting of ten ellipsoids with grey-values in $[0, 1]$. Images in shape space are digitized using $256 \times 256$ pixels. The parallel beam tomographic data consists of 362 parallel line projections along ten uniformly distributed directions in $[0^\circ, 180^\circ]$.  Hence, in a fully digitized setting, the inverse problem in \cref{eq:InvProb} amounts to solving for $65\,536$ unknowns from $3\,620$ linear equations. Besides being highly under-determined, data is also highly noisy with noise level (\ac{SNR}) of $7.06$\,dB.

As in test suite~1, the time interval $[0, 1]$ for the flow of diffeomorphism is sampled uniformly at $N = 20$ time points, the gradient step size is set as $\stepsize = 0.02$, and the regularization parameter is $\gamma = 10^{-7}$. 

The regularization parameters $(\gamma, \sigma)$ are varied and each choice results in a different indirectly registered image. Each of these are then matched against the target and the registration is quantitatively assessed using \ac{SSIM} \cite{WaBoShSi04} and \ac{PSNR} \cite{HuGh08} as figure of merits. 

The \ac{SSIM} and \ac{PSNR} values are tabulated in \cref{Test_suite_3:Sensitivity_table}, see also \cref{Test_suite_3:Sensitivity_figure} for some selected reconstructions. As given in \cref{Test_suite_3:Sensitivity_table},  the \ac{SSIM} and \ac{PSNR} values are quite similar for $\gamma = 10^{-7}, 10^{-5}, 10^{-3}, 10^{-1}$ and $\sigma = 1.0, 2.0, 2.5, 3.0, 4.0$. The corresponding reconstructed results in \cref{Test_suite_3:Sensitivity_figure} are almost the same and close to the target. But these values are obviously reduced when $\gamma = 10$ or $\sigma = 8.0$. As shown in \cref{Test_suite_3:Sensitivity_figure}, the resulting performance is not so good compared with the target. Since the choosing intervals for $\gamma$ and $\sigma$ are relatively large, this test shows that to some extent the proposed method is not sensitive to the precise choice of the regularization parameters $\gamma$ and $\sigma$ under quantitative figures of merits (SSIM and PSNR). However those values are chosen too big that would cause over-regularized results. 
\begin{table}
\centering
\begin{tabular}{c | r r r r r}
\diagbox{$\sigma$}{$\gamma$}  &  \multicolumn{1}{c}{$10^{-7}$} &  \multicolumn{1}{c}{$10^{-5}$} & \multicolumn{1}{c}{$10^{-3}$} 
 &  \multicolumn{1}{c}{$10^{-1}$} & \multicolumn{1}{c}{$10$}  \\ 
\toprule
 \multirow{2}{*}{1.0}              &  0.8958   &  0.8958   &  0.8957   &  0.8820   & 0.6800 
                              \\   &  26.69    &  26.69    &  26.69    &  26.18    & 17.82  \\ 
\hline                                  
 \multirow{2}{*}{2.0}              &  0.9039          &  0.9039   &  0.9038   &  0.8958   & 0.6967 
                              \\   &  26.96           &  26.96    &  26.96    &  26.72    & 18.45  \\ 
\hline                                   
 \multirow{2}{*}{2.5}              &  0.9018          &  0.9018   &  0.9017   &  0.8964   & 0.7057 
                              \\   &  26.90           &  26.90    &  26.89    &  26.74    & 18.71  \\ 
 \hline                                 
 \multirow{2}{*}{3.0}              &  0.9007          &  0.9007   &  0.9006   &  0.8968   & 0.7164 
                              \\   &  26.90           &  26.90    &  26.89    &  26.77    & 19.03  \\ 
\hline                                   
 \multirow{2}{*}{4.0}              &  0.8992          &  0.8992   &  0.8992   &  0.8979   & 0.7426 
                              \\   &  26.88           &  26.88    &  26.87    &  26.77    & 19.79  \\ 
\hline                                   
 \multirow{2}{*}{8.0}              &  0.8305          &  0.8305   &  0.8307   &  0.8438   & 0.7960 
                              \\   &  24.15           &  24.15    &  24.16    &  24.84    & 21.27  \\ 
\bottomrule                              
\end{tabular}
\caption{Test suite 3: \Ac{SSIM} and \ac{PSNR} values of indirectly registered images as compared to target for varying values of the regularization parameters $\gamma$ and $\sigma$, see \cref{Test_suite_3:Sensitivity_figure} for selected images. Each table entry has two values, the upper is the \ac{SSIM} and the bottom is the \ac{PSNR}.}
\label{Test_suite_3:Sensitivity_table}
\end{table}

\begin{figure}[h]
\centering
% First row
   \begin{minipage}[t]{0.32\textwidth}%
     \centering
     \includegraphics[trim=75 25 60 40, clip, width=\textwidth]%
       {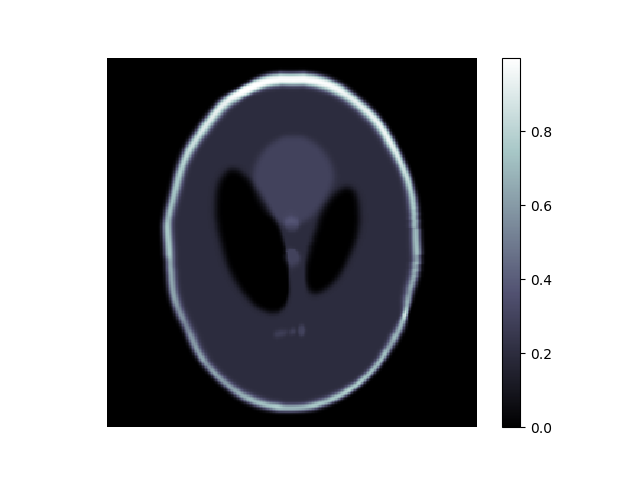}
     \vskip-0.25\baselineskip
     $\gamma=10^{-7}$, $\sigma = 2.0$
   \end{minipage}%
   \hfill
   \begin{minipage}[t]{0.32\textwidth}%
     \centering
     \includegraphics[trim=75 25 60 40, clip, width=\textwidth]%
       {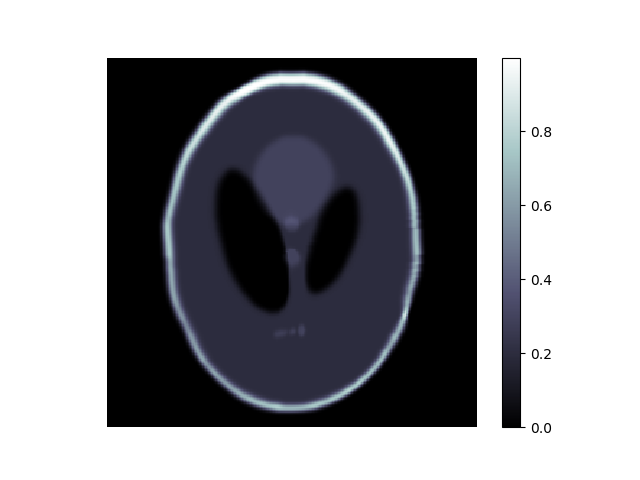}
     \vskip-0.25\baselineskip
     $\gamma=10^{-5}$, $\sigma = 2.0$
   \end{minipage}%
   \hfill
   \begin{minipage}[t]{0.32\textwidth}%
     \centering
     \includegraphics[trim=75 25 60 40, clip, width=\textwidth]%
       {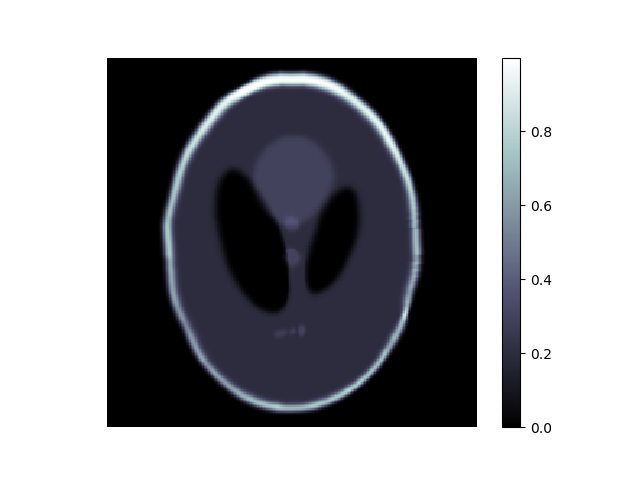}
     \vskip-0.25\baselineskip
     $\gamma=10^{-3}$, $\sigma = 2.0$
   \end{minipage}%
   \iffalse
\par\medskip      
% Second row
   \begin{minipage}[t]{0.32\textwidth}%
     \centering
     \includegraphics[trim=75 25 60 40, clip, width=\textwidth]%
       {shepp_logan_recon_LDDMM_geom500_time_pts20_angles10_snr_7_06_sigma2_5_eps0_02_lamb10-7}
     \vskip-0.25\baselineskip
     $\gamma=10^{-7}$, $\sigma = 2.5$
   \end{minipage}%
   \hfill
   \begin{minipage}[t]{0.32\textwidth}%
     \centering
     \includegraphics[trim=75 25 60 40, clip, width=\textwidth]%
       {shepp_logan_recon_LDDMM_geom500_time_pts20_angles10_snr_7_06_sigma2_5_eps0_02_lamb10-5}
     \vskip-0.25\baselineskip
     $\gamma=10^{-5}$, $\sigma = 2.5$
   \end{minipage}%
   \hfill
   \begin{minipage}[t]{0.32\textwidth}%
     \centering
     \includegraphics[trim=75 25 60 40, clip, width=\textwidth]%
       {shepp_logan_recon_LDDMM_geom500_time_pts20_angles10_snr_7_06_sigma2_5_eps0_02_lamb10-3}
     \vskip-0.25\baselineskip
     $\gamma=10^{-3}$, $\sigma = 2.5$
   \end{minipage}%
   \fi
\par\medskip      
% Third row
   \begin{minipage}[t]{0.32\textwidth}%
     \centering
     \includegraphics[trim=75 25 60 40, clip, width=\textwidth]%
       {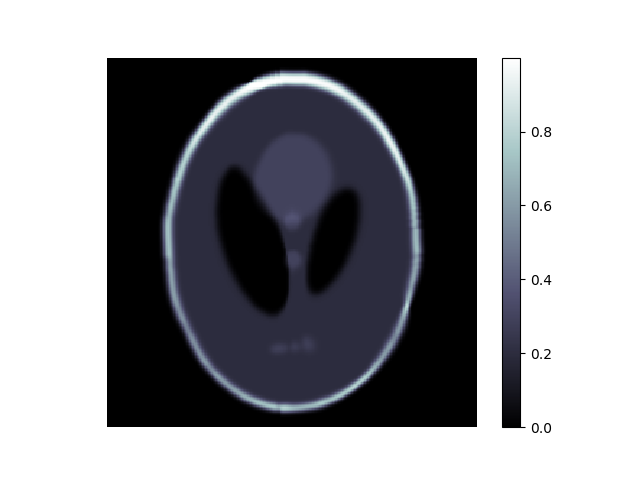}
     \vskip-0.25\baselineskip
     $\gamma=10^{-7}$, $\sigma = 3.0$
   \end{minipage}%
   \hfill
   \begin{minipage}[t]{0.32\textwidth}%
     \centering
     \includegraphics[trim=75 25 60 40, clip, width=\textwidth]%
       {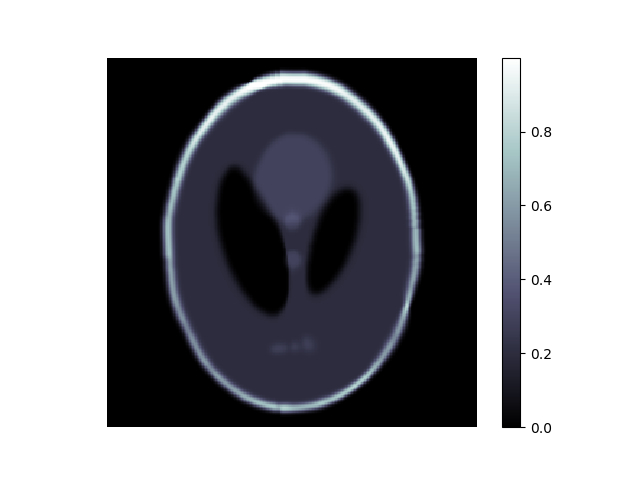}
     \vskip-0.25\baselineskip
     $\gamma=10^{-5}$, $\sigma = 3.0$
   \end{minipage}%
   \hfill
   \begin{minipage}[t]{0.32\textwidth}%
     \centering
     \includegraphics[trim=75 25 60 40, clip, width=\textwidth]%
       {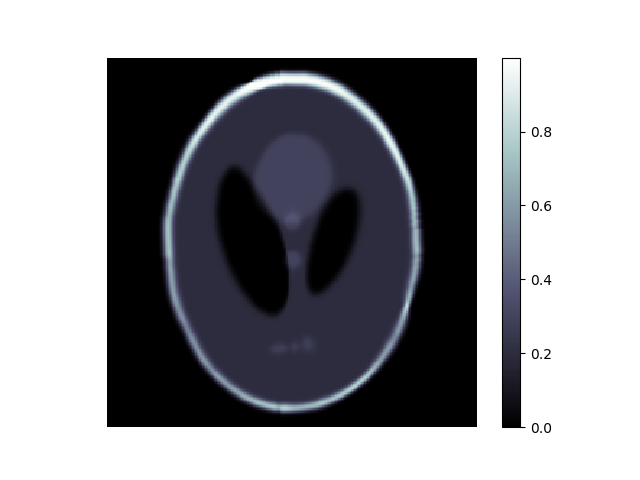}
     \vskip-0.25\baselineskip
     $\gamma=10^{-3}$, $\sigma = 3.0$
   \end{minipage}%
   \par\medskip      
% Fourth row
	\begin{minipage}[t]{0.32\textwidth}%
     \centering
     \includegraphics[trim=75 25 60 40, clip, width=\textwidth]%
       {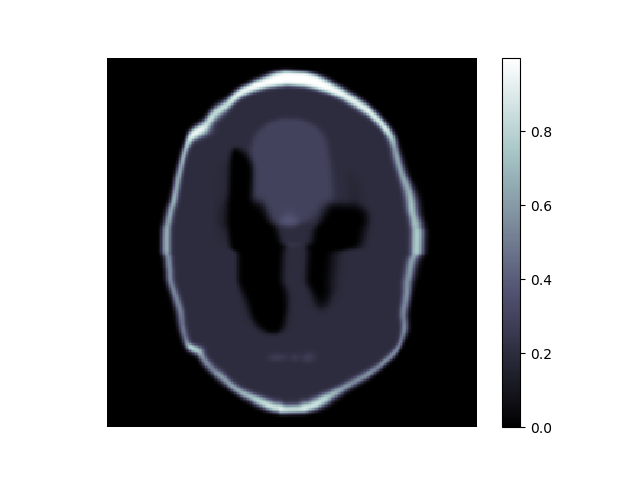}
     \vskip-0.25\baselineskip
     $\gamma=10^{-7}$, $\sigma = 8.0$
   \end{minipage}%
   \hfill
   \begin{minipage}[t]{0.32\textwidth}%
     \centering
     \includegraphics[trim=75 25 60 40, clip, width=\textwidth]%
       {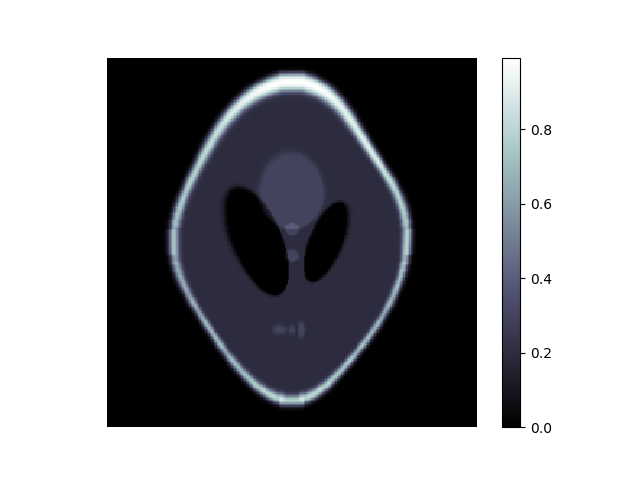}
     \vskip-0.25\baselineskip
     $\gamma=10.0$, $\sigma = 3.0$
   \end{minipage}%
   \hfill
   \begin{minipage}[t]{0.32\textwidth}%
     \centering
     \includegraphics[trim=75 25 60 40, clip, width=\textwidth]%
       {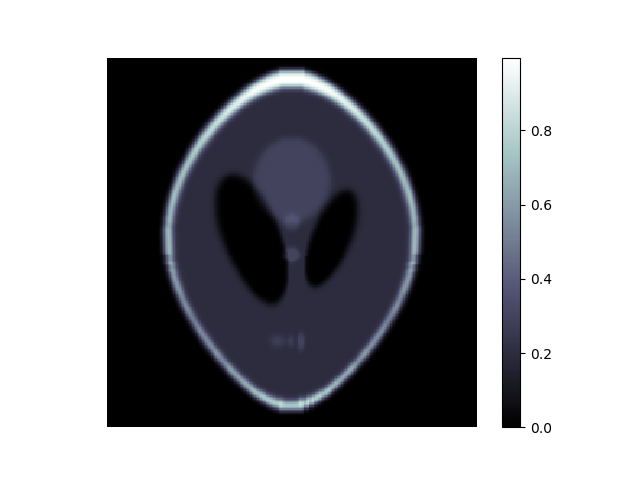}
     \vskip-0.25\baselineskip
     $\gamma=10.0$, $\sigma = 8.0$
   \end{minipage}%
   \par\medskip      
% Bottom row
   \begin{minipage}[t]{0.32\textwidth}%
     \centering
     \includegraphics[trim=75 25 60 40, clip, width=\textwidth]{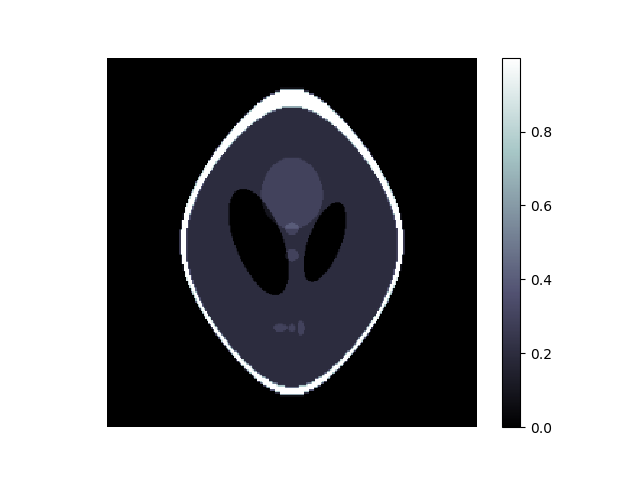}
     \vskip-0.25\baselineskip
     Template
   \end{minipage}%
   \quad
   \begin{minipage}[t]{0.32\textwidth}%
     \centering
     \includegraphics[trim=75 25 60 40, clip, width=\textwidth]{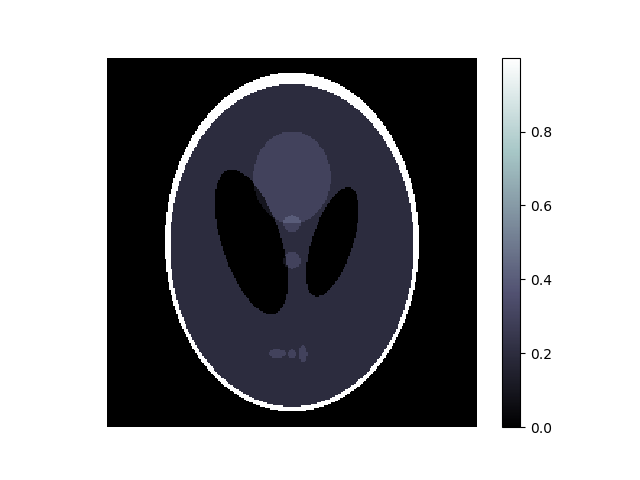}
     \vskip-0.25\baselineskip
     Target
   \end{minipage}%
\caption{Test suite 3: Sensitivity \wrt choice of regularization parameters $\gamma$ and $\sigma$. Images show final registered result (\ie, $t=1$) for different values of these parameters. See \cref{Test_suite_3:Sensitivity_table} for a more quantitative comparison that includes a wider range of parameter values.}
\label{Test_suite_3:Sensitivity_figure}
\end{figure}

\paragraph{Test suite 4: Topology of the template}
This test investigates the influence of a template with a topology that differs from the target. The test suite involves two tests, one where the the template lacks an object when compared to the target and the other where there is an additional object in the template that does not exist in the target. 

In both cases, the template and the target are modifications of the Shepp-Logan phantom used in test suite~3 and tomographic data is generated following the same protocol as in test suite~3, albeit with slightly differing noise levels (7.06\,dB and 6.46\,dB, respectively). Finally, apart from the number of iterations which here is 1000, indirect registration in both cases was performed using the same parameter setting as in test suite~3.

Results are shown in \cref{Test_suite_4:topology}. It is clear that the final indirectly registered template retains the same topology, which is to be expected since diffeomorphic registration using geometric group action can never introduce or remove an object that is not in the template. This also points to the importance of having a template with the correct topology when using diffeomorphic (indirect) registration with a geometric group action.

\begin{figure}[h]
\centering
% First row
   \begin{minipage}[t]{0.32\textwidth}%
     \centering
     \includegraphics[trim=75 25 60 40, clip, width=\textwidth]%
       {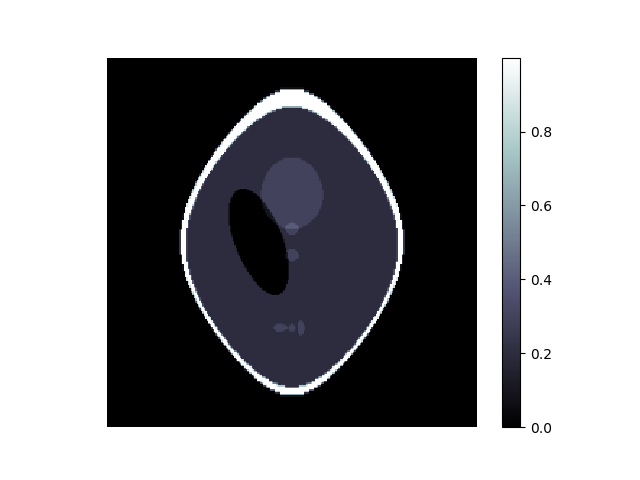}
     \vskip-0.25\baselineskip
     Template
   \end{minipage}%
   \hfill
   \begin{minipage}[t]{0.32\textwidth}%
     \centering
     \includegraphics[trim=75 25 60 40, clip, width=\textwidth]%
       {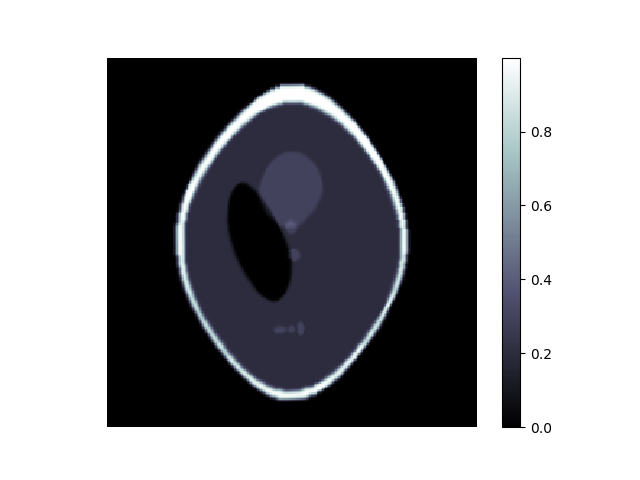}
     \vskip-0.25\baselineskip
     $t=0.25$
   \end{minipage}%
   \hfill
   \begin{minipage}[t]{0.32\textwidth}%
     \centering
     \includegraphics[trim=75 25 60 40, clip, width=\textwidth]%
       {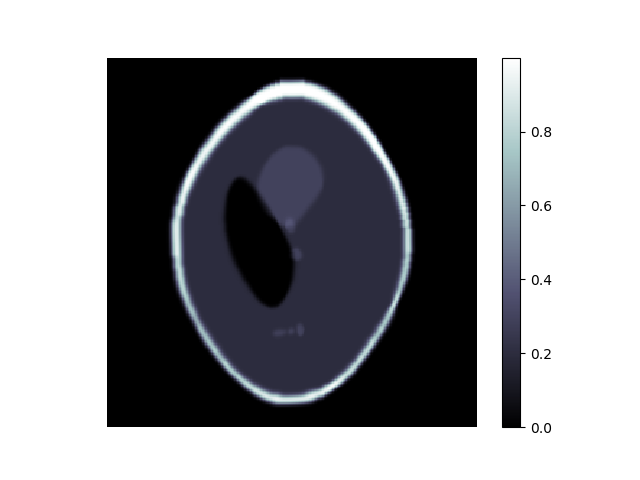}
     \vskip-0.25\baselineskip
     $t=0.50$
   \end{minipage}%
\par\medskip      
% Second row
   \begin{minipage}[t]{0.32\textwidth}%
     \centering
     \includegraphics[trim=75 25 60 40, clip, width=\textwidth]%
       {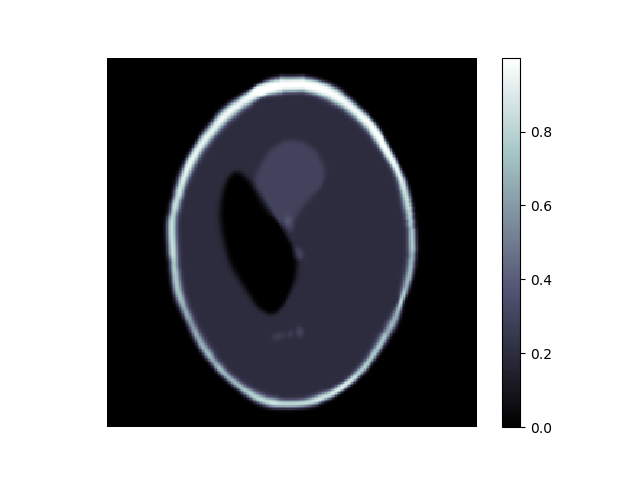}
     \vskip-0.25\baselineskip
     $t=0.75$
   \end{minipage}%
   \hfill
   \begin{minipage}[t]{0.32\textwidth}%
     \centering
     \includegraphics[trim=75 25 60 40, clip, width=\textwidth]%
       {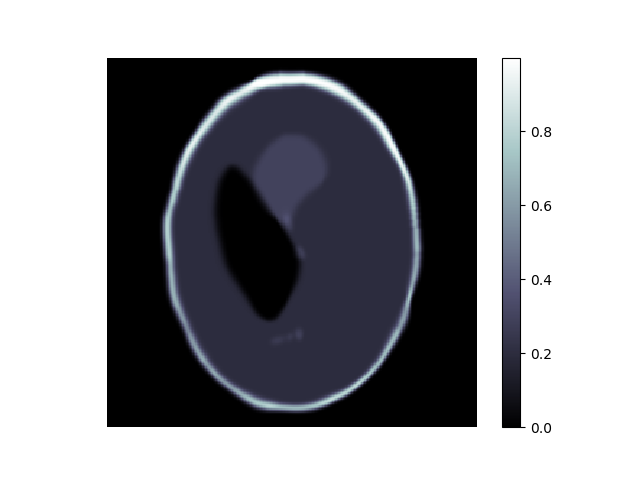}
     \vskip-0.25\baselineskip
     $t=1.0$
   \end{minipage}%
   \hfill
   \begin{minipage}[t]{0.32\textwidth}%
     \centering
     \includegraphics[trim=75 25 60 40, clip, width=\textwidth]%
       {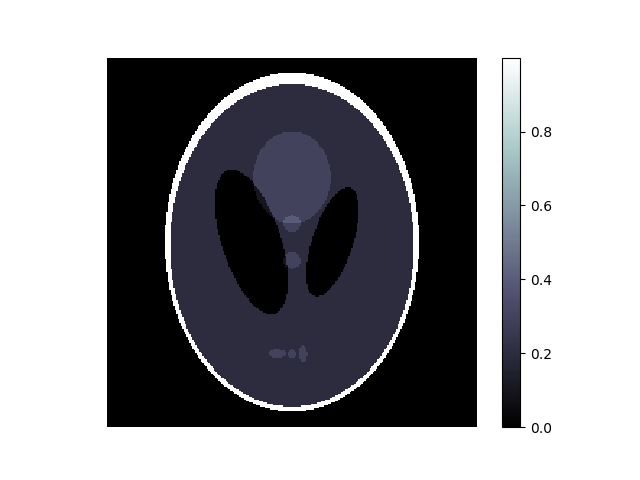}
     \vskip-0.25\baselineskip
     Target
   \end{minipage}%
\\[2em]   
% Third row
   \begin{minipage}[t]{0.32\textwidth}%
     \centering
     \includegraphics[trim=75 25 60 40, clip, width=\textwidth]%
       {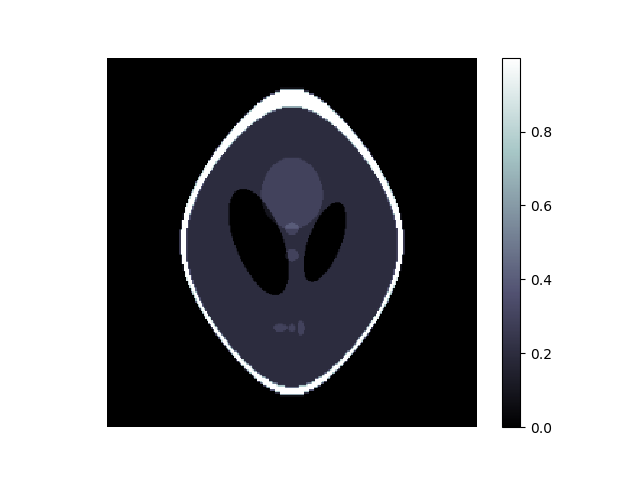}
     \vskip-0.25\baselineskip
     Template
   \end{minipage}%
   \hfill
   \begin{minipage}[t]{0.32\textwidth}%
     \centering
     \includegraphics[trim=75 25 60 40, clip, width=\textwidth]%
       {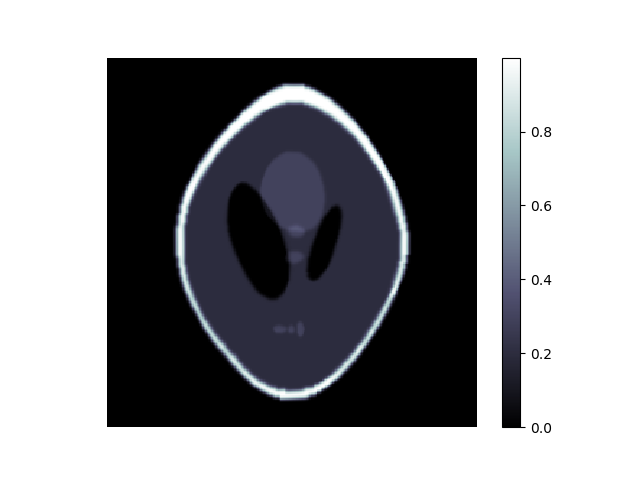}
     \vskip-0.25\baselineskip
     $t=0.25$
   \end{minipage}%
   \hfill
   \begin{minipage}[t]{0.32\textwidth}%
     \centering
     \includegraphics[trim=75 25 60 40, clip, width=\textwidth]%
       {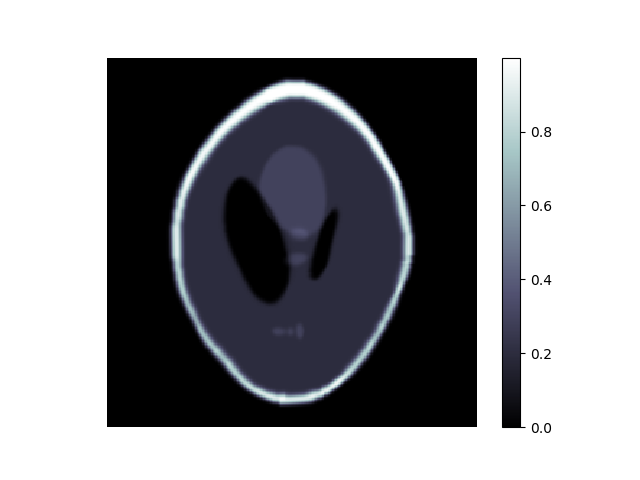}
     \vskip-0.25\baselineskip
     $t=0.50$
   \end{minipage}%
\par\medskip      
% Fourth row
   \begin{minipage}[t]{0.32\textwidth}%
     \centering
     \includegraphics[trim=75 25 60 40, clip, width=\textwidth]%
       {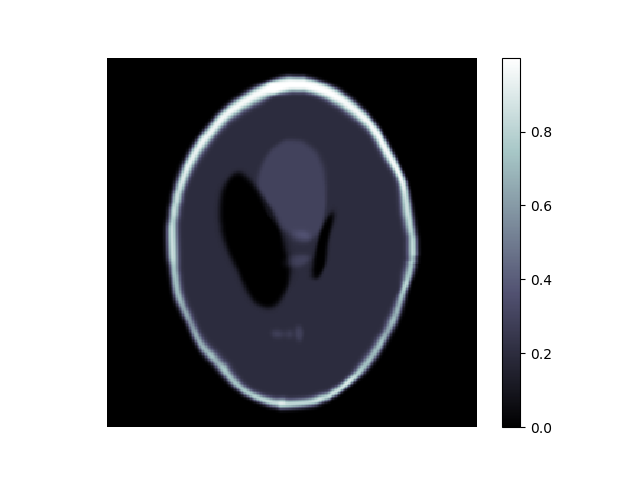}
     \vskip-0.25\baselineskip
     $t=0.75$
   \end{minipage}%
   \hfill
   \begin{minipage}[t]{0.32\textwidth}%
     \centering
     \includegraphics[trim=75 25 60 40, clip, width=\textwidth]%
       {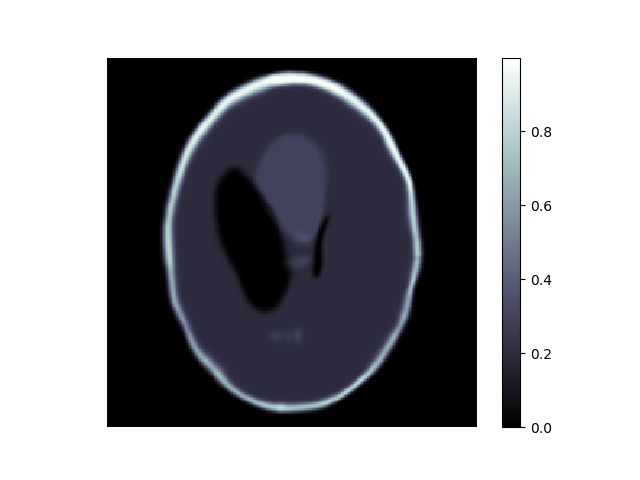}
     \vskip-0.25\baselineskip
     $t=1.0$
   \end{minipage}%
   \hfill
   \begin{minipage}[t]{0.32\textwidth}%
     \centering
     \includegraphics[trim=75 25 60 40, clip, width=\textwidth]%
       {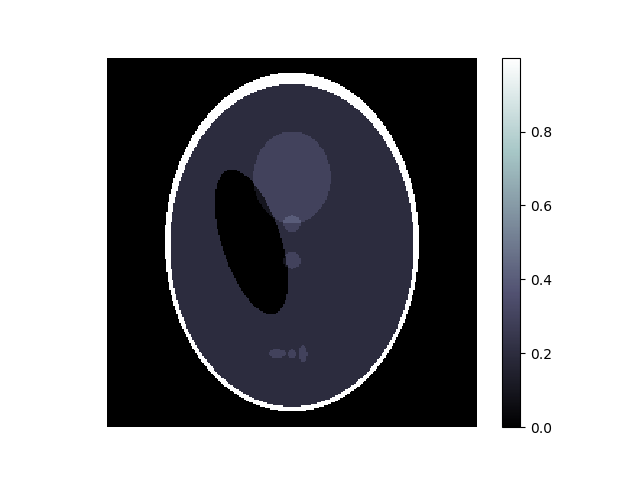}
     \vskip-0.25\baselineskip
     Target
   \end{minipage}%
\caption{Test suite 4: Topology of the template. First two rows is the case when the template lacks one object as compared to target, the two following rows is the case when template has one extra object. Images labeled with time $t$ show the final flow of diffeomorphic deformations obtained from indirect registration with $t=1.0$ denoting the final (indirectly) registered template.}
\label{Test_suite_4:topology}
\end{figure}

\section{Discussion}\label{sec:Discussion}

\subsection{Topology and manifold structure on groups of diffeomorphisms}\label{sec:TopoGDiff}
The \ac{LDDMM} framework can be seen as part of an even more abstract formulation  as outlined in \cite{BrHo15} that involves replacing the group of diffeomorphisms in \cref{eq:DeforSet} with an abstract Lie group $\LieGroup$ of transformations that acts on the shape space $\RecSpace$, and the vector space $\LieAlgebra$ is the corresponding Lie algebra of $\LieGroup$. The situation becomes intricate when one seeks to introduce a useful manifold structure on $\LieGroup$ in the infinite dimensional setting as outlined in \cite{Ba97}. 

As an example, $\DiffG^1(\domain)$ and $\DiffG^{\infty}(\domain)$ are both Fr\'echet--Lie groups, but neither is a Banach manifold. As nicely summarized in \cite{Br16}, working with such deformation groups allows using methods from geometry, but less tools from analysis are applicable. The other choice is to consider a group like $\DiffG^{p}(\domain)$. This group is a Banach manifold and a topological group so there are powerful tools from analysis for prove existence results. On the other hand, it is not a Lie group since the group operations are continuous, but not differentiable. Hence, many of the tools from differential geometry are not available. 

One may next consider groups generated as in \cref{eq:DeforSet} from an admissible Hilbert space of vector fields. The class of admissible vector spaces is however very large, which in turn limits how far one can develop a theory for image registration in this setting. As nicely outlined in \cite[p.~1538]{BrVi17}, $\DiffeoGroup$ in \cref{eq:DeforSet} does not need to have a differentiable structure in the general setting. Furthermore, it does not have be a topological group under the topology induced by the metric in \cref{eq:DiffMetric1}. In fact, there is no natural way to define a topology on $\DiffeoGroup$ in the general setting that makes it a topological group. Hence, one has studied specific classes of admissible spaces. On the other hand, restricting attention to velocity fields with Sobolev regularity generate classical groups of Sobolev diffeomorphisms. Such a group is a Hilbert manifold as well as a topological group, but it is not necessarily a Lie group \cite[p.~1512 and Theorem~8.3]{BrVi17} but the exponential map (\cref{rem:ExpMap}) is continuous \cite[Theorem~4.4]{BrVi17}.  

\subsection{Alternative frameworks for large deformations}\label{rem:AltLDDMM}
\Ac{LDDMM} is capable of generating a vast range of diffeomorphic deformations, but it comes with some drawbacks coupled to its \emph{non-parametric} nature. There is currently no natural parametrisation of deformations using variables that are easy to interpret from an image registration point-of-view. Furthermore, representing these non-parametric diffeomorphic deformations in software in a computationally feasible way is challenging and requires introducing further structure.  This need for further structure on the group of diffeomorphic deformations has resulted in a number of parametric frameworks, some of which are mentioned below. See the excellent introduction in \cite[p.~30-35]{Gr16} for a more informative survey.  

The poly-affine framework \cite{ArPeAy05,ArCoAyPe09} considers deformations obtained by integrating trajectories of vector fields that include local affine deformations. The weights for the each local affine transformation and corresponding regions for its action remain fixed during the integration of the flow. The approach is successfully used for registration of bones \cite{SePeRe12} and  images with cardiac motion \cite{McSeBePe15}, but the approach is not suitable for very large deformations. 
In the GRID model \cite{GrSrSa07,PoVr11} the idea is to fuse local deformations by fusing corresponding velocity fields. The local deformations can be more generic than in the poly-affine case, but each small deformation is located around a point (seed) that defines a local region for its action. The global deformation is a discrete temporal integration of a trajectory of corresponding vector fields. The framework has been used to model biological growth \cite{GrSrSa07}.
The diffeons framework \cite{Yo12} considers deformations are defined as the solution of a flow equation that integrates the identity map along a trajectory of vector fields. The difference to \ac{LDDMM} is that one only considers trajectories of vector fields that can be written as a finite linear combination of a finite family of vector fields, called diffeons, which form a finite dimensional subspace. A further development of diffeons is deformation modules \cite{Gr16}. Here, one introduces a new deformation model that can generate diffeomorphisms that can be locally constrained to a certain type of deformations with a natural  interpretation. The diffeomorphisms are as in \ac{LDDMM} defined as final values of a flows of trajectories of vector fields and the constraints corresponds to setting of a family generators of vector fields similar to the diffeons framework. 

In these parametrised frameworks, the trajectory of the vector field is constrained to be a combination of a few generators. Hence, by choosing these generators it should in principle be possible to build trajectories of diffeomorphisms that correspond to specific a priori deformations. It is however very difficult, if not impossible, to define generators that correspond to more complex deformations that typically arise in many imaging applications. Another issue that is important for image registration is to have a metric on the shape space. Another aspect relevant for (indirect) image registration is to let the metric on the group of deformations induce a metric on the shape space in a canonical manner. This is a key feature of \ac{LDDMM} and it also holds for the diffeons framework, whereas this is not possible within the polyaffine and GRID frameworks. 

A final recent development is to construct diffeomorphisms by using techniques from machine learning. Assume one has training data $\{ (\template_i, \signal_i) \}_i \subset \RecSpace \times \RecSpace$ where $\signal_i = \DeforOpG(\template_i,\diffeo_i)$ for some (unknown) deformation $\diffeo_i \in \DiffeoGroup$ and $\DeforOpG$ given by \cref{eq:DeforOp}. One may then use principles from machine learning to learn how to estimate the deformations from training data. The papers \cite{YaKwNi16,YaKwStNi17} implements such as scheme. More precisely, they use machine learning in order to learn how to estimate the optimal diffeomorphic deformations. The idea is to learn the initial momenta from the training data. Once these initial momenta are learned, one can use the shooting method from \ac{LDDMM} theory to generating the smallest diffeomorphic deformation \cite[Section~11.6.4]{Yo10}. The learning is done by using a (deep) convolution neural network that learns an operator  $\Lambda_{\Theta} \colon \RecSpace \times \RecSpace \to \DiffeoGroup$ whose parameters $\Theta$ are obtained by minimising the mean absolute error. The trained parameter yields the initial momenta, so it is now possible to generate elements in $\DiffeoGroup$ by means of the shooting method. The approach extends readily to the indirect registration setting. Training data is now $\{ (\template_i, \data_i) \}_i \subset \RecSpace \times \DataSpace$ where \cref{eq:ExactShapeAss} holds, \ie, $\data_i = \ForwardOp\bigl( \DeforOpG(\template_i,\diffeo_i) \bigr)$ for some (unknown) deformation $\diffeo_i \in \DiffeoGroup$. One can then perform learning using a loss function similar to the direct matching case. This time the learning is done by using a (deep) convolution neural network that learns an operator
$\Lambda_{\Theta} \colon \RecSpace \times \DataSpace \to \DiffeoGroup$ whose parameters $\Theta$ are obtained by minimising a natural distance notion in data space. Such an approach is possible especially when combined with \cite{AdOk17} for reconstruction.

\subsection{Spatiotemporal image reconstruction}
The goal here is to recover time dependent images $\signal_t \in \RecSpace$ from time dependent data $\data_t \in \DataSpace$ where $\data_t = \ForwardOp(\signal_t) + \noisedata_t$ with $\noisedata_t$ denoting the noise term. A spatiotemporal analogue of \cref{eq:ExactShapeAss} is to assume that $\signal_t = \DeforOpG \bigl(\gelement{0,t}{\velocityfield},\template \bigr)$ for some unknown curve $t \mapsto \gelement{0,t}{\velocityfield}$ in $\DiffeoGroup$. This allows one to formulate the following variational approach to the spatiotemporal image reconstruction problem:
\begin{equation}\label{eq:SpatioTemporal}
\inf_{\substack{\velocityfield \in \FlowSpace{2} \\ \template \in \RecSpace}}
\Biggl[ \int_0^1 \biggl(\gamma \int_0^t \bigl\Vert \velocityfield(s,\cdot) \bigr\Vert^2_{\LieAlgebra}\dint s + \MatchingFunctionalX  \circ \DeforOpG \bigl(\gelement{0,t}{\velocityfield},\template \bigr) \biggr) \dint t 
\Biggr].
\end{equation}
In the above, $\template \in \RecSpace$ is the starting image. The curve of diffeomorphism $t \mapsto \gelement{0,t}{\velocityfield}$ is generated by the velocity field $\velocityfield$ as a solution to \cref{eq:FlowEq}, $\DeforOpG \colon \DiffeoGroup \times \RecSpace \to \RecSpace$ is given by the group action, and $\MatchingFunctionalX  \colon \RecSpace \to \Real$ is the spatiotemporal analogue of  \cref{eq:MatchingFunctionalX}, \ie,
\[ 
\MatchingFunctionalX (\signal) :=  \mu \RegFunc(\signal) + \DataDisc\bigl( \ForwardOp(\signal), \data_t \bigr)
    \quad\text{for $\signal \in \RecSpace$.} 
\]

The optimization problem in \cref{eq:SpatioTemporal} is quite challenging and a natural approach is to consider an intertwined scheme where a reconstruction (of the template) step is followed by an indirect registration step, see \cite[section~12.3]{OkChDoRaBa16} for further details. In such case, indirect registration becomes a key element in spatiotemporal image reconstruction.

\subsection{Shape based reconstruction}
It is possible to view indirect registration as a variational reconstruction scheme that makes use of a priori shape information encoded by the template. The tests in \cref{sec:2DCT} indicate that if the template has the correct topology, then indirect registration performs fairly well. Furthermore, it is also robust against choice of template and regularization parameters. 

When using geometric group action, indirect registration is mostly useful as a reconstruction scheme in imaging problems where the aim is to recover the shape, so intensity variations are of little, or no, importance. One example of such an imaging problem in applications is when \ac{ET} is used to image the internal 3D structures at nano-scale of a specimen \cite{BlPeBaBa15}. This is however a rather limited category of inverse problems and many of the medical imaging applications do not fall under this category. One approach to address this is to consider mass-preserving group action instead of the geometric one. Another is to keep the geometric group action, but let it act on the intensity map. This leads to metamorphosis that is briefly described in \cref{sec:Intensity}.

\subsection{Metamorphosis}\label{sec:Intensity}
The \ac{LDDMM} approach with geometric deformations only moves intensities, it does not change them. Metamorphosis extends \ac{LDDMM} by allowing the diffeomorphisms to also act on intensities. This is achieved by coupling the flow equation in \eqref{eq:FlowEq} to a similar flow equation for the intensities. 

More precisely, in metamorphosis one needs to consider time dependent images. Hence, introduce $\RecSpaceFlow$ as the set of time dependent images that are square integrable in time and contained in $\RecSpace$ at all time points. Hence, if $\zeta \in \RecSpaceFlow$ then $\zeta(t, \Cdot) \in \RecSpace$ and 
\[  \Vert \zeta \Vert_{\RecSpaceFlow} :=  \sqrt{\int_0^1  \bigl\Vert \zeta(t,\Cdot) \bigr\Vert_{\RecSpace}^2\dint t} < \infty. \]
The Hilbert space structure of $\RecSpace$ induces a Hilbert space structures on $\RecSpaceFlow$. Furthermore, introduce $\intensityflow \colon [0,1]\times \domain \to \Real$ as the solution to 
\begin{equation}\label{eq:FlowEq:Intensity}
  \begin{cases}
    \dfrac{d}{dt}\intensityflow(t,x) = \zeta\bigl(t,\diffeo(t,x)\bigr) \quad (t,x) \in [0,1]\times \domain,  & \\[0.5em]
    \intensityflow(0,x) = \template(x) \quad x \in \domain. &
  \end{cases}
\end{equation}
In the above, $\diffeo$ is the diffeomorphism that solves \eqref{eq:FlowEq}, 
$\template \in \RecSpace$ is the template, and $\zeta \in \RecSpaceFlow$ is a ``source'' term 
that is contained in the space $\RecSpaceFlow$ of time dependent images introduced above.

Bearing in mind the above, metamorphosis based image registration in \cite{RiYo16} easily extends to the indirect setting, which in turn extends the \ac{LDDMM} based indirect registration in \cref{eq:OptimV}: 
\begin{equation}\label{eq:Metam:ODE}
\begin{cases}
\displaystyle{\inf_{\substack{\velocityfield \in \FlowSpace{2} \\ \zeta \in \RecSpaceFlow}}} 
\Bigl[ \gamma \Vert \velocityfield \Vert_{\FlowSpace{2}}^2 + \tau \Vert \zeta \Vert_{\RecSpaceFlow}^2 
+ \MatchingFunctionalX  \circ \DeforOpX \bigl(\velocityfield,\intensityflow(1,\Cdot) \bigr)
\Bigr] 
& \\[2em]
\dfrac{d}{dt}\intensityflow(t,x) = \zeta\bigl(t, \gelement{0,t}{\velocityfield}(x) \bigr)  
\quad (t,x) \in [0,1]\times\domain, & \\[0.5em]   
\intensityflow(0,x) = \template(x) \quad x\in \domain.  
\end{cases}  
\end{equation}
Here, $\MatchingFunctionalX  \colon \RecSpace \to \Real$ is given by \cref{eq:MatchingFunctionalX} and $\DeforOpX \colon \FlowSpace{2} \times \RecSpace \to \RecSpace$ is given by \cref{eq:DeforOpGeom} (geometric group action). Note that the velocity field $\velocityfield \in \FlowSpace{2}$ in the objective also generates, through \cref{eq:FlowEq}, the diffeomorphism $\gelement{0,t}{\velocityfield} \colon \domain \to \domain$ in the \ac{ODE} constraint.  Furthermore, $\intensityflow(1,\Cdot) \in \RecSpace$ in the objective functional depends on the optimization variable $\zeta\in \RecSpaceFlow$ through the \ac{ODE} constraint. Finally, as in \cref{sec:PDEFormulations} for \ac{LDDMM}, one can re-phrased \eqref{eq:Metam:ODE} as a \ac{PDE} constrained problem:
\begin{equation}\label{eq:DirectMatchingMetamorphosisPDE}
\begin{cases}
\displaystyle{\inf_{\substack{\velocityfield \in \FlowSpace{2} \\ \zeta \in \RecSpaceFlow}}} 
\Big[  
\gamma \Vert \velocityfield \Vert_{\FlowSpace{2}}^2
+ \tau \Vert \zeta \Vert_{\RecSpaceFlow}^2 +
\MatchingFunctionalX \bigl(\signal(1, \Cdot)\bigr)
\Bigr] & \\[1em]
\partial_t \signal(t,x) + \Bigl\langle \grad_x \signal(t,x), \velocityfield(t,x) \Bigr\rangle_{\Real^n}\!\! = \zeta(t,x)
\quad (t,x) \in [0,1]\times \domain, & \\[0.5em] 
\signal(0,x) = \template(x) \quad x\in \domain.
\end{cases}
\end{equation}

\subsection{Relation to optimal transport}\label{sec:OptimalTransport}
Optimal transport is a well developed mathematical theory that defines a family of metrics between probability distributions. These metrics measure the mass of an optimal displacement according to a base cost defined on the space supporting the distributions. Wasserstein distance refers to using $\LpSpace^p$-norm as base costs. 
It is a deep mathematical theory \cite{Vil09} with many applications. Applications to imaging is nicely surveyed in \cite{Pa16}, the focus here is on its usage in diffeomorphic image registration.

Optimal transport metrics are suitable for shape registration since they are sensitive to spatial displacements of the shape. This framework is \cite{ChScPeVi16,LiMiSa16} extended to handle measures of different total mass (unbalanced optimal transport), which makes it applicable to a wider range of image registration problems. Another aspect associated with using optimal transport distances is computability. Without approximations, it is computationally feasible only for 1D images. This situation has however changed during the last few years, much due to using entropic approximation schemes resulting in Sinkhorn-type iterations \cite{Cu13}. Such an approach is developed in \cite{KaRi16} for indirect image registration where the authors used Sinkhorn-type iterations to compute the proximal operator of the transport problem for large problems, such as indirect image registration in tomography. This scheme also applies to unbalanced optimal transport in \cite{FeChViPe17} to quantify the similarity in diffeomorphic registration of imaging data. Finally, \cite{MaRuScSi15} combines metamorphosis with optimal transport for image registration. The numerical method is based on adding a transport penalty to the objective in \cref{eq:DirectMatchingMetamorphosisPDE}, which is the \ac{PDE} constrained formulation of metamorphosis.

\section*{Acknowledgments}
The authors would like to thank Barbara Gris and Oliver Verdier for valuable comments and remarks. We especially thank Barbara Gris for contributions in stating and proving many of the results in \cref{sec:RegularizingProperty}.

\bibliographystyle{siamplain}
\bibliography{shapereferences}
\end{document}